\documentclass[a4paper,reqno]{amsart}
\usepackage{amsmath}
\usepackage{amssymb}
\usepackage{amsthm}
\usepackage{latexsym}
\usepackage{amscd}
\usepackage{xypic}
\xyoption{curve}
\usepackage{ifthen} 
\usepackage{hyperref} 
\usepackage{graphicx}
\usepackage{enumerate}
\usepackage{enumitem}
\usepackage{rotating}
\usepackage{xcolor}
\usepackage{mathtools}
\usepackage{mathrsfs}
\usepackage{todonotes}

\usepackage{float}
\restylefloat{table}

\numberwithin{equation}{section}

\def\cocoa{{\hbox{\rm C\kern-.13em o\kern-.07em C\kern-.13em o\kern-.15em A}}}

\newtheorem{theorem}{Theorem}[section]

\newtheorem{lemma}[theorem]{Lemma}
\newtheorem{proposition}[theorem]{Proposition}
\newtheorem{corollary}[theorem]{Corollary}

\theoremstyle{definition}
\newtheorem{remark}[theorem]{Remark}
\newtheorem{definition}[theorem]{Definition}

\newtheorem{construction}[theorem]{Construction}

\newcommand {\sHom}{\mathcal{H}\kern -0.25ex{\mathit om}}
\newcommand {\sExt}{\mathcal{E}\kern -0.25ex{\mathit xt}}
\newcommand {\sTor}{\mathcal{T}\kern -0.25ex{\mathit or}}

\newcommand {\Hilb}{\mathcal{H}\kern -0.25ex{\mathit ilb\/}}

\newcommand {\bA}{\mathbb{A}}

\def\a#1{{{\bA}^{#1}_{k}}}

\usepackage{tikz}
\usepackage{scalerel}
\usepackage{pict2e}
\usepackage{tkz-euclide}
\usetikzlibrary{calc}
\usetikzlibrary{patterns, arrows.meta}
\usetikzlibrary{shadows}
\usetikzlibrary{external}
\usepackage{tikz-cd}

\usepackage{caption}

\usepackage{subcaption}

\usepackage{pgfplots}
\pgfplotsset{compat=newest}
\usepgfplotslibrary{statistics}
\usepgfplotslibrary{fillbetween}
\usepackage[framemethod=TikZ]{mdframed}

\usepackage{witharrows}

\usepackage{hyperref}
\hypersetup{
    colorlinks=true,
    allcolors=blue}

\catcode`\@=11
\catcode`\@=12

		\title[Equivariant Poisson structures]{Equivariant Poisson structures on moment-angle manifolds and their quotients}
		
		
		\author[V. Aggarwal]{Vansh Aggarwal}
		\author[M. Poddar]{Mainak Poddar}

\begin{document}

\begin{abstract}
We construct several families of torus equivariant Poisson structures on moment-angle manifolds and their quotients such as quasitoric manifolds. We study the symplectic foliation and Poisson cohomology groups associated with these structures.

\end{abstract}

\maketitle

\section{Introduction}

A Poisson structure on a smooth manifold $P$ is a Lie bracket on the ring of smooth functions $C^{\infty}(P)$, 
\begin{equation}\label{poisson bracket}
   \{,\}:C^\infty(P)\times C^\infty(P)\rightarrow C^\infty(P) 
\end{equation}
such that for any $f \in C^{\infty}(P)$, 
\begin{equation}\label{def of X_f} X_f:= \{f, \cdot\}: C^{\infty}(P) \to C^{\infty}(P) \end{equation} is a derivation. The derivation $X_f$ is called the {\it Hamiltonian vector field} of the function $f$.

A Poisson structure on $P$ can also be specified by a bivector field 
$\Pi \in H^0(\wedge^2 TP)$ where 
\begin{equation}\label{bracket and bivector}
    \{f,g\} = \Pi(df, dg) \quad \forall \; f,\, g\, \in C^{\infty}(P) \,.
\end{equation}
The Jacobi condition \eqref{Jacobi id} on the Poisson bracket $\{ \, \}$ is equivalent to the vanishing of the Schouten bracket (see Section \ref{Schouten}) on the bivector field $\Pi$. In other words, a Poisson structure on $P$ is a bivector field $\Pi$ such that  
\begin{equation}\label{Jacobi on Pi}
[\Pi,\Pi] =0 \,.    
\end{equation} 

The bivector field $\Pi$ is naturally identified with a vector bundle morphism \begin{equation}\label{Poisson str as map}
   \Pi^{\sharp}: T^{*}P \to TP \,. 
\end{equation}
 The Poisson structure $\Pi$ is a {\it symplectic structure} if the morphism \eqref{Poisson str as map} is an isomorphism. Under this isomorphism,  the Poisson bivector field transforms into a closed nondegenerate $2$-form $\Omega$, called the symplectic form.  By the foundational work of Weinstein \cite{Wei}, every Poisson manifold admits a possibly singular foliation whose leaves are symplectic manifolds. 

 Continuous (families of) symmetries of a symplectic or Poisson manifold are of interest from both mathematical and physical viewpoints. In the physical context, Noether's theorem relates symmetries to conserved quantities. In Mathematics, symmetries may be used to obtain new symplectic or Poisson manifolds through appropriate quotient constructions or reductions, the latter terminology is motivated by dimension reduction in Physics.

 The action of a Lie group $G$ on a symplectic manifold $P$ is said to be Hamiltonian if it preserves the symplectic form, and the fundamental vector fields are Hamiltonian in a cohesive fashion
 that can be described in terms of a moment map $\mu: P \to Lie(G)^*$, see \cite[Section 22]{Ana} for a precise definition. 

An extremely well-studied class of symplectic manifolds is that of {\it symplectic toric manifolds}. A symplectic toric manifold is a compact connected symplectic manifold of dimension $2n$, equipped with a Hamiltonian action of an $n$-dimensional torus $\mathbb{T}^n$. The image of the moment map in this case is a polytope in $\mathbb{R}^n$.

A celebrated work of Delzant (\cite{Del}) gives a classification of symplectic toric manifolds in terms of the moment polytopes. The classification requires  the moment polytope to be simple with integral edge vectors at each vertex that form a basis of the character lattice of $\mathbb{T}^n$. Such polytopes are known as Delzant polytopes. 

Delzant used the symplectic reduction technique due to Marsden, Weinstein, and Meyer \cite{Ana}. The standard action of $\mathbb{T}^m$ on 
$\mathbb{C}^m$ equipped with the standard symplectic structure $$\Omega_0 = \frac{\iota}{2} \sum_{i=1}^m dz_i\wedge d\bar{z_i}$$ is Hamiltonian. Given a Delzant polytope $\Delta$ in $\mathbb{R}^n$ with $m$ codimension one faces, $m>n$, one naturally obtains a smooth submanifold $\mathcal{Z}_{\Delta}$ of $\mathbb{C}^m$ of dimension $m+n$, called a moment-angle manifold, and a subtorus $\mathbb{T}^{m-n}$ of $\mathbb{T}^{m}$ that acts freely on $\mathcal{Z}_{\Delta}$ such that the quotient 
$M=\mathcal{Z}_{\Delta}/\mathbb{T}^{m-n}$ inherits a symplectic structure from $\Omega_0$. The induced action of $\mathbb{T}^n = \mathbb{T}^m/ \mathbb{T}^{m-n}$ on $M$ is Hamiltonian with $M/\mathbb{T}^n \cong \Delta$.
This yields a one-to-one correspondence between Delzant polytopes in $\mathbb{R}^n$ and symplectic toric manifolds of dimension $2n$. 

However, there exists a wider class of smooth compact $2n$-dimensional manifolds that admit a smooth locally standard action of the $n$-dimensional torus such that the orbit space is a simple polytope $\Delta$  of dimension $n$. 
These are referred to as either toric or quasitoric manifolds by various authors in the literature \cite{Dav, Buc}. They share various nice topological properties with symplectic toric manifolds. Moreover, they can be realized as the quotient of a moment-angle manifold $\mathcal{Z}_{\Delta} \subset \mathbb{C}^m$ by the action of a suitable subtorus $\mathbb{T}^{m-n}$ of
$\mathbb{T}^m$, where, as earlier, $m$ is the number of codimension one faces of $\Delta$. 

It is known that a quasitoric manifold may not admit a symplectic structure in general. In this article, we construct families of torus invariant Poisson structures on moment-angle manifolds,  quasitoric manifolds and other intermediate quotient spaces. We also investigate some properties of these Poisson structures, such as symplectic foliations and Poisson cohomology. We hope that 
this large class of examples will provide a nice testing ground for new invariants and phenomena in Poisson geometry like the toric manifolds and toric varieties do in symplectic geometry and algebraic geometry, respectively. 

Our main tool is the theory of Poisson reduction developed by  Marsden and Ratiu (\cite{Mar}), which is a generalization of symplectic reduction.
A key point in our construction is that we do not stick to the standard symplectic structure $\Omega_0$ on $\mathbb{C}^m$ but cast our net for a wider class of $\mathbb{T}^m$-invariant Poisson structures on $\mathbb{C}^m$ that admit a reduction. Note that a Poisson structure 
$\Pi$ is said to be invariant under the action of a group $G$ on $P$ if each element of $G$ preserves
$\Pi$. 

Whereas the standard Poisson structure corresponding to $\Omega_0$ is a non-degenerate bivector that pairs radial directions with their corresponding angular directions, in our families, we additionally try to pair angular with other angular directions, and radial with other radial directions. In Section \ref{str on Cm}, we construct 6 such families of  Poisson structures on $\mathbb{C}^m$.  We could analytically determine certain sufficient conditions on the coefficient functions in the bivectors that ensure their integrability, but for necessary and sufficient conditions, we used some computational methods (Python).   
 The newly constructed structures are not reducible in general, and we need additional constraints on the coefficients of the terms of the Poisson bivector. To obtain these, we first study the tangential frames on $\mathcal{Z}_{\Delta}$ and the subtori that act freely on $\mathcal{Z}_{\Delta}$ (see Sections \ref{free torus action} and \ref{tangent spaces of Z}.) 
 We use this knowledge to construct 4 reducible subfamilies of the above Poisson structures
 in Section \ref{Reducible structures}.
 
 Next,  
in Section~\ref{Symplectic leaves} we describe the symplectic foliations of the
reducible structures on $\mathbb{C}^m$ and on $\mathcal{Z}_{\Delta}$. The key
observation is that the Hamiltonian vector fields of a purely angular structure
such as $\Pi_1'$ are tangent to the $\mathbb{T}^m$-orbits
$\mathbb{O}_p=\pi^{-1}(p)$, while those of a purely radial structure such as
${}_K\Pi_3^R$ are combinations of the Euler fields $\mathcal{P}_i$, whose
flows preserve every angular coordinate and the set of vanishing coordinates.
Consequently, the leaves of $\Pi_1'$ are orbits of connected subgroups of
$\mathbb{T}^m$ (in general immersed, possibly dense) inside a single torus
orbit, and the leaves of ${}_K\Pi_3^R$ lie inside a single ``slice''
$\mathbb{S}^I_t$ transverse to the orbits, where, in the radial coordinates
$p_i$, they are open subsets of affine subspaces; under $\pi$ they map into the
relative interiors of the faces of $\Delta$. For the mixed structure
${}_K\Pi_4^R$ the radial and angular Hamiltonian distributions are pointwise
complementary and their flows commute, so each leaf is a product of a radial
leaf and an angular leaf. All of these descriptions restrict to
$\mathcal{Z}_{\Delta}$, since the bivectors are tangent to it.

Sections~\ref{P cohom Cm}--\ref{P cohom of quotients} concern the first Poisson cohomology
of these structures. We first prove a general criterion
(Theorem~\ref{10.9}): if $X_1,\dots,X_n$ are Poisson vector fields and there
are functions $c_1,\dots,c_n$, Casimir on a dense open set, with $X_i(c_j)=0$
for $i\neq j$ and $X_i(c_i)$ nonvanishing almost everywhere, then
$\sum_i f_iX_i$ is Poisson if and only if every $f_i$ is a Casimir. Applying
this with the radial fields $\mathcal{P}_i$ and the Casimirs $p_i$, or with
the angular fields $\Theta_i$ and the angle functions $\theta_i$, we obtain
for each of $\Pi_1$, $\Pi_3$, $\Pi_3^R$, $\Pi_4$, $\Pi_4^R$ explicit families
of Poisson vector fields on $\mathbb{C}^m$ with Casimir coefficients, and we
show that distinct members of each family are not cohomologous, since
Hamiltonian vector fields of an angular (resp.\ radial) structure are angular
(resp.\ radial). In particular, $H^1$ of each of these structures is
infinite-dimensional. On $\mathcal{Z}_{\Delta}$ the role of the
$\mathcal{P}_i$ is played by the vector fields $(P\zeta)^I_u$ built from the
tangent frames of Section~\ref{tangent spaces of Z}, and the same conclusions hold
for the restricted structures. Finally, for the quotient
$\mathcal{O}=\mathcal{Z}_{\Delta}/\mathbb{T}^d_D$ we show that averaging over
$\mathbb{T}^d_D$ followed by pushforward along the quotient map induces a map
$\mathcal{H}_*\colon H^k_{\Pi}(\mathcal{Z}_{\Delta})\to
H^k_{\overline{\Pi}}(\mathcal{O})$, we characterise when two
$\mathbb{T}^d_D$-invariant Poisson vector fields have the same image under
$\mathcal{H}_*$ (Theorem~\ref{12.29}), and we use this to exhibit
infinite families of pairwise distinct classes in $H^1$ of the reduced
structures $\overline{\Pi_1'}$, ${}_K\overline{\Pi_3^R}$ and
${}_K\overline{\Pi_4^R}$ on $\mathcal{O}$, including on quasitoric manifolds.

 
  Some related work has been published in the literature on invariant real Poisson structures in complex toric manifolds \cite{Can1, Can2}.  The invariance is with respect to the action of a complex torus.   There is also related work on torus invariant holomorphic Poisson structures on complex toric varieties  \cite{Hon}.




 \section{Preliminaries on Poisson manifolds}
Let $(P, \{\,\})$ be a Poisson manifold. Denote the set of Hamiltonian vector fields on $P$ by 
$\operatorname{Ham}(\mathfrak{X}^1(P),\Pi)$. Let 
$$\mathfrak{h} : C^{\infty}(P) \to \operatorname{Ham}(\mathfrak{X}^1(P),\Pi)$$
be the map that takes $f $ to $X_f$. 
 The Jacobi identity
\begin{equation}\label{Jacobi id}
   \{ f,\{g,h\}\}+\{h, \{f,g\}\}+\{ g,\{h,f\}\}=0 
\end{equation}
implies that  
\begin{equation}\label{Lie bracket Ham vf}
     X_{\{f,g\}}=[X_{f},X_{g}]  \,.
\end{equation}
Hence, $\operatorname{Ham}(\mathfrak{X}^1(P),\Pi)$ is a subalgebra of the Lie algebra $\mathfrak{X}(P)$ of smooth vector fields on $P$, and  $\mathfrak{h}$ is a Lie algebra homomorphism.

\begin{definition}\label{Poisson vf}   $X \in \mathfrak{X}(P)$ is called a Poisson vector field if the Lie derivative 
$$
L_{X}(\Pi)=0
$$
or, equivalently,
$$
X\{f,g\}=\{Xf,g\}+\{f,Xg\} \quad \forall f,g \in C^\infty(P) \,.
$$
\end{definition}

 The set of Hamiltonian vector fields $(\operatorname{Ham}(\mathfrak{X}^1(P),\Pi))$ is a Lie subalgebra of Poisson vector fields $\operatorname{Pois}(\mathfrak{X}^1(P),\Pi)$. 



\begin{definition}\label{Casimir function}
A smooth function $g \in C^\infty(P)$  is called a Casimir function if its Poisson bracket with every other smooth function on $P$ vanishes,
\[
    \{g, f\} = \Pi(dg, df) = 0 \quad \text{for all } f \in C^\infty(P), 
\]
or equivalently, if $X_g=0$.
\end{definition}

 The set of Casimir functions, denoted by $\operatorname{Cas}(P,\Pi)$, is a Poisson subalgebra of the algebra of $C^\infty(P)$.


\subsection{ Symplectic Foliation}\label{subsec: symp foliation}

\begin{theorem}[Weinstein splitting theorem,\cite{Cran}]  Let $(P,\Pi)$ be a Poisson manifold. For each point $p \in P$ there exists an open chart $(U_{p},x_{1},\ldots,x_{n},y_{1},\ldots,y_{n},$ 
$ q_{1},\ldots, q_{r})$
 such that $q(p):=(q_1(p),\dots,q_r(p))=0$ and
$$
\Pi|_{U_{p}}=\sum_{i=1}^{n}\frac{\partial}{\partial x_{i}} \wedge\frac{\partial}{\partial y_{i}} +\sum_{1\leq j<k\leq r}f_{jk}(q) \frac{\partial}{\partial q_{k}} \wedge \frac{\partial}{\partial q_{j}}
$$
where $f_{jk}$ are smooth functions of $q_{i}$'s satisfying $f_{jk}(0)=0$.
\end{theorem}


Let $S_{p} =\{s \in U_{p}|q(s)=0\}$.
Then, $ S_p$ is a regular submanifold of $P$ such that the Poisson bivector restricted to $S_p$ is non-degenerate and induces a symplectic form on it. In fact, $S_p$ is an open subset of the leaf through $p$ of a possibly singular symplectic foliation of $P$. Note that the tangent bundle of $S_p$ is generated by Hamiltonian vector fields. Since $\operatorname{Ham}(\mathfrak{X}^1(P),\Pi)$  is a Lie subalgebra of  $(\mathfrak{X}(P))$, it induces an integrable (singular) foliation on $P$ by the Stefan-Sussmann integrability theorem. 
The leaves of the foliation have the following characterisation. Define an equivalence relation on $P$ by declaring $x \sim y$ if and only if there exists a piece-wise smooth path between them such that tangent vectors to the path are made up of Hamiltonian directions:
$$
x\sim y \iff \exists \, g_{1},\dots,g_{k}\in C^\infty(P)\quad \text{such that} \quad \phi_{X_{g_{k}}}\circ\dots\circ \phi_{X_{g_{1}}}(x)=y \,,
$$
where $\phi_{X_{g_{i}}}$ is the flow produced by the Hamiltonian vector field $X_{g_{i}}$. The equivalence classes, often called orbits, are the desired leaves.




\subsection{Schouten Bracket}\label{Schouten}

Some concepts in Poisson geometry can be studied using multivector fields and a Lie bracket defined on the algebra of multivector fields. This bracket is called the Schouten bracket, which, when restricted to the Lie algebra of vector fields, is the same as the Lie bracket on $\mathfrak{X}(P)$.


A degree $k$ multivector field   $X^k\in \mathfrak{X}^k(P)=\Gamma(\wedge^k(TP))$ can be regarded as a $C^\infty(P)$ multilinear alternating map on $\Omega^1(P)$, the space of one forms,
$$
X^k:\Omega^1(P)\times\dots \times \Omega^1(P)\rightarrow C^\infty(P)\,.
$$
It can also be thought of as an alternating multi-derivation on $C^\infty(P)$:
$$
L_{X^k}:C^\infty(P)\times\dots \times C^\infty(P) \rightarrow C^\infty(P)
$$
where
$$
L_{X^k}(f_{1},\ldots, f_{k})=X^k(df_{1},\ldots, df_{k}) \quad 
$$
and
$$
L_{X^k}(f_{1},\dots,f_{r}h_{r},\dots,f_{k})=f_{r}L_{X^k}(f_{1},\dots,h_{r},\dots,f_{k})+h_{r}L_{X^k}(f_{1},\dots,f_{r},\dots,f_{k})
$$

\begin{definition}   Given $\alpha \in \mathfrak{X}^{k+1}(P)$ and $\beta \in \mathfrak{X}^{l+1}(P)$, the Schouten bracket  $[\alpha,\beta]\in \mathfrak{X}^{k+l+1}(P)$ is a $(k+l+1)$ order multivector field such that 
$$
L_{[\alpha,\beta]}=L_{\alpha}\circ L_{\beta}-(-1)^{kl}L_{\beta}\circ L_{\alpha}
$$
where 
$$L_{\alpha}\circ L_{\beta}(f_{1},\dots,f_{k+l+1})=\sum_{\sigma \in S_{k,l+1}}(-1)^{\sigma}L_{\alpha}(f_{\sigma(1)},\dots,f_{\sigma(k)},L_{\beta}(f_{\sigma(k+1)},\dots,f_{\sigma(k+l+1)}))$$
\end{definition}
The above definition is an extension of the definition of the Lie bracket of vector fields,
$$
L_{[X,Y]}=L_{X}L_{Y}-L_{Y}L_{X} \,.
$$
such that it satisfies the graded Leibniz identity, i.e. if $\alpha \in \mathfrak{X}^{k+1}(P)$, $\beta \in \mathfrak{X}^{l+1}(P)$ and $\gamma \in \mathfrak{X}^{m+1}(P)$, then
$$
[\alpha,\beta\wedge\gamma]=[\alpha,\beta]\wedge\gamma +(-1)^{k(l+1)}\beta\wedge[\alpha,\gamma]
$$
Using the Schouten bracket, we can rephrase the condition for a vector field $X$ to be Poisson as 
\begin{equation}\label{Poisson vf Schouten} [X,\Pi]=0
\end{equation}
Moreover, the Jacobiator  $J : C^\infty(P)\times C^\infty(P)\times C^\infty(P) \rightarrow C^\infty(P)$,
$$
J(f,g,h) := \{f,\{g,h\}\}+\{h,\{f,g\}\}+\{g,\{h,f\}\}\,,
$$
satisfies 
\begin{equation} {2}J=[\Pi,\Pi]. 
\end{equation}

\subsection{Poisson Cohomology}


The Poisson bivector $\Pi$ induces a differential on the graded algebra of multivector fields $\mathfrak{X}^\bullet(P) = \bigoplus_{k} \Gamma(\bigwedge^k TP)$. This differential $d_\Pi \colon \mathfrak{X}^k(P) \to \mathfrak{X}^{k+1}(P)$, called the Poisson coboundary operator,  is defined by taking the Schouten bracket with $\Pi$,
\begin{equation}
    d_\Pi(\alpha) = [\Pi, \alpha]
\end{equation}
for any $\alpha \in \mathfrak{X}^k(P)$.

By the graded Jacobi identity of the Schouten bracket and the integrability condition $[\Pi, \Pi] = 0$, 
\begin{equation}
    d_\Pi(d_\Pi(X)) = [\Pi, [\Pi, X]] = \frac{1}{2} [[\Pi, \Pi], X] = 0.
\end{equation}
Thus, $d_\Pi^2 = 0$, making $(\mathfrak{X}^\bullet(P), d_\Pi)$ a cochain complex, known as the Lichnerowicz-Poisson complex.

The cohomology of this complex is the \emph{Poisson cohomology} of $(P, \Pi)$, denoted by $H^\bullet_\Pi(P)$. The $k$-th Poisson cohomology group is given by
\begin{equation}
    H^k_\Pi(P) = \frac{\operatorname{Ker}(d_\Pi \colon \mathfrak{X}^k(P) \to \mathfrak{X}^{k+1}(P))}{\operatorname{Im}(d_\Pi \colon \mathfrak{X}^{k-1}(P) \to \mathfrak{X}^k(P))}.
\end{equation}

The lower degree Poisson cohomology groups have important geometric interpretations:
\begin{itemize}
    \item $H^0_\Pi(P)$ consists of Casimir functions.
    \item $H^1_\Pi(P)$ is the quotient of the space of Poisson vector fields by the space of Hamiltonian vector fields, measuring the nontrivial symmetries of the Poisson structure.
\end{itemize}

\subsection{Reduction of Poisson structure}\label{subsec: reduction}

Let $(P,\{\, \}_{P})$ be a Poisson manifold, $M \subset P$ a submanifold and $i: M \rightarrow P$ the inclusion. Let $G$ be a Lie group acting on $P$ and let $\mathfrak{g}$ be the Lie algebra of $G$. Each element $X\in \mathfrak{g}$ generates a fundamental vector field $X^\#$ on $P$.  One can think of
$$
E_p=\{X^\#_p|X\in\mathfrak{g}\} \subseteq T_pP
$$
 as the subspace of tangent vectors at $p$ that are tangent to the orbits produced by the action of $G$ on $P$.
 
Let $E$ be the distribution given by the union of $E_p$'s. Consider the following conditions on this distribution.
\begin{itemize}
\item (A1) $E \cap TM$ is an integrable distribution on $TM$, and thus defines a foliation $\Phi$ on M.
\item (A2) The space of orbits $O=M/\Phi$ is a manifold with the projection map $\mathscr{Q}: M\rightarrow O=M/\Phi$ a submersion, in other words, the foliation $\Phi$ is regular.
\item (A3) For $f,h\in C^\infty(P)$, if $df$ and $dh$ vanish on $E$, then $d\{f,h\}_P$ also vanishes on $E$.
\item (A4) ${\displaystyle \Pi^{\sharp}_p(E^0_p)\subset T_pM + E_p \quad \forall p\in M}$
where $E^0_{p}=\{\alpha_{p}\in T^*_{p}P|\alpha_{p}(E_{p})=0\}$, and $\Pi^{\sharp}_p : T^*_pP\rightarrow T_pP$ is the linear map induced by the bundle map $\Pi^{\sharp}$.
\end{itemize}
\begin{theorem}[\cite{Mar}]\label{2.5}  Assuming $(A1)-(A3)$, $O=M/\Phi$ inherits a unique Poisson structure induced by $\{\,\}_P$ if and only if (A4) holds.
\end{theorem}
The induced Poisson bracket $\{\,\}_{O}$ on $O$  is defined as follows. For $f_j\in C^\infty(O)$, $j=1,2$, we consider a smooth extension $F_j$ of $ f_j\circ\mathscr{Q}=\mathscr{Q}^*(f_j)$  with differential $dF_j$ vanishing along $E$. In other words, the extensions $F_1$ and $F_2$ are $G$-invariant. Then $\{f_1,f_2\}_{O}$ is the smooth function defined by
$$
\mathscr{Q}^*\{f_1,f_2\}_{O}=i^*\{F_1,F_2\}_{P}\, .
$$

Condition (A4) makes  $\{\,\}_{O}$ independent of the choice of smooth extensions.

We need to examine the constraints on the group action imposed by the assumptions (A1)-(A3).
Since $E$ is produced by a Lie group action, $E$ is integrable. Since $M$ is a submanifold, $TM$ is also integrable. Therefore, $E\cap TM$ is integrable and (A1) holds for any group action. The foliation produced by an integrable distribution is regular if the action is free and proper. So, (A2) holds if $G$ is compact and its action is free.

Let $f\in C^{\infty}(P)$ be a smooth function such that $df$ vanishes on $E$. 
This implies that $df(X^\#)=X^{\sharp}(f)=0\quad\forall X\in\mathfrak{g}$.  Since $E_p$  is generated by $\{X^\#_p : X \in \mathfrak{g} \} $, 
$$
df_p(v)=0\quad \forall \, v\in E_p, \, p\in P \iff df(X^\#)=X^\#(f)=0\quad\forall X\in\mathfrak{g} \,.
$$
Thus, (A3) can be reformulated as:  
If $f,h\in C^\infty(P)$ and $X^\#(f)=X^\#(h)=0$ for all $X\in\mathfrak{g}$ then $X^\#(\{f,h\}_P)=0$ for all $X\in \mathfrak{g}$. 

Let $\Pi$ be the bivector field of the Poisson structure $\{\,\}_P$. Note that $$(L_{X^\#}\Pi)(f, h)=[X^\#,\Pi](f,h)=X^\#(\Pi(f,h))-\Pi(X^\# f,h)-\Pi(f,X^\#h)$$ for all $f,h\in C^{\infty}(P)$.
Thus, (A3) holds if
\begin{equation}\label{action derivative of Pi}
L_{X^\#}\Pi=0\quad \forall X\in \mathfrak{g} \,,
\end{equation}
that is,  if $X^\#$ is a Poisson vector field on P for each $X\in\mathfrak{g}$.

Let $\Psi_g$ be the diffeomorphism of $P$ produced by the action of $g \in G$.  Suppose that for each $g\in G$, $\Psi_g$ is a Poisson diffeomorphism, that is,
\begin{equation}\label{invariance of Pi}
d\Psi_g(\Pi)=\Pi \,.
\end{equation}
Then,  for each $X\in \mathfrak{g}$, $\operatorname{exp}(tX)$ is a Poisson diffeomorphism for all $t$ for which the exponential map is defined. It follows that 
$L_{X^{\sharp}} \Pi = 0$ for each $X\in \mathfrak{g}$. Thus we have the following lemma which will be used extensively in the sequel.

\begin{lemma}\label{2.6} 
If $\Psi_g$ is a Poisson diffeomorphism for all $g\in G$, then (A3) holds.
\end{lemma}

We recall that if a Poisson bivector $\Pi$ satisfies \eqref{invariance of Pi}, then the bivector is said to be invariant under the group action.

\begin{lemma}\label{lem: Poisson on submanifold}
 If $(P,\{\}_{P})$ is a Poisson manifold and $M\subset P$ is a submanifold of $P$ such that,
 $$
	\Pi^{\#}\big(T_{p}^*(P)\big)\subset T_{p}(M)\iff \Pi_{p}\in \wedge^2 T_{p}(M) \quad \forall p\in M
 $$
 then $M$ inherits a unique Poisson structure induced by $\{\}_{P}$. 
\end{lemma}
\begin{proof}
	Let $E_{\emptyset}$ be a distribution generated by zero tangent vectors.
	$$
		(E_{\emptyset})_{p}=\{0\in T_{p}(P)\}\quad \forall p\in P
	$$
	Since this distribution together with $M$ satisfy conditions $(A1)-(A4)$, by Theorem \ref{2.5}, the statement holds true. 
\end{proof}
If $\{\}_{M}$ denote the induced Poisson structure, then for any $f,h\in C^{\infty}(M)$
$$
	\{f,h\}_{M}=i^*\Big(\{F,H\}_{P}\Big)
$$ where $F,H\in C^{\infty}(P)$ are smooth extensions of $f$ and $h$ respectively. In other words $f=F|_{M}=i^*(F)$ and $h=H|_{M}=i^*(H)$.\\ 
\begin{remark} We refer the condition:
$$
	\Pi^{\sharp}\big(T_{p}^*(P)\big)\subset T_{p}(M)\iff \Pi_{p}\in T_{p}(M)\wedge T_{p}(M) \quad \forall p\in M
$$
as $(A4)'$. For any distribution $E$, $(A 4)'\implies (A 4)$
\end{remark}
\begin{lemma}\label{2.8}
	Assuming $(A1)-(A3)$, if $(A 4)'$ holds, then $O=M/\Phi$ inherits a unique Poisson structure induced from $\{\}_{M}$ such that if $\overline{f},\overline{h}\in C^{\infty}(O)$,
 $$
	\mathscr{Q}^*\Big(\{\overline{f},\overline{h}\}_{O}\Big)=\{f,h\}_{M}
 $$
 where $f=\mathscr{Q}^*(\overline{f})$ and $h=\mathscr{Q}^*(\overline{h})$
\end{lemma}
\begin{proof}
	By Theorem \ref{2.5}, $O$ inherits a Poisson structure $\{\}_{O}$, such that for any $\overline{f},\overline{h}\in C^{\infty}(O)$
	$$
		\mathscr{Q}^*\Big(\{\overline{f},\overline{h}\}_{O}\Big)=i^*\Big(\{F,H\}_{P}\Big)
	$$
	where $F$ and $H$ are $G$-invariant extension of $\mathscr{Q}^*(f)$ and $\mathscr{Q}^*(h)$ respectively.
	By Lemma \ref{lem: Poisson on submanifold}, $M$ inherits a Poisson structure $\{\}_{M}$, such that
	$$
			\{f,h\}_{M}=i^*\Big(\{F,H\}_{P}\Big)
	$$
	$\implies \mathscr{Q}^*\Big(\{\overline{f},\overline{h}\}_{O}\Big)=\{f,h\}_{M}$
	
\end{proof}

\section{Poisson Structures on $\mathbb{C}^m$}\label{str on Cm}

Let $z_i=x_i+ \iota y_i$, where $\iota= \sqrt{-1}$ and $i=1, \ldots, m$, be the standard  complex coordinates on $\mathbb{C}^m$. Let $ r_i= \sqrt{x_i^2 + y_i^2}$ and $p_i= r_i^2$. Let $\theta_i$ be the angular coordinate defined by $x_i = r_i\cos \theta_i $ and $ y_i = r_i \sin \theta_i$.


The standard Poisson structure on $\mathbb{C}^m$ is given by the bivector

\begin{equation}
\Pi_{st}=\sum_{q=1}^m\frac{\partial}{\partial x_{q}}\wedge \frac{\partial}{\partial y_{q}} =-{2\iota} \sum_{q=1}^m\frac{\partial}{\partial z_{q}}\wedge \frac{\partial}{\partial \overline{z}_{q}}
\end{equation}

Define 
$$\mathbb{C}_1^m := \{ (z_1, \ldots, z_m) \in \mathbb{C}^m \, \mid \,  z_i \neq 0 \, \forall i \}$$

$$
\mathcal{P}_{i}=\frac{1}{2}\left( x_{i}\frac{\partial}{\partial x_{i}}+y_{i}\frac{\partial}{\partial y_{i}} \right)\quad \mathcal{P}_{i}|_{\mathbb{C}_1^m} =p_{i}\frac{\partial}{\partial p_{i}}\quad 1\leq i\leq m
$$
$$
\
\Theta_{j}=\left(- y_{j}\frac{\partial}{\partial x_{j}}+x_{j}\frac{\partial}{\partial y_{j}}\right)\quad \Theta_{j}|_{\mathbb{C}_1^m}=\frac{\partial}{\partial \theta_{j}}\quad 1\leq j\leq m
$$

Note that 
$$
\Pi_{st}|_{\mathbb{C}^m_{1}}=\sum_{q=1}^m 2\, \frac{\partial}{\partial p_{q}}\wedge \frac{\partial}{\partial \theta_{q}}
$$

We construct several $(S^1)^m$ equivariant families of Poisson structures on $\mathbb{C}^m$. Below is a list of bivector fields and some constraints on them so that they satisfy the vanishing Schouten bracket condition \eqref{Jacobi on Pi}. For simplicity, we often restrict our attention to a subset of the solution space of \eqref{Jacobi on Pi}.

\subsection{Angular Poisson structures}

\begin{equation}
\Pi_{1}=\sum_{1\leq i<j\leq m} a_{ij}\Theta_{i}\wedge \Theta_{j} 
\end{equation}

\begin{equation*}
\Pi_{1}|_{\mathbb{C}^m_{1}}=\sum_{1\leq i<j\leq m} a_{ij}\frac{\partial}{\partial \theta_{i}}\wedge \frac{\partial}{\partial \theta_{j}}
\end{equation*}
where each $a_{ij}\in C^\infty(\mathbb{C}^m)$ and $\Theta_{k}(a_{ij})=0 \quad \forall1\leq k \leq m$.
In other words, each $a_{ij}$ is a function of $p_{1},\dots,p_m$ only.

\subsection{Augmented Angular Poisson structures}
\begin{equation*}
\Pi_{2}=\Pi_{st}+\Pi_{1}
\end{equation*}
\begin{equation}
\Pi_{2}=\sum_{q=1}^m\frac{\partial}{\partial x_{q}}\wedge \frac{\partial}{\partial y_{q}}+\sum_{1\leq i<j\leq m} b_{ij}\Theta_{i}\wedge \Theta_{j}
\end{equation}
\begin{equation*}
\Pi_{2}|_{\mathbb{C}^m_{1}}=\sum_{q=1}^m2\frac{\partial}{\partial p_{q}}\wedge \frac{\partial}{\partial \theta_{q}}+\sum_{1\leq i<j\leq m} b_{ij}\frac{\partial}{\partial \theta_{i}}\wedge \frac{\partial}{\partial \theta_{j}}
\end{equation*}
For simplicity, we assume that $b_{ij}$ is a $C^\infty$ function that depends only on $p_i$ and $p_j$, that is,
$$
b_{ij}=b_{ij}(x_{i}^2+y_{i}^2, x_{j}^2+y_{j}^2)
$$
\vspace{1cm}
\subsection{Radial Poisson structures}
\begin{equation}
\Pi_{3}=\sum_{1\leq i<j\leq m} {c_{ij}}\mathcal{P}_{i}\wedge \mathcal{P}_{j}
\end{equation}
\begin{equation*}
\Pi_{3}|_{\mathbb{C}^m_{1}}=\sum_{1\leq i<j\leq m} c_{ij}p_{i}p_{j}\frac{\partial}{\partial p_{i}}\wedge \frac{\partial}{\partial p_{j}}
\end{equation*}
Again, we  only consider a subset of the solution space of \eqref{Jacobi on Pi},  and  assume that the  $c_{ij}$'s are constant functions.

\subsection{Radial-Angular Poisson structures}

\begin{equation*}
\Pi_{4}=\Pi_{3}+\Pi_{1}
\end{equation*}
\begin{equation}
\Pi_{4}=\sum_{1\leq i<j\leq m} {e_{ij}}\mathcal{P}_{i}\wedge \mathcal{P}_{j}+\sum_{1\leq r<s\leq m} d_{rs}\Theta_{r}\wedge \Theta_{s}
\end{equation}
\begin{equation*}
\Pi_{4}|_{\mathbb{C}^m_{1}}=\sum_{1\leq i<j\leq m} e_{ij}p_{i}p_{j}\frac{\partial}{\partial p_{i}}\wedge \frac{\partial}{\partial p_{j}}+\sum_{1\leq r<s\leq m} d_{rs}\frac{\partial}{\partial \theta_{r}}\wedge \frac{\partial}{\partial \theta_{s}}
\end{equation*}
where $e_{ij}$ and $d_{rs}$ are constant functions.

\vspace{0.5cm}

 We don't know whether the structures of type $\Pi_3$ and $\Pi_4$ are reducible or not in the sense of Section \ref{subsec: reduction} with $P=\mathbb{C}^m$ and $M=\mathcal{Z}_{\Delta}$ (moment-angle manifold associated to $\Delta$) and $G$ a Lie subgroup of $(S^1)^m$ acting in the standard manner on $\mathbb{C}^m$.

We construct two more families  of bivector fields containing radial terms that are reducible under some restrictions.
To begin with, choose at most 3 distinct $i_1,i_2,i_3 \in\{1,\dots,m\}$ and let $R=\{i_1,i_2,i_3\}$. Define bivector fields $\Pi^R_3$ and $\Pi^R_4$ on $\mathbb{C}^m$ as follows.

\subsection{Reducible Radial Poisson structures}
\begin{equation}
\Pi^R_3= \sum_{1\leq k<j\leq 3} {c^R_{i_ji_k}}\mathcal{P}_{i_{j}}\wedge \mathcal{P}_{i_{k}}
\end{equation}
$$
\Pi^R_{3}|_{\mathbb{C}^m_{1}}=\sum_{1\leq k<j\leq 3} c^R_{i_ji_k}p_{i_j}p_{i_k}\frac{\partial}{\partial p_{i_j}}\wedge \frac{\partial}{\partial p_{i_k}}
$$
In this case, a sufficient condition on the coefficient functions to satisfy the condition \eqref{Jacobi on Pi} is
$$
c^R_{i_ji_k}=c^R_{i_ji_k}(x_{l}^2+y_{l}^2)\in C^{\infty}(\mathbb{C}^m)
$$
where $l\in\{1,\dots,m\} \setminus \{ i_j,i_k\}$. In other words, $$\mathcal{P}_{i_{j}}(c^R_{i_ji_k})=\mathcal{P}_{i_{k}}(c^R_{i_ji_k})=0$$ for each pair $(k,j)\in \{(1,2),(2,3),(1,3)\}$.

\vspace{1cm}
\subsection{Reducible Radial-Angular Poisson structures}
\begin{equation*}
\Pi_{4}^R=\Pi^R_{3}+\Pi_{1}
\end{equation*}
\begin{equation}
\Pi^R_{4}=\sum_{1\leq k<j\leq 3} {e^R_{i_ji_k}}\mathcal{P}_{i_{j}}\wedge \mathcal{P}_{i_{k}}+\sum_{1\leq r<s\leq m} d_{rs}^R\Theta_{r}\wedge \Theta_{s}
\end{equation}
\begin{equation*}
\Pi^R_{4}|_{\mathbb{C}^m_{1}}=\sum_{1\leq k<j\leq 3} e^R_{i_ji_k}p_{i_j}p_{i_k}\frac{\partial}{\partial p_{i_j}}\wedge \frac{\partial}{\partial p_{i_k}}+\sum_{1\leq r<s\leq m} d_{rs}^R\frac{\partial}{\partial \theta_{r}}\wedge \frac{\partial}{\partial \theta_{s}}
\end{equation*}
In this case, a sufficient condition on the coefficient functions to satisfy the condition \eqref{Jacobi on Pi} is,
$$
e^R_{i_ji_k}=e^R_{i_ji_k}(x_{l}^2+y_{l}^2)\in C^{\infty}(\mathbb{C}^m)
$$
where $l\in\{1,\dots,m\} \setminus 
\{i_j,i_k\}$; in other words, $$\mathcal{P}_{i_{j}}(e^R_{i_ji_k})=\mathcal{P}_{i_{k}}(e^R_{i_ji_k})=0$$ for each pair $(k,j)\in \{(1,2),(2,3),(1,3)\}$. The $d_{rs}^R$ are constant functions.
\begin{remark}
$|R|=3$ is not a necessary condition under which $\Pi_{3}^R$ and $\Pi_{4}^R$ (with additional constraints mentioned in Section \ref{Reducible structures}) become reducible. One can construct $\Pi_{3}^R$ and $\Pi_{4}^R$ with $|R|>3$ such that under additional constraints, they become reducible in the sense of \ref{subsec: reduction} with $P=\mathbb{C}^m$ and $M=\mathcal{Z}_{\Delta}$ and $G$ a Lie subgroup of $(S^1)^m$ acting in the standard manner on $\mathbb{C}^m$.

\end{remark}
\subsection{Computational verification}
      We have written Python code that compute the necessary and sufficient constraints on the coefficient functions of the bivector fields $\Pi$ listed above so that $[\Pi,\Pi]$ vanishes when restricted to $\mathbb{C}^m_{1}$, for $m=3,4,5$. The functions are expressed in terms of the radial coordinates $p_i$ and the angular coordinates $\theta_j$. This is sufficient since, by continuity, the trivector field $[\Pi,\Pi]$ is then zero on the whole of $\mathbb{C}^m$.
      The code can be generalized in a straightforward manner for other $m$. The output of the code is a list of functions that should vanish. The code is availabe at, \\
   \url{ https://github.com/Lexion-research/Math-project}

     \subsection{Analytic verification}
    We can analytically verify that the sufficiently simplified conditions on the coefficient functions of bivectors mentioned in Sections 3.1-3.6, satisfy \eqref{Jacobi on Pi}. But first, we need to prove a few lemmas and corollaries.

\begin{lemma}\label{3.1}
 The Lie bracket of radial and angular vector fields defined above is zero.
$$
[\mathcal{P}_{i},\mathcal{P}_{k}]=0 \quad [\Theta_{j},\Theta _{l}]=0\quad [\mathcal{P}_{i},\Theta _{j}]=0 \quad \forall 1\leq i,j,k,l\leq m
$$
\end{lemma}

\begin{proof}
    For any function $h\in C^\infty(\mathbb{C}^m)$,
\begin{align*}
\mathcal{P}_{i}\mathcal{P}_{k}h&=\frac{1}{4} x_{i}\frac{\partial}{\partial x_{i}}\left( x_{k}\frac{\partial h}{\partial x_{k}} \right)+\frac{1}{4}x_{i}\frac{\partial}{\partial x_{i}}\left( y_{k}\frac{\partial h}{\partial y_{k}} \right)\\
&+\frac{1}{4}y_{i}\frac{\partial}{\partial y_{i}}\left( x_{k}\frac{\partial h}{\partial x_{k}} \right)+\frac{1}{4}y_{i}\frac{\partial}{\partial y_{i}}\left( y_{k}\frac{\partial h}{\partial y_{k}} \right)\\
&= \frac{1}{4}x_{k}\frac{\partial}{\partial x_{k}}\left( x_{i}\frac{\partial h}{\partial x_{i}} \right)+\frac{1}{4}x_{k}\frac{\partial}{\partial x_{k}}\left( y_{i}\frac{\partial h}{\partial y_{i}} \right)\\
&+\frac{1}{4}y_{k}\frac{\partial}{\partial y_{k}}\left( x_{i}\frac{\partial h}{\partial x_{i}} \right)+\frac{1}{4}y_{k}\frac{\partial}{\partial y_{k}}\left( y_{i}\frac{\partial h}{\partial y_{i}} \right)\\&= 
\frac{1}{4}\left( x_{k}\frac{\partial}{\partial x_{k}}+y_{k}\frac{\partial}{\partial y_{k}} \right)\left( x_{i}\frac{\partial}{\partial x_{i}}+y_{i}\frac{\partial}{\partial y_{i}} \right)(h)\\&=\mathcal{P}_{k}\mathcal{P}_{i}h
\end{align*}
Thus, $[\mathcal{P}_{i},\mathcal{P}_{k}]=\mathcal{P}_{i}\Big(\mathcal{P}_{k}(h)\Big)-\mathcal{P}_{k}\Big(\mathcal{P}_{i}(h)\Big)=0$ for all $1\leq i,k\leq m$.
Similarly, $[\Theta_{j},\Theta _{l}]=0\quad [\mathcal{P}_{i},\Theta _{j}]=0 \quad \forall 1\leq i,j,l\leq m$
\end{proof}
\begin{corollary}\label{3.2}
	$$
		[\mathcal{P}_{i}\wedge \mathcal{P}_{j},\mathcal{P}_{k}\wedge \mathcal{P}_{l}]=0 \quad [\Theta_{i}\wedge\Theta_{j},\Theta_{k}\wedge\Theta _{l}]=0\quad [\mathcal{P}_{i}\wedge \mathcal{P}_{j},\Theta _{k}\wedge \Theta_{l}]=0 \quad \forall 1\leq i,j,k,l\leq m
	$$
\end{corollary}
\begin{proof}
	The result follows directly from Lemma \ref{3.1} and the graded Leibniz property.
\end{proof}
\begin{corollary}\label{3.3}
If $f_{ij},h_{kl}\in C^{\infty}(\mathbb{C}^m)$ such that $\mathcal{P}_{i}(f_{ij})=\mathcal{P}_{j}(f_{ij})=0$ and $\Theta_{k}(h_{kl})=\Theta_{l}(h_{kl})=0$, then
$$
[\alpha \mathcal{P}_{i}\wedge \mathcal{P}_{j},f_{ij}]=0\quad \text{and}\quad [\beta\Theta_{k}\wedge \Theta_{l},h_{kl}]=0
$$
where $\alpha,\beta\in C^{\infty}(\mathbb{C}^m)$

\end{corollary}
\begin{proof}
	The statement follows directly from the graded Leibniz identity.
\end{proof}

We show in the following that $[\Pi_1,\Pi_1]=[\Pi_{3},\Pi_{3}]=[\Pi_{3}^R,\Pi_{3}^R] =0$ under the constraints on the coefficients $a_{ij},c_{ij}$ and $c_{i_{j}i_{k}}^R$ in Section \ref{str on Cm}.

\begin{equation*}
\begin{split}
	[\Pi_1,\Pi_1]&=\sum_{\substack{1\leq i<j\leq m\\ 1\leq k<l\leq m}}\Big[ a_{ij}\wedge \Theta_{i}\wedge \Theta_{j},a_{kl}\wedge \Theta_{k}\wedge \Theta_{l}\Big]\\
	&=\sum_{\substack{1\leq i<j\leq m\\ 1\leq k<l\leq m}}\Big(\Big[a_{ij}\wedge \Theta _{i}\wedge \Theta_{j},a_{kl}\Big]\wedge \Theta_{k}\wedge \Theta_{l}+a_{ij}\wedge a_{kl}\wedge\Big[\Theta_{i}\wedge \Theta_{j},\Theta_{k}\wedge \Theta_{l}\Big]\\
    &+a_{kl}\wedge\Big[a_{ij},\Theta_{k}\wedge \Theta_{l}\Big]\wedge \Theta_{i}\wedge \Theta_{j}\Big)
\end{split}
\end{equation*}
Since $\Theta_{k}(a_{ij})=0$ for all $1\leq i,j,k\leq m$, by Corollary \ref{3.2} and \ref{3.3},
$$
[\Pi_1,\Pi_1]=0
$$

\begin{equation*}
\begin{split}
	[\Pi_3,\Pi_3]&=\sum_{\substack{1\leq i<j\leq m\\ 1\leq k<l\leq m}}\Big[ c_{ij}\wedge\mathcal{P}_{i}\wedge \mathcal{P}_{j},c_{kl}\wedge\mathcal{P}_{k}\wedge \mathcal{P}_{l}\Big]\\
	&=\sum_{\substack{1\leq i<j\leq m\\ 1\leq k<l\leq m}}\Big(\Big[c_{ij}\wedge\mathcal{P}_{i}\wedge \mathcal{P}_{j},c_{kl}\Big]\wedge \mathcal{P}_{k}\wedge \mathcal{P}_{l}
	+c_{ij}\wedge c_{kl}\wedge\Big[\mathcal{P}_{i}\wedge \mathcal{P}_{j},\mathcal{P}_{k}\wedge \mathcal{P}_{l}\Big]\\
&+c_{kl}\wedge\Big[c_{ij},\mathcal{P}_{k}\wedge \mathcal{P}_{l}\Big]\wedge \mathcal{P}_{i}\wedge \mathcal{P}_{j}\Big)
\end{split}
\end{equation*}
Since for each $1\leq i,j\leq m$, $c_{ij}$ is a constant function, then by Corollary \ref{3.2} and \ref{3.3},

	$$
		[\Pi_3,\Pi_3]=0
	$$

$$
[\Pi^R_3,\Pi^R_3]=\Big[\sum_{1\leq k<j\leq 3} c^R_{i_ji_k}\mathcal{P}_{i_{j}}\wedge \mathcal{P}_{i_{k}},\sum_{1\leq k<j\leq 3} c^R_{i_ji_k}\mathcal{P}_{i_{j}}\wedge \mathcal{P}_{i_{k}}\Big]\quad \quad \quad\quad\quad\quad\quad\quad\quad\quad
$$
Using the linearity of the Schouten bracket, any term in the expansion of $[\Pi^R_3,\Pi^R_3]$ takes one of the following two forms,
$$
	\Big[c^R_{i_{j}i_{k}}\mathcal{P}_{i_{j}}\wedge \mathcal{P}_{i_{k}},c^R_{i_{j}i_{k}}\mathcal{P}_{i_{j}}\wedge \mathcal{P}_{i_{k}}\Big]\quad \text{or}\quad \Big[c^R_{i_{j}i_{k}}\mathcal{P}_{i_{j}}\wedge \mathcal{P}_{i_{k}},c^R_{i_{j}i_{l}}\mathcal{P}_{i_{j}}\wedge \mathcal{P}_{i_{l}}\Big]
$$
Without loss of generality, let $i_{1}=1,i_{2}=2$ and $i_{3}=3$ and we show that 
$$
	\Big[c^R_{12}\wedge \mathcal{P}_{1}\wedge\mathcal{P}_{2},c^R_{12}\wedge \mathcal{P}_{1}\wedge\mathcal{P}_{2}\Big]=0\quad \text{and} \quad \Big[c^R_{12}\wedge \mathcal{P}_{1}\wedge\mathcal{P}_{2},c^R_{13}\wedge \mathcal{P}_{1}\wedge\mathcal{P}_{3}\Big]=0
$$
\begin{equation*}
\begin{split}
	\Big[c^R_{12}\wedge \mathcal{P}_{1}\wedge\mathcal{P}_{2},c^R_{12}\wedge \mathcal{P}_{1}\wedge\mathcal{P}_{2}\Big]=&\Big[c^R_{12}\wedge \mathcal{P}_{1}\wedge\mathcal{P}_{2},c_{12}^R\Big]\wedge \mathcal{P}_{1}\wedge \mathcal{P}_{2}\\
    &+c_{12}^R\wedge\Big[c_{12}^R,\mathcal{P}_{1}\wedge\mathcal{P}_{2}\Big]\wedge \mathcal{P}_{1}\wedge \mathcal{P}_{2}\\
    &+ c_{12}^R\wedge c_{12}^R\wedge\Big[\mathcal{P}_{1}\wedge \mathcal{P}_{2},\mathcal{P}_{1}\wedge \mathcal{P}_{2}\Big]
\end{split}
\end{equation*}
Since $\mathcal{P}_{1}(c_{12}^R)=\mathcal{P}_{2}(c_{12}^R)=0$, by Corollary \ref{3.3},

$$
	\implies \Big[c_{12}^R\wedge\mathcal{P}_{1}\wedge \mathcal{P}_{2},c_{12}^R\Big]=\Big[\mathcal{P}_{1}\wedge \mathcal{P}_{2},c_{12}^R\Big]=0
$$
So by Corollary \ref{3.2},
$$
	\Big[c^R_{12}\wedge \mathcal{P}_{1}\wedge\mathcal{P}_{2},c^R_{12}\wedge \mathcal{P}_{1}\wedge\mathcal{P}_{2}\Big]=c_{12}^R\wedge c_{12}^R\wedge\Big[\mathcal{P}_{1}\wedge \mathcal{P}_{2},\mathcal{P}_{1}\wedge \mathcal{P}_{2}\Big]=0
$$
\begin{equation*}
\begin{split}
	\Big[c^R_{12}\wedge \mathcal{P}_{1}\wedge\mathcal{P}_{2},c^R_{13}\wedge \mathcal{P}_{1}\wedge\mathcal{P}_{3}\Big]
	=&\Big[c^R_{12}\wedge \mathcal{P}_{1}\wedge\mathcal{P}_{2},c_{13}^R\Big]\wedge \mathcal{P}_{1}\wedge \mathcal{P}_{3}
	\\
    &+c_{13}^R\wedge\Big[c_{12}^R,\mathcal{P}_{1}\wedge\mathcal{P}_{3}\Big]\wedge \mathcal{P}_{1}\wedge \mathcal{P}_{2}\\
	&+ c_{12}^R\wedge c_{13}^R\wedge\Big[\mathcal{P}_{1}\wedge \mathcal{P}_{2},\mathcal{P}_{1}\wedge \mathcal{P}_{3}\Big]
\end{split}
\end{equation*}
Since $\mathcal{P}_{1}(c_{13}^R)=0$
$$
	\implies \Big[c^R_{12}\wedge \mathcal{P}_{1}\wedge\mathcal{P}_{2},c_{13}^R\Big]=c_{12}^R\wedge \mathcal{P}_{1}(c_{13}^R)\wedge \mathcal{P}_{2}-c_{12}^R\wedge \mathcal{P}_{2}(c_{13}^R)\wedge \mathcal{P}_{1}=-c_{12}^R\wedge \mathcal{P}_{2}(
	c_{13}^R)\wedge \mathcal{P}_{1}
$$
Thus,
$$
\Big[c^R_{12}\wedge \mathcal{P}_{1}\wedge\mathcal{P}_{2},c_{13}^R\Big]\wedge \mathcal{P}_{1}\wedge \mathcal{P}_{3}=-c_{12}^R\wedge \mathcal{P}_{2}(
	c_{13}^R)\wedge \mathcal{P}_{1}\wedge \mathcal{P}_{1}\wedge \mathcal{P}_{3}=0
$$
Similarly,
$$
	c_{13}^R\wedge\Big[c_{12}^R,\mathcal{P}_{1}\wedge\mathcal{P}_{3}\Big]\wedge \mathcal{P}_{1}\wedge \mathcal{P}_{2}=0
$$
Thus, by Corollary \ref{3.2},
$$
	\Big[c^R_{12}\wedge \mathcal{P}_{1}\wedge\mathcal{P}_{2},c^R_{13}\wedge \mathcal{P}_{1}\wedge\mathcal{P}_{3}\Big]=c_{12}^R\wedge c_{13}^R\wedge\Big[\mathcal{P}_{1}\wedge \mathcal{P}_{2},\mathcal{P}_{1}\wedge \mathcal{P}_{3}\Big]=0
$$
As a result
$$
	[\Pi^R_3,\Pi^R_3]=0
$$

Similarly one can show that under the constraints on the coefficient functions, mentioned in Section \ref{str on Cm},
    $$ [\Pi_{2},\Pi_{2}]=[\Pi_{4},\Pi_{4}]=[\Pi^R_{4},\Pi^R_{4}]=0 $$

\section{Convex Polytopes}

\begin{definition}
A convex polytope of dimension $n$ is called simple if every vertex is the intersection of exactly $n$ facets (that is, codimension one faces). 
\end{definition}

Let $\Delta \subset \mathbb{R}^n$ be an $n$-dimensional simple polytope with $m$ facets, $F_1, \ldots, F_m$. We present $\Delta$ as the intersection of $m$ closed half-spaces in $\mathbb{R}^n$:
\begin{equation}\label{polytope in Rn}
    \Delta = \{ x \in \mathbb{R}^n \mid \langle a_i, x \rangle + b_i \ge 0 \quad \text{for } i = 1, \ldots, m \}
\end{equation}
where $a_i \in \mathbb{R}^n$ is an inward-pointing normal vector to the facet $F_i$, and $b_i \in \mathbb{R}$. We can also write this as a matrix inequality:
\[
    \Delta = \{ x \in \mathbb{R}^n \mid Ax + \mathbf{b} \ge \mathbf{0} \}
\]
where $A$ is the $m \times n$ matrix whose rows are $a_1,\dots,a_m$, and $\mathbf{b} \in \mathbb{R}^m$.
$$
A=\begin{pmatrix}
a_{1}\\..\\a_{m}
\end{pmatrix}
\quad
a_{i}=(a_{i1},\dots,a_{in})
$$

Consider the affine map $i_A: \mathbb{R}^n \to \mathbb{R}^m$ defined by
\[
    i_A(x) = Ax + \mathbf{b}.
\]
Since the normal vectors to the facets of $\Delta$ span 
$\mathbb{R}^n$, the matrix $A$ has full column rank $n$. Thus, $i_A$ is an affine embedding and the image $i_A(\mathbb{R}^n)$ is an $n$-dimensional affine plane in $\mathbb{R}^m$.
The condition $Ax + \mathbf{b} \ge \mathbf{0}$ implies that the intersection of the non-negative orthant $\mathbb{R}^m_+$ and the $n$ dimensional plane $i_A(\mathbb{R}^n)$ is the embedded copy of $\Delta$ in $\mathbb{R}^m$,
\[
    \Delta \cong i_A(\Delta) = i_A(\mathbb{R}^n) \cap \mathbb{R}_{+}^m \,.
\]


\subsection{Gale Dual}

An $n$-dimensional plane in $\mathbb{R}^m$ is defined by exactly $k = m - n$ linearly independent normal directions. These normal directions are precisely what the Gale dual captures. Let $\Gamma$ be a $k \times m$ matrix whose rows form a basis for the left null space of the matrix $A$. 
\[
\Gamma=\begin{pmatrix}
\gamma_{1}&..&..&\gamma_{m}
\end{pmatrix}=\begin{pmatrix}
    \lambda_1\\..\\\lambda_{m-n}
\end{pmatrix}
\]
That is, $\Gamma$ is a matrix of full rank $k = m - n$ satisfying:
\[
    \Gamma A = 0=A^T\Gamma^T.
\]
Geometrically, the rows $\lambda_i$ of $\Gamma$ are orthogonal to the column space of $A$. Since the column space of $A$ gives the directions of our affine plane $i_A(\mathbb{R}^n)$, the rows of $\Gamma$ generate the normal space to the embedded plane $i_A(\mathbb{R}^n)$ in $\mathbb{R}^m$.

Given $y \in \mathbb{R}^m$, it belongs to the plane $i_A(\mathbb{R}^n)$ if and only if its projection onto this orthogonal normal space is a specific constant. Specifically, if $y = Ax + \mathbf{b}$, multiplying by $\Gamma$ yields
\[
    \Gamma y = \Gamma A x + \Gamma \mathbf{b} = 0 + \Gamma \mathbf{b} = \mathbf{c}.
\]
where $\mathbf{c} = \Gamma \mathbf{b} \in \mathbb{R}^{m-n}$. 

Thus,
\begin{equation}\label{eq:polytope_Gale}
    i_A(\Delta)= \{ y \in \mathbb{R}_{+}^m \mid \Gamma y = \mathbf{c} \}.
\end{equation}

In the terminology of \cite{Buc}, the columns of $\Gamma$, denoted $\gamma_1, \ldots, \gamma_m \in \mathbb{R}^{m-n}$ where $\gamma_{j}^{T}=(\gamma_{1j},...,\gamma_{(m-n)j})$, form the Gale dual configuration of the rows $a_1, \ldots, a_m $ of $A$. 

\section{Moment-Angle Manifolds}
 Consider the map $\pi: \mathbb{C}^m \to \mathbb{R}_{+}^m$ defined by taking the squared magnitude of each coordinate,
\begin{equation}\label{eq: map pi}
    \pi(z_1, \ldots, z_m) = (|z_1|^2, \ldots, |z_m|^2) \,.
\end{equation}
 For every point $y$ strictly interior to the orthant (that is, $y_i > 0$ for each $i$), the preimage $\pi^{-1}(y)$ is an $m$-dimensional torus $\mathbb{T}^m=(S^1)^m$, where the radius of the $i$-th circle is $\sqrt{y_i}$. If a coordinate $y_i = 0$ (the point lies on a wall of the orthant), the corresponding circle collapses to a point.

The moment-angle manifold $\mathcal{Z}_{\Delta}$ is simply the preimage of the slice $i_A(\Delta)$ under this map,
\begin{equation}\label{eq: moment angle}
    \mathcal{Z}_{\Delta} = \pi^{-1}(i_A({\Delta})).
\end{equation}

As a result, the following diagram is commutative,
\begin{center}
\begin{tikzcd}
\mathcal{Z}_{\Delta} \arrow[r, "\mathcal{E}"] \arrow[d, "\pi_{\mathcal{Z}_{\Delta}}"]
& \mathbb{C}^m \arrow[d, "\pi" ] \\
{\Delta} \arrow[r, "i_A"]
& \mathbb{R}^m_{+}
\end{tikzcd}
\end{center}
where $\mathcal{E}:\mathcal{Z}_{\Delta}\to \mathbb{C}^m$ is an embedding.

To describe $\mathcal{Z}_{\Delta}$, we may combine \eqref{eq:polytope_Gale} with the map $\pi$ by substituting $y_i = |z_i|^2$ into the linear equation $\Gamma y = \mathbf{c}$.

\begin{theorem}[Buchstaber--Panov \cite{Buc}]\label{thm: momemt amgle}
The moment-angle manifold $\mathcal{Z}_{\Delta}$ is exactly the intersection of $m-n$ real quadrics (Hermitian quadrics) in $\mathbb{C}^m$, determined entirely by the Gale transform $\Gamma$ and translation vector $\mathbf{c}$,
\[
    \mathcal{Z}_{\Delta} = \left\{ \mathbf{z} \in \mathbb{C}^m \;\middle|\; \sum_{i=1}^m \gamma_i |z_i|^2 = \mathbf{c} \right\}.
\]
\end{theorem}
Because ${\Delta}$ is a simple polytope, these quadrics intersect transversally, guaranteeing that $\mathcal{Z}_{\Delta}$ is a smooth closed manifold of dimension $m+n$. Furthermore, the $m$-dimensional compact torus $\mathbb{T}^m$ acts naturally on $\mathcal{Z}_{\Delta}$ by rotating the phases of the coordinates $z_i$, and the orbit space modulo this action is diffeomorphic as a manifold with corners to the original polytope ${\Delta}$.

\subsection{Quasitoric Manifolds}

We describe now a class of manifolds known as quasitoric manifolds obtained by taking the quotient of $\mathcal{Z}_{\Delta}$ by a freely acting $(m-n)$-dimensional subtorus.\\

Let ${\Delta} \subset \mathbb{R}^n$ be a simple $n$ polytope with $m$ facets $F_1, \dots, F_m$. A quasitoric manifold $Q$ over ${\Delta}$ is a $2n$-dimensional compact smooth manifold with a locally standard action of the $n$-dimensional torus $\mathbb{T}^n$, such that the orbit space modulo this action is diffeomorphic, as a manifold with corners, to ${\Delta}$.

The combinatorial data required to build $Q$ is captured by a {\it characteristic map} which assigns to each facet $F_i$ a primitive vector $\eta_i \in \mathbb{Z}^n$.  For the resulting torus action on the manifold to be locally standard, the characteristic vectors must satisfy the \textit{smoothness condition}: if $n$ facets $F_{i_1}, \dots, F_{i_n}$ intersect at a vertex of ${\Delta}$, their corresponding vectors $\eta_{i_1}, \dots, \eta_{i_n}$ must form a $\mathbb{Z}$-basis for $\mathbb{Z}^n$. The $n \times m$ integer matrix $\eta$ formed by these column vectors is called the {\it characteristic matrix}.

The characteristic map induces a surjective homomorphism of tori
\[
    \ell:\mathbb{T}^m \to \mathbb{T}^n
\]
where the standard basis vectors $\mathbf{e}_i$ of the Lie algebra of $\mathbb{T}^m$ are mapped to $\eta_i$.
The kernel of the characteristic map $\ell$ defines a closed subtorus $H \subset\mathbb{T}^m$ of dimension $k = m - n$:
\[
    H = \operatorname{Ker}(\ell) \cong\mathbb{T}^{m-n}.
\]
The smoothness condition guarantees that the $(m-n)$ dimensional subtorus $H$ acts freely on the moment-angle manifold $\mathcal{Z}_{\Delta}$. 

Because the action of $H$ on $\mathcal{Z}_{\Delta}$ is free and proper, the orbit space is a smooth manifold of dimension $(m+n) - (m-n) = 2n$. We define the quasitoric manifold  as this quotient,
\[
    Q = \mathcal{Z}_{\Delta} / H.
\]
The residual torus $\mathbb{T}^m / H \cong \mathbb{T}^n$ naturally acts on $Q$, and the orbit space is simply $(\mathcal{Z}_{\Delta} / H) / (\mathbb{T}^m / H) \cong \mathcal{Z}_{\Delta} /\mathbb{T}^m \cong {\Delta}$.

 The Marsden-Weinstein-Meyer theorem (symplectic reduction) induce a symplectic structure on a quasitoric manifold precisely when $\Delta$ comes from a Delzant polytope \cite{Del}. 

\section{Free Torus Action }\label{free torus action}


We will investigate the family of subtori of $\mathbb{T}^m$  that act freely on $\mathcal{Z}_{\Delta}$. 
A $d$-parameter subgroup of $\mathbb{T}^m$  is of the form
\begin{equation}\label{eq: torus T_D}
   \mathbb{T}^d_{D}=\{(e^{2\pi i\langle\delta_{1},\phi\rangle},\dots,e^{2\pi i\langle\delta_{m},\phi\rangle})\in\mathbb{T}^m\} 
\end{equation}
where $\phi =(\phi_{1},\dots,\phi_{d})\in \mathbb{R}^d$ and $D=(\delta_{1},\ldots,\delta_{m})$ is a tuple of $m$ vectors in $\mathbb{R}^d$. Denote any element of $\mathbb{T}^d_{D}$ as $$g(\phi)=(e^{2\pi i\langle\delta_{1},\phi\rangle},\dots,e^{2\pi i\langle\delta_{m},\phi\rangle})\,.$$
Define the linear map
$$
\mathscr{L}:\mathbb{R}^{d}\rightarrow \mathbb{R}^m
$$
to be the left multiplication by the matrix 
$$
(\mathscr{L})=\begin{pmatrix}
\delta_{1}\\..\\..\\ \delta_{m}
\end{pmatrix}
$$
$\mathbb{T}^d_{D}$ is a subtorus of dimension $d$ (in other words, a closed subgroup of $\mathbb{T}^m$)  if $\operatorname{rank}(\mathscr{L})=d$ and $D$ satisfies the following equivalent conditions:
\begin{enumerate}
\item  There exist $\mathbb{R}$-linearly independent vectors $z_{1},\ldots ,z_{d}\in \mathbb{Z}^m$ such that each $z_{i}\in \operatorname{Im}(\mathscr{L})$.
\item There exist linearly independent vectors $v_{1},...,v_{d}\in \mathbb{R}^m$ such that they form a basis of $\mathscr{L}(\mathbb{R}^d)$ and each vector lies in $\mathbb{Q}^m$.
\end{enumerate}

 It turns out that the freeness of the group action of $\mathbb{T}^d_{D}$ on $\mathcal{Z}_{\Delta}$ does not depend on the geometric realisation of the polytope, i.e. on the normal vectors to the facets of the polytope. It only depends on the combinatorial type of the polytope.

Suppose we have $z \in \mathcal{Z}_{\Delta}$ such that $z=(0,\dots,0,z_{{r+1}},\dots,z_{m})$ with $z_i\not=0$ for all $(r+1) \leq i\leq m$ and $\mathbb{T}^d_{D}$ is acting on it. Now, given any $g(\phi) \in \mathbb{T}^d_{D}$,  its first $r$ components 
will act trivially on the first $r$ components of $z$. So $g(\phi)$ will act  trivially on $z$ if and only if the last $m-r$ components are identity. 
For $\mathbb{T}^d_{D}$ to act freely on $z$, if there is an element $g(\phi)$ such that its last $m-r$ components are identity, then necessarily its first $r$ component must be identity. 
Suppose, there exists a $\phi$ such that $\langle\delta_{j},\phi\rangle \in \mathbb{Z} \quad \forall r+1 \leq j\leq m$. Then for it to act freely on $z$,  $\langle\delta_{i},\phi\rangle$ must be an integer for all $1\leq i\leq r$.  

To apply the above observation, 
we define a stratification of $\mathcal{Z}_{\Delta}$, based on a stratification of $\Delta$ using the relative interiors of its faces. Recall the definition of $\Delta$ by the system of inequalities \eqref{polytope in Rn}.



\begin{construction}\label{6.1}
Define a stratification of ${\Delta}$ by its faces. Let $F$ be a face of ${\Delta}$ with $\operatorname{dim} F = n-r$. Then $F$ is characterized as a subset of ${\Delta}$ by $r$ equations
\begin{equation}\label{eq: F_I}
    \langle x, a_i \rangle +b_{i} = 0, \quad i = i_1 , \ldots , i_r 
\end{equation} 
where $1 \leq i_1 < i_2 \ldots < i_r \leq m$. 
We write $F = F_I$ where $I = (i_1, \ldots, i_r)$. Note that $i_A(F_{I})$ is the intersection of $i_A({\Delta})$ and the plane $Pl_I=\{y \in \mathbb{R}^m_{+}\,|\, y_{i}=0 \quad i \in I\}$. 
Let $F_I^{\circ}$ denote the relative interior of $F_I$. 
If $I=\emptyset$, then $F^{\circ}_{\emptyset}$ represents the interior of the polytope $\Delta$. Denote by $\mathscr{I}$, the set of  tuples $I$ such that $F_I$ is a face of $\Delta$. The $F_I^{\circ}$'s such that $I\in \mathscr{I}$, form a stratification of $\Delta$.
Define 
\begin{equation}\label{eq: stratify}
    N_I := \pi^{-1}(i_A(F_I^{\circ}) ) \,.
\end{equation} 
Then, the $N_I$'s such that $I\in \mathscr{I}$, form a stratification of $\mathcal{Z}_\Delta$.
\end{construction}
 Each $N_{I}$ is closed under $\mathbb{T}^m$ action. So, this stratification is  the stratification of $\mathcal{Z}_{\Delta}$ into $\mathbb{T}^m$ orbit types. More specifically, the stabilizer of $N_{I}$ is 
\begin{equation}\label{TI}
    \mathbb{T}_I := \{(e^{2\pi i t_1}, \ldots, e^{2\pi i t_m}) \mid e^{2\pi i t_s} = 1, \forall s \notin I\}\,.
\end{equation}

So $\mathbb{T}^d_{D}$ will act freely on $\mathcal{Z}_{\Delta}$ if and only if $\mathbb{T}^d_{D}\cap \mathbb{T}_{I}= 1 \, \forall I$ such that $F_{I}$ is a face of ${\Delta}$.

For $I = (i_1, \ldots, i_r)$, define  its complement $J=(j_1,\dots ,j_{m-r})$ where $1\leq j_1<..<j_{m-r}\leq m$ and no $j_k$ is in $I$.
Define two matrices with certain $\delta_{i}$'s as column vectors,
$$ \mathscr{D}_I = \begin{pmatrix}
\delta_{i_{1}}& \dots & \delta_{i_{r}}
\end{pmatrix}
=\begin{pmatrix}
u_{1}\\..\\u_{d}
\end{pmatrix}
$$
$$
\mathscr{D}^c_{I}=\begin{pmatrix}
\delta_{j_{1}}&\dots&\delta_{j_{m-r}}
\end{pmatrix}
=\begin{pmatrix}
w_{1}\\..\\w_{d}
\end{pmatrix}
$$

\begin{theorem}   $\mathbb{T}^d_{D}\cap \mathbb{T}_{I}=1$ if and only if , whenever
there exists $(\phi_{1},\dots,\phi_{d})\in \mathbb{R}^{d}$ such that $\sum_{i}\phi_{i}w_{i}\in \mathbb{Z}^{m-r}$ then $\sum_{i}\phi_{i}u_{i}\in \mathbb{Z}^r$.
\end{theorem}
\begin{proof} Assume $\mathbb{T}^d_{D}\cap \mathbb{T}_{I}=1$. Suppose there exist $\phi$ such that $\sum_{i}\phi_{i}w_{i}\in \mathbb{Z}^{m-r}$. Now this is equivalent to $\langle\phi,\delta_{j_{k}}\rangle \in \mathbb{Z}$ for all $j_{k} \not \in I$. That means the corresponding $g(\phi) \in \mathbb{T}^d_{D}$ also belongs to $\mathbb{T}_{I}$. By the above assumption, $g(\phi)=1$ and thus $\langle\phi,\delta_{i_{k}}\rangle\in \mathbb{Z}$ for all $i_{k} \in I$ which is equivalent to $\sum_{i}\phi_{i}u_{i}\in \mathbb{Z}^r$.

Now assume that whenever
there exist $(\phi_{1},\dots,\phi_{d})\in \mathbb{R}^{d}$ such that $\sum_{i}\phi_{i}w_{i}\in \mathbb{Z}^{m-r}$ then $\sum_{i}\phi_{i}u_{i}\in \mathbb{Z}^r$.
Let $g(\phi)\in\mathbb{T}^d_{D}\cap \mathbb{T}_{I}$. Since $g(\phi) \in \mathbb{T}_{I}$, the component of $g(\phi)$ with index $j_{k} \not \in I$ are the identity. Thus  $\langle\phi,\delta_{j_{k}}\rangle \in \mathbb{Z}$ for all $j_{k} \not \in I$. This implies $\sum_{i}\phi_{i}w_{i}\in \mathbb{Z}^{m-r}$ and by assumption $\sum_{i}\phi_{i}u_{i}\in \mathbb{Z}^r$. Hence, $\langle\phi,\delta_{i_{k}}\rangle\in \mathbb{Z}$ for all $i_{k} \in I$ and so $g(\phi)=1$. 
\end{proof}

Note that $\mathbb{T}_{I}$ is maximal when $F_{I}$ is a vertex, because $\mathbb{T}_{J} \subset \mathbb{T}_{I}$  if and only if $F_{I} \subset F_{J}$. 

\begin{definition}\label{Tv}
     Let $v$ be a vertex of $\Delta$ and 
$\mathbb{T}_{v}$ be the stabilizer subgroup of the stratum $N_I$ corresponding to the face $v$. 
\end{definition}

\begin{proposition}\label{Prop: free action} The action of $\mathbb{T}^d_{D}$ is free if and only if $\mathbb{T}^d_{D}\cap \mathbb{T}_{v}=1\quad \forall$ vertices $v$ in ${\Delta}$.
\end{proposition}
\begin{proof}
    If the action of $\mathbb{T}^d_{D}$ is free then its intersection with each stabilizer subgroup $\mathbb{T}_I$, and in particular with each $\mathbb{T}_{v}$,  is trivial.
    Conversely, suppose the intersection of  $\mathbb{T}^d_{D}$ with each $\mathbb{T}_{v}$ is trivial.
    Since each stabilizer subgroup $\mathbb{T}_J$  is contained in the stabilizer subgroup of a vertex, the intersection of $\mathbb{T}^d_{D}$ with each $\mathbb{T}_{J}$ is trivial. Then the action of $\mathbb{T}^d_{D}$ is free.
\end{proof}

\section{Tangent Spaces of $\mathcal{Z}_{\Delta}$}\label{tangent spaces of Z}


Recall from Theorem \ref{thm: momemt amgle} that the moment-angle manifold $\mathcal{Z}_{\Delta}$ may be defined as
$$ \mathcal{Z}_{\Delta} = \{z\in\mathbb{C}^{m}|\quad\Gamma|z|^{2}=\Gamma \mathbf{b}=\mathbf{c}\} $$
where $|z|^{2}:=(|z_{1}|^{2}, \ldots ,|z_{m}|^{2})^T$. Define $\mu:\mathbb{C}^{m}\rightarrow\mathbb{R}^{m-n}$ by 
\begin{equation}\label{eq: mu}
    \mu(z)=\Gamma|z|^{2}.
\end{equation} Then, $\mathcal{Z}_{\Delta}=\mu^{-1}(\mathbf{c})$, and the tangent space of $\mathcal{Z}_{\Delta}$ at $p$ is given by $$T_{p}\mathcal{Z}_{\Delta}=\operatorname{Ker}(d\mu_{p})\subset T_{p}\mathbb{C}^{m}\,.$$

The tangent spaces of $\mathcal{Z}_{\Delta}$ can be analyzed by using the stratification \eqref{eq: stratify}, 
 $$ \mathcal{Z}_{\Delta} := \bigsqcup N_I\,.$$ 
Recall that given $I=(i_{1},...,i_{r})$ where $1\le i_{1}<i_{2}...<i_{r}\le m$ such that $F_I$ is a face of $\Delta$ \eqref{eq: F_I}, $F^{\circ}_{I}$ is the relative interior of $F_{I}$, and $N_{I}=\pi^{-1}(i_{A}(F^{\circ}_{I}))$.  Let $I^c=J=(j_{1},...,j_{m-r})$ where $1\le j_{1}<j_{2}...<j_{m-r}\le m$. By construction,
$$ N_{I}=\{z\in i(\mathcal{Z}_{\Delta})\,|\, z_{k}= 0,z_{l}\ne0 \text{ for } k\in I, l\in J\} \,.$$

At a point $p=(x_{1},y_{1},...,x_{m},y_{m})\in\mathcal{Z}_{\Delta}$, the differential of $\mu$ is expressed as
$$ d\mu_{p} = \begin{pmatrix} 2\gamma_{11}x_{1} & 2\gamma_{11}y_{1} & \cdots & 2\gamma_{1m}x_{m} & 2\gamma_{1m}y_{m} \\ \vdots & \vdots & \ddots & \vdots & \vdots \\ 2\gamma_{(m-n)1}x_{1} & 2\gamma_{(m-n)1}y_{1} & \cdots & 2\gamma_{(m-n)m}x_{m} & 2\gamma_{(m-n)m}y_{m} \end{pmatrix} $$

\begin{construction}\label{7.1}
Define a family of linear maps indexed by $I$, $\mathscr{E}_{I}:\mathbb{R}^{m-r}\rightarrow\mathbb{R}^{m-n}$, given by left multiplication by the matrix,
$$ \mathscr{E}_{I}=\begin{pmatrix}\gamma_{j_1}&...&\gamma_{j_{m-r}}\end{pmatrix} $$
where the column vectors $\gamma_{j_i}$ come from the Gale dual $\Gamma$.

If $F_{I}$ is a face of $\Delta$, then the vectors $\{a_{i_{1}},\ldots,a_{i_{r}}\}$ are linearly independent and by Theorem 1.2.4 of \cite{Buc}, the vectors $\{\gamma_{j_{1}},\ldots,\gamma_{j_{m-r}}\}$ span $\mathbb{R}^{m-n}$. By the rank-nullity theorem, the dimension of $\mathscr{F}_{I}=\operatorname{Ker}(\mathscr{E}_{I})$ is $(n-r)$. Let 
$$\{\nu_{1}^{I},\ldots,\nu_{n-r}^{I}\} \text{ denote a basis of } \mathscr{F}_{I}
$$
where $(\nu_{j}^{I})^{T}=(_{1}\nu_{j}^{I},...,_{(m-r)}\nu_{j}^{I})$. Define vectors $\mu_{1}^{I},...,\mu_{n-r}^{I}\in\mathbb{R}^{m}$, 
$$
(\mu_{k}^{I})^{T}:=(_{1}\mu_{k}^{I},...,_{m}\mu_{k}^{I}) \quad {}_{i_{\alpha}}\mu_{k}^{I}=0 \quad {}_{j_{\beta}}\mu_{k}^{I}=_{\beta}\nu_{k}^{I}
$$for $1\le\alpha\le r$ and $1\le\beta\le m-r$. Each $\mu^I_{k}$ is constructed by inserting zeros at the coordinate positions indexed by $I$, while the remaining components are derived from $\nu_{k}^{I}$, preserving their original order. 
The construction associates a pair of sets of vectors $\{\nu_{1}^{I},...,\nu_{n-r}^{I}\}$ and $\{\mu_{1}^{I},...,\mu_{n-r}^{I}\}$ to each index set $I$ for which $F_{I}$ is a face of $\Delta$. 
\end{construction}

\begin{proposition}\label{7.2}
$\Gamma\mu_{k}^{I}=0$ for all $1\leq k\leq n-r$. Consequently, $\operatorname{span}\{\mu_{1}^{I},...,\mu_{n-r}^{I}\}\subset\operatorname{Ker}(\Gamma)$.
\end{proposition}

\begin{proof}
The result follows from direct computation:
$$ \Gamma\mu_{k}^{I}=\sum_{i=1}^{m}(_{i}\mu_{k}^{I}\gamma_{i})=\sum_{\alpha=1}^{r}(_{i_{\alpha}}\mu_{k}^{I}\gamma_{i_{\alpha}})+\sum_{\beta=1}^{m-r}(_{j_{\beta}}\mu_{k}^{I}\gamma_{j_{\beta}})=\sum_{\beta=1}^{m-r}(_{\beta}\nu_{k}^{I}\gamma_{j_{\beta}})=0 $$
\end{proof}
\begin{construction}\label{7.3}
Consider the vectors $$\chi_{1}^{I},\dots,\chi_{r}^{I},\rho_{1}^{I},...,\rho_{r}^{I},\zeta_{1}^{I},...,\zeta_{n-r}^{I},\phi_{1}^{I},...,\phi_{m-r}^{I} \in T_{p}\mathbb{C}^{m}$$ defined below, for $p\in N_{I}$, assuming that the derivatives below are evaluated at $p$ unless specified otherwise:
$$ \chi_{k}^{I}=\frac{\partial}{\partial x_{i_{k}}}, \quad \rho_{k}^{I}=\frac{\partial}{\partial y_{i_{k}}}, \quad \zeta_{k}^{I}=\sum_{l=1}^{m-r}\left({}_{l}\nu_{k}^{I}\left(\frac{x_{j_{l}}}{p_{j_{l}}}\frac{\partial}{\partial x_{j_{l}}}+\frac{y_{j_{l}}}{p_{j_{l}}}\frac{\partial}{\partial y_{j_{l}}}\right)\right), $$ 
$$\quad \phi_{k}^{I}=\left(y_{j_{k}}\frac{\partial}{\partial x_{j_{k}}}-x_{j_{k}}\frac{\partial}{\partial y_{j_{k}}}\right) $$
\end{construction}

\begin{theorem}\label{7.4}
The tangent vectors $\chi_{1}^{I},\dots,\chi_{r}^{I},\rho_{1}^{I},...,\rho_{r}^{I},\zeta_{1}^{I},...,\zeta_{n-r}^{I},\phi_{1}^{I},...,\phi_{m-r}^{I}$ form a basis for $T_{p}\mathcal{Z}_{\Delta}$.
\end{theorem}

\begin{proof}
It must first be established that each vector belongs to the kernel of $d\mu_{p}$.
\vspace{0.3cm}
$$ d\mu_{p}(\chi_{k}^{I})=\begin{pmatrix}2\gamma_{1i_{k}}x_{i_{k}}\\ \vdots \\ 2\gamma_{(m-n)i_{k}}x_{i_{k}}\end{pmatrix}=\begin{pmatrix}0\\ \vdots \\ 0\end{pmatrix} $$
\vspace{0.5cm}
$$ d\mu_{p}(\rho_{k}^{I})=\begin{pmatrix}2\gamma_{1i_{k}}y_{i_{k}}\\ \vdots \\ 2\gamma_{(m-n)i_{k}}y_{i_{k}}\end{pmatrix}=\begin{pmatrix}0\\ \vdots \\ 0\end{pmatrix} $$
\vspace{0.5cm}
$$ d\mu_{p}(\zeta_{k}^{I})=\begin{pmatrix}\sum_{l}( 
{}_{l}\nu_{k}^{I}\gamma_{1j_{l}}(x_{j_{l}}^{2}+y_{j_{l}}^{2})/p_{j_{l}})\\ \vdots \\ \sum_{l}({}_{l}\nu_{k}^{I}\gamma_{(m-n)j_{l}}(x_{j_{l}}^{2}+y_{j_{l}}^{2})/p_{j_{l}})\end{pmatrix}=\begin{pmatrix}\sum_{l}({}_{l}\nu_{k}^{I}\gamma_{1j_{l}})\\ \vdots \\ \sum_{l}({}_{l}\nu_{k}^{I}\gamma_{(m-n)j_{l}})\end{pmatrix}=\begin{pmatrix}0\\ \vdots \\ 0\end{pmatrix} $$
\vspace{0.5cm}
$$ d\mu_{p}(\phi_{k}^{I})=\begin{pmatrix}2\gamma_{1j_{k}}x_{j_{k}}y_{j_{k}}-2\gamma_{1j_{k}}y_{j_{k}}x_{j_{k}}\\ \vdots \\ 2\gamma_{(m-n)j_{k}}x_{j_{k}}y_{j_{k}}-2\gamma_{(m-n)j_{k}}y_{j_{k}}x_{j_{k}}\end{pmatrix}=\begin{pmatrix}0\\ \vdots \\ 0\end{pmatrix} $$
\vspace{0.3cm}

Since the four groups involve disjoint coordinate blocks and the vectors in $\{\nu_{1}^{I},\ldots,\nu_{n-r}^{I}\}$ are linearly independent, the tangent vectors in Constructed \ref{7.3}  are linearly independent. As the dimension of $\mathcal{Z}_{\Delta}$ is $m+n$, the statement holds true.

\end{proof}

\begin{proposition}\label{7.5}
Let $K$ be an index tuple such that $K\subset I$, and $F_{K}$ is a face. Then $F_{I}\subset F_{K}$, and $\operatorname{span}(\zeta_{1}^{I},\dots,\zeta_{n-r}^{I})\subset T_{p}\mathcal{Z}_{\Delta}$ for all $p\in N_{K}$.
\end{proposition}

\begin{proof}
The proof proceeds similarly to the preceding theorem.
\end{proof}

\begin{proposition}\label{7.6}
The space spanned by $\{\phi_{1},\dots,\phi_{m}\}$ is a subspace of $T_{p}\mathcal{Z}_{\Delta}$ for all $p\in\mathcal{Z}_{\Delta}$, where $\phi_{i}=y_{i}\frac{\partial}{\partial x_{i}}-x_{i}\frac{\partial}{\partial y_{i}}$.
\end{proposition}

\begin{proof}
For $p=(x_{1},y_{1},...,x_{m},y_{m})$ in $\mathcal{Z}_{\Delta}$, we have,
$$ d\mu_{p}(\phi_{i})=\begin{pmatrix}2\gamma_{1i}x_{i}y_{i}-2\gamma_{1i}y_{i}x_{i}\\ \vdots \\ 2\gamma_{(m-n)i}x_{i}y_{i}-2\gamma_{(m-n)i}y_{i}x_{i}\end{pmatrix}=\begin{pmatrix}0\\ \vdots \\ 0\end{pmatrix} $$
\end{proof}

\section{Reducible Poisson Structures}\label{Reducible structures}

In this section, we construct Poisson bivectors on $\mathbb{C}^{m}$ that admit Marsden-Ratiu reduction (see Section \ref{subsec: reduction}) under the free action of a subtorus $\mathbb{T}^d_{D}$ \eqref{eq: torus T_D} on the submanifold $M=\mathcal{Z}_{\Delta}$ of $P=\mathbb{C}^{m}$. We construct four such families of Poisson bivector fields, denoted $\Pi_{1}, \Pi_{2}, {}_{K}\Pi_{3}^{R}$, and ${}_{K}\Pi_{4}^{R}$. 

Let $E = E_{D}^{d}$ be the distribution generated by the  action of $\mathbb{T}_{D}^{d}$ on $\mathbb{C}^{m}$. This distribution inherently satisfies the conditions (A1) and (A2) of Section \ref{subsec: reduction}, so it remains to check conditions (A3) and (A4). 

The first family of bivector fields is defined as
$$ \Pi_{1} = \sum_{1 \le i < j \le m} a_{ij} \Theta_i \wedge\Theta_j$$
where each $a_{ij} \in C^{\infty}(\mathbb{C}^{m})$ satisfies $\Theta_k(a_{ij}) = 0$ for all $1 \le k \le m$. This constraint dictates that each $a_{ij}$ is strictly a function of the radial coordinates $p_{k}$ for $k \in \{1, \dots, m\}$.
\vspace{0.5cm}

The second family of bivector fields is given by
$$ \Pi_{2} = \sum_{q=1}^{m} \frac{\partial}{\partial x_{q}} \wedge \frac{\partial}{\partial y_{q}} + \sum_{1 \le i < j \le m} b_{ij} \Theta_i \wedge\Theta_j $$
where $b_{ij} \in C^{\infty}(\mathbb{C}^{m})$ are smooth functions such that $b_{ij}$ depends exclusively on $p_{i} $ and $p_{j}$, meaning $b_{ij} = b_{ij}(p_{i}, p_{j})$. 
\vspace{0.5cm}

To describe the remaining two families, we consider index tuples $I = (i_{1}, \dots, i_{r})$ with $1 \le i_{1} < \dots < i_{r} \le m$ such that $F_{I}$ forms a codimension $r$ face of $\Delta$ and $m-r \le 3$. Let $I_{\Delta}^{3}$ denote the set of such tuples for $\Delta$. For each $K = (k_{1}, \dots, k_{q}) \in I_{\Delta}^{3}$, let $L = (l_{1}, \dots, l_{m-q})=K^c$, and  define $P^{K} = \prod_{i=1}^{m-q} p_{l_{i}}$. The third and fourth families of Poisson bivector fields are then defined as follows.

$$ {}_{K}\Pi_{3}^{R} = \sum_{1 \le i < j \leq (n-q)} P^{K} \sigma_{ij}^{K} \zeta_{i}^{K} \wedge \zeta_{j}^{K} $$
where $\sigma_{ij}^{K}$ are constant functions.

$$ {}_{K}\Pi_{4}^{R} = \sum_{1 \le i < j \leq (n-q)} P^{K} \upsilon_{ij}^{K} \zeta_{i}^{K} \wedge \zeta_{j}^{K} + \sum_{1 \le i < j \le m} \tau_{ij} \Theta_i \wedge\Theta_j $$
where $\upsilon_{ij}^{K}$ and $\tau_{ij}$ are constant functions.

\begin{theorem}\label{8.1}
The families  $\Pi_{1}, {}_{K}\Pi_{3}^{R}$, and ${}_{K}\Pi_{4}^{R}$ of Poisson structures on $\mathbb{C}^m$ are reducible in the sense of Marsden and Ratiu under the free  action of the torus $\mathbb{T}_{D}^{d}$ on $\mathcal{Z}_{\Delta} \subset \mathbb{C}^m$.
\end{theorem}

\begin{proof}
\textit{Verification of Condition (A3):} By Lemma \ref{2.6} it suffices to prove the invariance of the relevant Poisson bivector field under the action of each $g \in \mathbb{T}_{D}^{d}$. We prove the invariance of the relevant Poisson structures under the canonical action of $\mathbb{T}^{m}$ on $\mathbb{C}^{m}$; the result follows naturally for the subtorus $\mathbb{T}_{D}^{d}$. 

Let $z \in \mathbb{C}^{m}$ with
$$
z=(r_1e^{\iota\theta_{1}},\dots,r_me^{\iota\theta_{m}})=(x_1,y_1,\dots,x_m,y_m)
$$

For any element $g = (e^{\iota t_{1}}, \dots, e^{\iota t_{m}}) \in \mathbb{T}^{m}$, the group action does not alter the radial components: $r_{i}(z) = r_{i}(g \cdot z)$ for all $1 \le i \le m$. In the standard Cartesian coordinates $(x_{1}, y_{1}, \dots, x_{m}, y_{m})$, the action can be expressed as follows.
$$ g \cdot z = (x_{1}\cos(t_{1}) - y_{1}\sin(t_{1}), x_{1}\sin(t_{1}) + y_{1}\cos(t_{1}), \dots, x_{m}\sin(t_{m}) + y_{m}\cos(t_{m})) $$
 The differential $d\Psi_{g}$ is a block-diagonal matrix composed of $2 \times 2$ rotation blocks $R(t_{i})$. Computing the pushforward on angular and radial vector fields yields
$$ d\Psi_{g} \left( \Theta_i|_{z} \right) = \Theta_i|_{g \cdot z}, \quad d\Psi_{g} \left( \mathcal{P}_i|_{z} \right) = \mathcal{P}_i|_{g \cdot z} \quad \forall 1 \le i \le m
$$
 Applying these to the bivector $\Pi_{1}$, we have 
$$ d\Psi_{g}(\Pi_{1}|_{z}) = \Pi_{1}|_{g \cdot z} $$
since the coefficient functions $a_{ij}$ of $\Pi_{1}$  depend solely on the radial coordinates ensuring $a_{ij}(z) = a_{ij}(g \cdot z)$. Similar arguments for ${}_{K}\Pi_{3}^{R}$, ${}_{K}\Pi_{4}^{R}$, and $\Pi_{2}$ show that the canonical $\mathbb{T}^{m}$ action is via Poisson diffeomorphisms, thereby satisfying the condition (A3).\\

\textit{Verification of Condition (A4):} We must satisfy ${\displaystyle \Pi^{\sharp}_q({E}^0_q)\subset T_q \mathcal{Z}_{\Delta} + E_q}$. For $\Pi_{1}$, ${}_{K}\Pi_{3}^{R}$, and ${}_{K}\Pi_{4}^{R}$, the bivectors evaluated at any element of $T_q^*\mathbb{C}^m$ yield tangent vectors contained in $T_{q}\mathcal{Z}_{\Delta}$. Consequently, $\Pi^{\sharp}_{q}(E_{q}^{0}) \subset T_{q}\mathcal{Z}_{\Delta}$ holds independent of $\mathbb{T}_{D}^{d}$.

\end{proof}
\begin{remark}\label{re8.2} Similar arguments can be used to show that $\mathbb{T}^{m}$ acts via Poisson diffeomorphisms on each of the Poisson structure constructed in Section \ref{str on Cm} 
\end{remark}
For $\Pi_{2}$, the reduction necessitates additional constraints on $\Delta$ and the subtorus $\mathbb{T}_{D}^{d}$. The rationality and smoothness conditions formulated by Delzant on simple polytopes in his classification of symplectic toric manifolds \cite[Chapter XI]{Ana}, and the associated constraints that uniquely determine the quotienting torus, are sufficient to guarantee the reducibility of $\Pi_{2}$.
\begin{lemma}
The family  $\Pi_{2}$ of Poisson structures on $\mathbb{C}^m$ is reducible in the sense of Marsden and Ratiu under the free action of the torus $\mathbb{T}_{D}^{d}$ on $\mathcal{Z}_{\Delta} \subset \mathbb{C}^m$, where $\Delta$ is a Delzant polytope and $\mathbb{T}^d_D$ is the associated torus.  
\end{lemma}

 \begin{proof}

\textit{Verification of Condition (A3):} The proof is the same as the verification of (A3) in Theorem \ref{8.1}.\\

\textit{Verification of Condition (A4):} By construction of $\mathbb{T}^d_{D}$ and $\mathcal{Z}_{\Delta}$ from a Delzant polytope, 
$$
	\Pi_{st}^{\sharp}|_{z}\Big({E^d_{D}}^0\Big)\subset T_{z}(\mathcal{Z}_{\Delta})\quad \forall z\in \mathcal{Z}_{\Delta}
$$
	As shown in verification of (A4) in Theorem \ref{8.1}, the component consisting of only angular directions also satisfies (A4). Thus,
	$$
		\Pi^{\sharp}_{2}|_{z}\Big({E^d_{D}}^0\Big)\subset T_{z}(\mathcal{Z}_{\Delta})\quad \forall z\in \mathcal{Z}_{\Delta}
	$$

\end{proof}

\section{Symplectic Leaves of Reducible Poisson Structures}\label{Symplectic leaves}

To ascertain the symplectic leaves of the quotient manifold with its induced Poisson structure, we first examine the symplectic foliation generated by the family of Poisson bivectors on the moment-angle manifold $\mathcal{Z}_{\Delta}$. As explained in Section \ref{subsec: symp foliation},  symplectic leaves are recovered by integrating along the corresponding Hamiltonian directions.

\begin{lemma}\label{9.1}
The manifold $\mathcal{Z}_{\Delta}$ admits Poisson structures $\Pi_{1}|_{\mathcal{Z}_{\Delta}}$, ${}_{K}\Pi_{3}^{R}|_{\mathcal{Z}_{\Delta}}$, and ${}_{K}\Pi_{4}^{R}|_{\mathcal{Z}_{\Delta}}$ induced by $\Pi_{1}$, ${}_{K}\Pi_{3}^{R}$, and ${}_{K}\Pi_{4}^{R}$ respectively.
\end{lemma}

\begin{proof} 
    Since at each point $q\in \mathcal{Z}_{\Delta}$, the image of each of the bivectors, $\Pi_1,{}_K\Pi^R_{3}$ and ${}_K\Pi^R_{4}$, regarded as a map from 
    $T^{*}_q \mathbb{C}^m$ to $T_q\mathbb{C}^m$, lies inside $T_q(\mathcal{Z}_{\Delta})$, the statement follows directly from Lemma 
    \ref{lem: Poisson on submanifold}.
\end{proof}
\subsection{Sum of linear maps}

Let $W$ and $W'$ be finite-dimensional vector spaces. Consider two linear maps:
$$ A_{1}: W \rightarrow W', \quad A_{2}: W \rightarrow W' $$



\begin{definition}\label{9.2}
If $\operatorname{Ker}(A_1)+\operatorname{Ker}(A_2)=W$, then the pair of linear maps $A_{1}, A_{2}: W \rightarrow W'$  is said to be linearly independent by domain.
\end{definition}

Define the sum of the linear maps as $A = A_{1} + A_{2}: W \rightarrow W'$. 

\begin{lemma}\label{9.3}
If $A_{1}$ and $A_{2}$ are linearly independent by domain, then $\operatorname{Im}(A) = \operatorname{Im}(A_{1}) + \operatorname{Im}(A_{2})$.
\end{lemma}
\begin{definition}\label{9.4}
	If $A_{1}$ and $A_{2}$ are  linear maps that are linearly independent by domain and $\operatorname{Im}(A_{1}) \cap \operatorname{Im}(A_{2})=0$, then they are said to be linearly independent maps.
\end{definition}
\begin{lemma}\label{9.5}
If $A_{1}$ and $A_{2}$ are linearly independent maps, then $\operatorname{Im}(A) = \operatorname{Im}(A_{1}) \oplus \operatorname{Im}(A_{2})$ and $\operatorname{dim}\Big(\operatorname{Im}(A)\Big) = \operatorname{dim}\Big(\operatorname{Im}(A_{1})\Big) + \operatorname{dim}\Big(\operatorname{Im}(A_{2})\Big)$
\end{lemma}
One can extend the idea to vector bundles. Consider two vector bundle maps:
\begin{center}
\begin{tikzcd}
V\arrow[r, "B_{1}"] \arrow[d, "\pi"]
&V'  \arrow[d, "\pi'" ] \\
M\arrow[r, "id"]
& M
\end{tikzcd}
\begin{tikzcd}
V\arrow[r, "B_{2}"] \arrow[d, "\pi"]
&V'  \arrow[d, "\pi'" ] \\
M\arrow[r, "id"]
& M
\end{tikzcd}
\end{center}
where $V$ and $V'$ are vector bundles over the same base $M$. Let $q\in M$, $V_{q}=\pi ^{-1}(q)\subset V$ and $V'_{q}=\pi '^{-1}(q)\subset V'$. Then,
$$
	 B_{1}|_{V_{q}}: V_{q} \rightarrow V'_{q}, \quad B_{2}|_{V_{q}}: V_{q} \rightarrow V'_{q} 	
$$
are linear maps. Assume $B_{1}|_{V_{q}}$ and $B_{2}|_{V_{q}}$ are linearly independent by domain.
\begin{definition}\label{9.6}
If the above condition is satisfied for each $q\in M$, then the pair of bundle maps $B_{1}, B_{2}: V \rightarrow V'$  are said to be linearly independent by domain.
\end{definition}
Define the sum of bundle maps $B=B_{1}+B_{2}:V\rightarrow V'$.
\begin{lemma}\label{9.7}
If $B_{1}$ and $B_{2}$ are linearly independent by domain, then $\operatorname{Im}(B|_{V_{q}}) = \operatorname{Im}(B_{1}|_{V_{q}}) + \operatorname{Im}(B_{2}|_{V_{q}})$ for each $q\in M$.
\end{lemma}
\begin{definition}
	If $B_{1}$ and $B_{2}$ are  bundle maps that are linearly independent by domain and $\operatorname{Im}(B_{1}|_{V_{q}}) \cap \operatorname{Im}(B_{2}|_{V_{q}})= 0$ for each $q\in M$, then they are said to be linearly independent bundle maps.
\end{definition}
\begin{lemma}\label{9.9}
If $B_{1}$ and $B_{2}$ are linearly independent maps, then $\operatorname{Im}(B|_{V_{q}}) = \operatorname{Im}(B_{1}|_{V_{q}}) \oplus \operatorname{Im}(B_{2}|_{V_{q}})$ and $\operatorname{dim}\Big(\operatorname{Im}(B|_{V_{q}})\Big) = \operatorname{dim}\Big(\operatorname{Im}(B_{1}|_{V_{q}})\Big) + \operatorname{dim}\Big(\operatorname{Im}(B_{2}|_{V_{q}})\Big)$ at each $q\in M$.
\end{lemma}
\subsection{Hamiltonian distribution on $\mathbb{C}^m$}

\begin{definition}\label{def:9.2}
Let $(P, \Pi)$ be a Poisson manifold. A vector $v \in T_{q}P$ is called a Hamiltonian vector at $q \in P $ if $v \in \operatorname{Im}(\Pi^{\sharp}_{q})$, see \eqref{Poisson str as map}. The distribution $\mathcal{D}$, generated by Hamiltonian vectors at each point, is called the Hamiltonian distribution.
\end{definition}
By \cite[Theorem 4.1]{Cran}, $\mathcal{D}_q$ equals the tangent space of the symplectic leaf through $q$. Note that 
 if $\Pi_q = \sum_{1 \le i < j \le n} \alpha_{ij}w_{i} \wedge w_{j}$ for $w_{i} \in T_{q}(P)$ and $\alpha_{ij}\in \mathbb{R}$,  then any Hamiltonian vector at $q$ is expressible as a linear combination $ \sum_{i} a_{i} w_{i}$ where $a_{i} \in \mathbb{R}$.

The complex space $\mathbb{C}^m$
can be decomposed into a disjoint union of $\mathbb{T}^m$ orbits $\mathbb{O}_{p} := \pi^{-1}(p)$ of varying dimensions,
for $p = (p_{1},\dots,p_{m}) \in \mathbb{R}_{+}^{m}$. Similarly, $\mathcal{Z}_{\Delta}$ as an embedded submanifold of $\mathbb{C}^m$, can be decomposed as 
$$
	\mathcal{Z}_{\Delta}=\bigcup_{p\in i_{A}(\Delta)}\mathbb{O}_{p}
$$
For a Hamiltonian distribution $\mathcal{D}_{\Theta}$ on $\mathbb{C}^{m}$ containing exclusively angular directions, each maximal integral submanifold (symplectic leaf) lies within a single orbit $\mathbb{O}_{p}$. 

$\mathbb{C}^m$ can also be decomposed into a disjoint union of slices $\mathbb{S}^I_{t}$, transverse to the orbits,
$$
	\mathbb{S}^I_{t}=\{
    z=(r_1e^{\iota t_{1}},\dots,r_{m} e^{\iota t_{m}})\;\text{and} \; p_{i}=r^2_{i}=0 \iff i\in I\}
$$
where $t=(t_{1},\dots,t_{m})\in \mathbb{R}^m$,  $I=(i_{1},\dots,i_{j})$ with $1\leq i_{1}<\dots<i_{j}\leq m$. 
Similarly, $\mathcal{Z}_{\Delta}$ can be decomposed as
$$
	\mathcal{Z}_{\Delta}=\bigcup_{t,I
    } \mathbb{S}_{t}^I\cap \mathcal{Z}_{\Delta}
$$
For a Hamiltonian distribution $\mathcal{D}_{R}$ composed entirely of radial directions, each maximal integral submanifold (symplectic leaf) lies within a single slice. In other words, integrating along a trajectory shifts the point between different orbits $\mathbb{O}_{p}$ while strictly preserving its angular coordinates and the location of any zero entries.
\begin{lemma}\label{9.11}
	Let $\mathcal{D}_{R}$ be a Hamiltonian distribution produced by the Poisson bivector $\Pi_{R}$ such that
at each point $z\in \mathbb{C}^m$, an arbitrary Hamiltonian vector $v\in (\mathcal{D}_{R})_{z}$ is of the form
	$$
		v=\sum_{i=1}^ma_{i}(\mathcal{P}_{i})_{z}
	$$
	Then each symplectic leaf lies within a single slice. 
\end{lemma}
\begin{proof}
Let $c(s)=(x_{1}(s),y_{1}(s),\dots,x_{m}(s),y_{m}(s))$ be an integral curve of a Hamiltonian vector field in $(\mathbb{C}^m,\Pi_R)$ such that $c(0)\in \mathbb{S}^I_{t}$ and the tangent vectors to the path lie in $\mathcal{D}_{R}$.
$$
	c:[0,\alpha ]\rightarrow \mathbb{C}^m\quad \quad dc\left(\frac{d}{ds} \right)\in \mathcal{D}_{R}
$$
Thus, at each $s\in (0,\alpha)$,
$$
	\dot{x_{i}}(s)=x_{i}(s)f_{i}(c(s))\quad \dot{y_{i}}(s)=y_{i}(s)f_{i}(c(s))
$$
$$
	\implies x_{i}(s)=x_{i}(0)\operatorname{\exp}\left(\int_{0}^s f_{i}(c(u))du\right)\quad y_{i}(s)=y_{i}(0)\operatorname{\exp}\left(\int_{0}^s f_{i}(c(u))du\right)
$$
$$
	\implies p_{i}(s)=p_{i}(0)\operatorname{\exp}\left(2\int_{0}^s f_{i}(c(u))du\right) \quad \theta_{j}=\theta_{j}(0)
$$
for each $1\leq i\leq m$  and $j\in I^c$.
By uniqueness of solutions of an ODE, any trajectory like $c(t)$ preserves the angular coordinates and the location of any zero entries. Hence the statement holds true.
\end{proof}

\subsection{Leaf structure of $\Pi_{1}$}

The bivector $\Pi_{1}$ contains exclusively angular directions, and its coefficients $a_{ij}$ depend only on radial coordinates $p_{k}$. Within any specific torus orbit $\mathbb{O}_{p}$, $a_{ij}$ evaluates to a constant $a_{ij}(p)$. As each Hamiltonian direction is confined to the tangent space of $\mathbb{O}_{p}$, the resulting symplectic leaves lie inside the individual orbits. If $F_{\mathbb{O}_{p}}$ is the foliation generated as a result of integrating along the trajectories lying inside $\mathbb{O}_{p}$, then the foliation of $\mathbb{C}^m$ is,
$$
	F_{\mathbb{C}^m}=\bigcup_{p\in \mathbb{R}^m_{+}}F_{\mathbb{O}_{p}}
$$
 Restricting the foliation to the moment-angle manifold results in,
$$ F_{\mathcal{Z}_{\Delta}} = \bigcup_{p \in i_{A}(\Delta)} F_{\mathbb{O}_{p}} $$
\begin{remark} The foliation produced by $(\mathbb{C}^m,\Pi_{1})$ in turn foliates each torus orbit into subgroup orbits. Similarly, the foliation produced by $(\mathcal{Z}_{\Delta},\Pi_{1}|_{\mathcal{Z}_{\Delta}})$ in turn foliates each torus orbit, lying inside $\mathcal{Z}_{\Delta}$, into subgroup orbits.
\end{remark}
\subsection{Leaf structure of ${}_{K}\Pi_{3}^{R}$}

The bivector ${}_{K}\Pi_{3}^{R}$ contains exclusively radial directions and its coefficients depend strictly on radial coordinates $p_{k}$. As each Hamiltonian direction is confined to the tangent space of $\mathbb{S}^I_{t}$, by Lemma \ref{9.11}, the resulting symplectic leaves are completely immersed within these individual slices. 
If $F_{\mathbb{S}^I_{t}}$ is the foliation of $\mathbb{S}^I_{t}$, generated as a result of integrating the Hamiltonian distribution restricted to $\mathbb{S}^I_{t}$, then the foliation of $\mathbb{C}^m$ is, 
$$
		F_{\mathbb{C}^m}=\bigcup_{\substack{I\in \mathscr{I}\\t\in \mathbb{R}^m}}F_{\mathbb{S}^I_{t}}
$$
Since at each $q\in \mathcal{Z}_{\Delta}$, 
$$
{}_{K}\Pi_{3}^{R}|_{q}\in T_{q}(\mathcal{Z}_{\Delta})\wedge T_{q}(\mathcal{Z}_{\Delta})
$$the distribution $\mathcal{D}|_{\mathcal{Z}_{\Delta}}$ produced by ${}_{K}\Pi^R_{3}|_{\mathcal{Z}_{\Delta}}$ is the same as restricting the distribution $\mathcal{D}$ produced by ${}_{K}\Pi_{3}^{R}$ to $\mathcal{Z}_{\Delta}$ and,
$$
	(\mathcal{D}|_{\mathcal{Z}_{\Delta}})|_{q}\subset T_{q}(\mathcal{Z}_{\Delta})\cap T_{q}(\mathbb{S}^I_{t})\quad \forall q\in \mathcal{Z}_{\Delta}
$$
where $I$ and $t$ depend on $q$.
Thus, each symplectic leaf produced by $(\mathcal{Z}_{\Delta},{}_{K}\Pi^R_{3}|_{\mathcal{Z}_{\Delta}})$ will lies within individual $\mathbb{S}^I_{t}\cap\mathcal{Z}_{\Delta}$ for some $I,t$.
If $F_{\mathbb{S}^I_{t}\cap \mathcal{Z}_{\Delta}}$ is the foliation of $\mathbb{S}^I_{t}\cap \mathcal{Z}_{\Delta}$, generated as a result of integrating $\mathcal{D}|_{\mathcal{Z}_{\Delta}}$ restricted to $\mathbb{S}^I_{t}\cap \mathcal{Z}_{\Delta}$, then the foliation of $\mathcal{Z}_{\Delta}$ is, 
$$
		F_{\mathcal{Z}_{\Delta}}=\bigcup_{\substack{I\in \mathscr{I}\\t\in \mathbb{R}^m}}F_{\mathbb{S}^I_{t}\cap \mathcal{Z}_{\Delta}}
$$

Define $(\mathbb{R}^m_{+})^I$ as a face of $\mathbb{R}^m_{+}$,
$$
	(\mathbb{R}^m_{+})^I=\{(p_{1},\dots,p_{m})\in \mathbb{R}^m_{+}|\quad p_{i}=0\iff i\in I\}
$$
where $I=(i_{1},\dots,i_{j})$ with $1\leq i_{1}<\dots<i_{j}\leq m$. The quotient map, $\pi$, restricted to each slice $\mathbb{S}^I_{t}$ is a diffeomorphism onto its image, $(\mathbb{R}^m_{+})^I$.
$$
	\pi|_{\mathbb{S}^I_{t}}:\mathbb{S}^I_{t}\rightarrow (\mathbb{R}^m_{+})^I
$$
Similarly, the quotient map, $\pi$, restricted to each $\mathbb{S}^I_{t}\cap \mathcal{Z}_{\Delta}$ is a diffeomorphism onto its image, $(\mathbb{R}^m_{+})^I\cap i_{A}(\Delta )=i_{A}(F_{I}^o)$. As a result, since each leaf of $(\mathcal{Z}_{\Delta},{}_{K}\Pi^R_{3}|_{\mathcal{Z}_{\Delta}})$ lies inside $\mathbb{S}^I_{t}\cap \mathcal{Z}_{\Delta}$, its image under the quotient map, $\pi$, lies inside some $i_{A}(F^o_{I})$.

 \begin{remark} The foliation produced by $(\mathbb{C}^m, {}_{K}\Pi^R_{3})$ in turn foliates each slice $\mathbb{S}^I_{t}$. As a result, the diffeomorphism $\pi|_{\mathbb{S}^I_{t}}$ transports the foliation of $\mathbb{S}^I_{t}$ to a foliation of $(\mathbb{R}^m_{+})^I$ whose leaves are open subsets of affine subspaces. Similarly, the foliation produced by $(\mathcal{Z}_{\Delta},{}_{K}\Pi^R_{3}|_{\mathcal{Z}_{\Delta}})$ in turn foliates $\mathbb{S}^I_{t}\cap \mathcal{Z}_{\Delta}$ and $F^o_{I}$ such that each leaf in the foliation of $F^o_{I}$, is an open subset of an affine subspace, lying inside $F^o_{I}$.
 \end{remark}
\subsection{Leaf structure of ${}_{K}\Pi_{4}^{R}$}

The structure ${}_{K}\Pi_{4}^{R}$ combines radial components $\Pi_{Ra}$ and angular components $\Pi_{An}$, whose respective images in the tangent space are disjoint and linearly independent. Consequently, the image of ${}_{K}\Pi_{4}^{R}|_{z}$ is equal to the direct sum of the images of its radial and angular components.
The components $\Pi_{Ra}$ and $\Pi_{An}$, are themselves Poisson bivectors and form a pair of linearly independent bundle maps. Thus the flow generated by the Hamiltonian vector fields of the two components commutes and we can simply integrate along radial and angular Hamiltonian directions separately. The product of the resultant pair of maximal submanifolds containing $z$ is the desired symplectic leaf of $(\mathbb{C}^m,{}_{K}\Pi_{4}^{R})$.
Let $F_{An}$ and $F_{Ra}$ be the foliations produced by $(\mathbb{C}^m,\Pi_{An})$ and $(\mathbb{C}^m,\Pi_{Ra})$ respectively. Let $A$ and $B$ be elements of $F_{An}$ and $F_{Ra}$ respectively containing $w=(|w_{1}|e^{\iota s_{1}},\dots,|w_{m}|e^{\iota s_{m}})$. Then $A$ is a subgroup orbit in $\mathbb{O}_{p}$ containing $w$ where $p=(|w_{1}|^{2},\dots,|w_{m}|^{2})$ and $B$ is a submanifold of $\mathbb{S}^I_{s}$ containing $w$ where $s=(s_{1},\dots,s_{m})$. Let $C \subset \mathbb{C}^m$,
\[
	C=\left\{ z=(|z_{1}|e^{\iota t_{1}},\dots,|z_{m}|e^{\iota t_{m}})\in  \mathbb{C}^m\Bigg| \begin{aligned} \; &(|w_{1}|e^{\iota t_{1}},\dots,|w_{m}|e^{\iota t_{m}})\in A \\ & (|z_{1}|e^{\iota s_{1}},\dots,|z_{m}|e^{\iota s_{m}})\in B\end{aligned} \right\}
\]
\begin{lemma}\label{9.12}
	The map,
	$$
		\psi:C\rightarrow A\times B
	$$
	$$
		\psi(|z_{1}|e^{\iota t_{1}},\dots,|z_{m}|e^{\iota t_{m}})=\Big(|w_{1}|e^{\iota t_{1}},\dots,|w_{m}|e^{\iota t_{m}},|z_{1}|e^{\iota s_{1}},\dots,|z_{m}|e^{\iota s_{m}}\Big)
	$$
	is a diffeomorphism.
\end{lemma}
\begin{proof}
	Define,
	$$
		\psi':A\times B\rightarrow C
	$$
	$$
		\psi'\Big(|w_{1}|e^{\iota t_{1}},\dots,|w_{m}|e^{\iota t_{m}},|z_{1}|e^{\iota s_{1}},\dots,|z_{m}|e^{\iota s_{m}}\Big)=(|z_{1}|e^{\iota t_{1}},\dots,|z_{m}|e^{\iota t_{m}})
	$$
	Clearly, both maps are smooth and inverses of each other. 
\end{proof}

Then the foliation of $(\mathbb{C}^m,{}_{K}\Pi_{4}^{R})$ can be expressed using such triplets $(A,B,C)$,
$$
	F_{Ra}^{An}=\{A\times B\in F_{An}\times F_{Ra}|\quad A\cap B \not= \emptyset\}
$$
$$
F_{\mathbb{C}^m}=\{C\subset \mathbb{C}^m|\exists A\times B\in F_{Ra}^{An} \text{ s.t.} \psi:C\rightarrow A\times B\text{ is a diffeomorphism} \}
$$

Since at each point $q\in \mathcal{Z}_{\Delta}$,
$$
	{}_{K}\Pi_{4}^{R}|_{q}\in T_{q}(\mathcal{Z}_{\Delta})\wedge T_{q}(\mathcal{Z}_{\Delta})
$$
Using the similar method as above, one can separately integrate along angular and radial Hamiltonian directions while remaining inside $\mathcal{Z}_{\Delta}$ and the resultant pair of maximal submanifolds will lie inside $\mathcal{Z}_{\Delta}$.
Let $F_{An}|_{\mathcal{Z}_{\Delta}}$ and $F_{Ra}|_{\mathcal{Z}_{\Delta}}$ be the foliations produced by $(\mathcal{Z}_{\Delta},\Pi_{An}|_{\mathcal{Z}_{\Delta}})$ and $(\mathcal{Z}_{\Delta},\Pi_{Ra}|_{\mathcal{Z}_{\Delta}})$ respectively. Then,
$$
	F_{An}|_{\mathcal{Z}_{\Delta}}=\{A\in F_{An}|A\subset \mathcal{Z}_{\Delta}\}\quad F_{Ra}|_{\mathcal{Z}_{\Delta}}=\{B\in F_{Ra}|B\subset \mathcal{Z}_{\Delta}\}
$$
Utilizing similar arguments as above,the foliation of $(\mathcal{Z}_{\Delta},{}_{K}\Pi_{4}^{R}|_{\mathcal{Z}_{\Delta}})$ can be expressed as,

$$
		F_{Ra}^{An}|_{\mathcal{Z}_{\Delta}}=\{A\times B\in F_{An}|_{\mathcal{Z}_{\Delta}}\times F_{Ra}|_{\mathcal{Z}_{\Delta}}|\quad A\cap B \not= \emptyset\}
$$
$$
F_{\mathcal{Z}_{\Delta}}=\{C\subset \mathcal{Z}_{\Delta}|\exists A\times B\in F_{Ra}^{An}|_{\mathcal{Z}_{\Delta}} \text{ s.t. } \psi:C\rightarrow A\times B\text{ is a diffeomorphism} \}
$$

\section{Poisson Cohomology of $\mathbb{C}^m$}\label{P cohom Cm}
In this section, we will analyze the Poisson cohomology of the new families of Poisson structures on $\mathbb{C}^m$.
Let $z_{i}=x_{i}+\iota y_{i}$, where $\iota=\sqrt{ -1 }$ and $i=1,\dots,m$  be the standard complex coordinates of $\mathbb{C}^m$. Let $p_i=r_i^2=x_i^2+y_i^2$ and $\theta_i$, $1\leq i\leq m$, denote the radial and angular coordinates.

\begin{lemma}\label{10.1}
If $f_{i},h_{j}\in C^{\infty}(\mathbb{C}^m)$ such that $\mathcal{P}_{i}(f_{i})=0$ and $\Theta_{j}(h_{j})=0$, then
$$
	[f_{i}\Theta_{l},\mathcal{P}_{i}]=[f_{i}\mathcal{P}_{k},\mathcal{P}_{i}]=0
$$
and 
$$
	[h_{j}\Theta_{l},\Theta_{j}]=[h_{j}\mathcal{P}_{k},\Theta _{j}]=0
$$
for all $1\leq k,l\leq m$

\end{lemma}

\begin{proof}
    The proof follows from Lemma \ref{3.1} and the graded Leibniz property.

\end{proof}
\begin{corollary}\label{10.2}
$$
	[\Theta_{k}\wedge\Theta_{l},\mathcal{P}_{t}]=[\Theta_{k}\wedge\Theta_{l},\Theta_{s}]=0
$$
and 
$$
	[\mathcal{P}_{i}\wedge\mathcal{P}_{j},\Theta _{s}]=[\mathcal{P}_{i}\wedge\mathcal{P}_{j},\mathcal{P}_{t}]=0
$$
for all $1\leq i,j,k,l,s,t\leq m$

\end{corollary}

\begin{proof}
    The proof follows from Lemma \ref{3.1} and the graded Leibniz property.

\end{proof}




Suppose $X$ is a  vector field on a Poisson manifold $(P, \Pi)$. Then,
\begin{equation}\label{eq:Xpiongh}
    [X,\Pi](g,h)=X(\Pi(g,h))-\Pi(Xg,h)-\Pi(g,Xh)\quad \forall g,h \in C^\infty(P)
\end{equation}


\begin{lemma}\label{10.3}
Given a Poisson vector field $X$ on $(P, \Pi)$, and $f \in C^\infty(P)$,  $fX$ is a Poisson vector field on $(P, \Pi)$ if and only if
\begin{equation}
X(g)\Pi(f,h)+X(h)\Pi(g,f)=0 \quad \forall g,h \in C^\infty(P)
\end{equation}
\end{lemma}
\begin{proof}
    $fX$ is a Poisson vector field if and only if $[fX,\Pi]=0$. For any functions $g,h \in C^\infty(P)$, by \eqref{eq:Xpiongh} we have
$$
[fX,\Pi](g,h)=f([X,\Pi](g,h)) -X(g)\Pi(f,h) - X(h)\Pi(g,f)
$$
Since $[X,\Pi]=0$, 
$$
[fX,\Pi]=0 \iff X(g)\Pi(f,h)+X(h)\Pi(g,f)=0 \quad \forall h,g\in C^{\infty}(P) 
$$
\end{proof}

\begin{corollary}\label{10.4}
If $X$ is a Poisson vector field and $f\in C^{\infty}(P)$ is a Casimir function on $(P, \Pi)$, then $fX$ is a Poisson vector field on $(P,\Pi)$.
\end{corollary}
\begin{proof}
	If $f$ is a Casimir function then $\Pi(f,\cdot)=X_{f}=0$. Thus 
	$$
		 X(g)\Pi(f,h)+X(h)\Pi(g,f)=X(g)X_{f}(h)-X(h)X_{f}(g)=0 \quad \forall h,g \in C^{\infty}(P)
	$$
\end{proof}
\begin{theorem}\label{10.5} Let $X$ be a Poisson vector field on $(P,\Pi)$.
	If there exists a Casimir function $c_{X} \in C^\infty(P)$ such that $X(c_{X})$ is non-zero almost everywhere, then
	$$
		\text{$fX$ is Poisson} \iff \text{$f$ is a Casimir function} 
	$$
\end{theorem}
\begin{proof}
	If $f$ is a Casimir function, then by Corollary \ref{10.4}, $fX$ is a Poisson vector field on $(P,\Pi)$.\\
	Conversely, assume $fX$ is a Poisson vector field on $(P,\Pi)$. Then 
	$$
		X(g)\Pi(f,h)+X(h)\Pi(g,f)=0 \quad \forall h,g\in C^{\infty}(P)
	$$
	Let $g=c_{X}$, then  
	$$
		X(c_{X})\Pi(f,h)=0 \quad \forall h \in C^\infty(P)
	$$
Since $X(c_{X})$ is non-zero almost everywhere, $\Pi(f,h)$ is zero almost everywhere and by smoothness $\Pi(f,h)=0$ for all $h\in C^{\infty}(P)$. Thus $f$ is a Casimir function.
\end{proof}

\begin{theorem}\label{10.6}
    Let $X_i$,  $(1\leq i \leq n)$, be $n$ Poisson vector fields on $(P,\Pi)$. If there exist $n$ Casimir functions $c_{X_1},\dots,c_{X_{n}}\in C^{\infty}(P)$ such that 
    $$
		X_{i}(c_{X_{j}}) = \begin{cases} \text{Non-zero almost everywhere} & \text{if } i=j \\ 0 & \text{if } i>j \\
		  \end{cases}
	$$
    then,
    $$
      \text{$\sum_{i=1}^nf_iX_i$ is Poisson} \iff \text{Each $f_i$ is a Casimir function}   
      $$
\end{theorem}
\begin{proof}$(\impliedby)$
Let $Y=\sum_{i=1}^nf_iX_i$. Since each $X_{i}$ is a Poisson vector field on $(P,\Pi)$,
$$
	[Y,\Pi]\big(g,h\big)=\sum_{i=1}^n\Big(X_{i}(g)\Pi(f_{i},h)+X_{i}(h)\Pi(g,f_{i})\Big) \quad\forall g,h\in C^{\infty}(P)
$$
If each $f_{i}$ is a Casimir function then 
$$
X_{i}(g)\Pi(f_{i},h)+X_{i}(h)\Pi(g,f_{i})=0\quad \forall 1\leq i\leq n
$$
Thus $[Y,\Pi]=0$.\\
$(\implies)$
 If $Y$ is a Poisson vector field in $(P,\Pi)$, then,
 $$
		[Y,\Pi]\big(g,h\big)=\sum_{i=1}^n\Big(X_{i}(g)\Pi(f_{i},h)+X_{i}(h)\Pi(g,f_{i})\Big)=0 \quad\forall g,h\in C^{\infty}(P)
 $$
 
 Substituting $g$ with $c_{X_{1}}$,

$$
			[Y,\Pi]\big(c_{X_{1}},h\big)=X_{1}(c_{X_{1}})\Pi(f_{1},h)=0 \quad \forall h \in C^\infty(P)
$$

By Theorem \ref{10.5}, $f_{1}$ is a Casimir function in $(P,\Pi)$.
$$
	\implies		[Y,\Pi]\big(g,h\big)=\sum_{i=2}^n\Big(X_{i}(g)\Pi(f_{i},h)+X_{i}(h)\Pi(g,f_{i})\Big)=0 \quad\forall g,h\in C^{\infty}(P)
 $$
 
Substituting $g$ with $c_{X_{2}}$,
$$
				[Y,\Pi]\big(c_{X_{2}},h\big)=X_{2}(c_{X_{2}})\Pi(f_{2},h)=0 \quad \forall h \in C^\infty(P)
$$
By Theorem \ref{10.5}, $f_{2}$ is a Casimir function in $(P,\Pi)$.
Iterating this argument, each $f_{i}$ is a Casimir function.
    
\end{proof}

\begin{lemma}\label{10.8}
	Let $U\subset P$ be an open set such that $\operatorname{closure}(U)=P$. If $Y\in \mathfrak{X}^{k}(P)$ such that
	$$
		[\Pi|_{U},Y|_{U}]=0
	$$ then,
	$$
		\big[\Pi,Y\big]=0
	$$
\end{lemma}

\begin{theorem}\label{10.9}
 	Let $X_i$  $(1\leq i \leq n)$ be $n$ Poisson vector fields on $(P,\Pi)$. If there exist an open set $U\subset P$ and $n$ functions $c_{X_1},\dots,c_{X_{n}}\in C^{\infty}(U)$ such that $\operatorname{closure}(U)=P$ and
    $$
		  X_{i}|_{U}(c_{X_{j}}) = \begin{cases} \text{Non-zero almost everywhere} & \text{if } i=j \\ 0 & \text{if } i\neq j \end{cases}\quad  \big[\Pi|_{U},c_{X_{i}}\big]=0\quad \forall 1\leq i\leq n
	$$
    then,
    $$
      \text{$\sum_{i=1}^nf_iX_i$ is Poisson} \iff \text{Each $f_i$ is a Casimir function}   
      $$
\end{theorem}
\begin{proof}($\impliedby$)
	As shown in Theorem \ref{10.6}, if each $f_{i}$ is a Casimir function then $\sum_{i=1}^nf_iX_i$ is a Poisson vector field on $(P,\Pi)$.\\
    ($\implies$)
	Assume $Y=\sum_{i=1}^nf_iX_i$ is a Poisson vector field. Then each $X_{i}|_{U}$ and $Y|_{U}$ are Poisson vector fields on $(U,\Pi|_{U})$. By Theorem \ref{10.6} each $f_{i}|_{U}\in C^{\infty}(U)$ is a Casimir function in $(U,\Pi|_{U})$. Since for each $f_{i}\in \mathfrak{X}^{0}(P)$,
	$$
		[\Pi|_{U},f_{i}|_{U}]=0 \,.
	$$
	, Lemma \ref{10.8} gives
	$$
		[\Pi,f_{i}]=0\quad \forall 1\leq i\leq n \,.
	$$  
\end{proof}

\begin{lemma}\label{10.10}
	If $X\in \mathfrak{X}^{1}(\mathbb{C}^m)$ is a fundamental vector field produced by the action of $\mathbb{T}^m$ on $\mathbb{C}^m$, then $X$ is a Poisson vector field on $\mathbb{C}^m$ with respect to each of the  Poisson bivectors constructed in Section 3.
\end{lemma}
\begin{proof}
	As mentioned in Remark \ref{re8.2}, each $g\in \mathbb{T}^m$ acts as a Poisson diffeomorphism on $\mathbb{C}^m$ with respect to the Poisson bivectors in Section \ref{str on Cm}. Thus, as shown in the proof of Lemma \ref{2.6}, each fundamental vector field produced by $\mathbb{T}^m$ is Poisson with respect to these bivectors.
\end{proof}

\subsection{Poisson cohomology of $\Pi_{1}$}

$$
\Pi_{1}= \sum_{1\leq i<j\leq m} a_{ij}\Theta_{i}\wedge \Theta_{j}
$$

where each $a_{ij}\in C^\infty(\mathbb{C}^m)$ and $\Theta_k(a_{ij})=0$  for all $1\leq i,j,k \leq m$.
In other words, $a_{ij}$ is a function of $p_{1},\dots,p_m$ only.

\begin{lemma}\label{10.11}
	Each $p_{i}=x_{i}^{2}+y_{i}^{2}\in C^{\infty}(\mathbb{C}^m)$ is a Casimir function on $(\mathbb{C}^m,\Pi_{1})$, that is,
	$$
		\Pi_{1}(p_{i},\cdot)=0\quad \forall 1\leq i\leq m
	$$
\end{lemma}
\begin{proof}
Follows since
	$\Theta _{i}(p_{k})=0\quad \forall 1\leq i,k\leq m$.
	
\end{proof}
\begin{construction}\label{10.20}
Let 
$$
\mathbb{C}_{-} =\mathbb{C}\setminus \{z=x+\iota y\in \mathbb{C}|y=0\text{ and }x\geq 0\}
$$
Then $\operatorname{closure}(\mathbb{C}_{-})=\mathbb{C}$ and the angular coordinate function $\theta$ is well-defined on $\mathbb{C}_{-}$. 
Let 
$$
\mathbb{C}_{-}^m :=\mathbb{C}_{-}\underset{\text{m times}}{\times\dots \times}\mathbb{C}_{-}\underset{\text{open}}{\subset} \mathbb{C}^m
$$
Then $\operatorname{closure}(\mathbb{C}^m_{-})=\mathbb{C}^m$ and the angular coordinate functions $\theta_{1},\dots,\theta _m$  are well-defined on $\mathbb{C}^m_{-}$.
\end{construction}
\begin{lemma}\label{10.20'}
$\theta_{l}$ is a Casimir function on $\Big(\mathbb{C}_{-}^m,\Pi_{1}|_{\mathbb{C}_{-}^m}\Big)$ if and only if $a_{lj}=0 \quad \forall 1\le j\le m$.
\end{lemma}
\begin{proof}
$\Pi_{1}\Big{|}_{\mathbb{C}_{-}^m}(\theta_{l})=\left(-\sum_{j=1}^{l-1} {a_{jl}}\Theta_{j}+ \sum_{j=l+1}^m {a_{lj}}\Theta_{j}\right)\Big{|}_{\mathbb{C}_{-}^m}$.
Since the vector fields $\Theta_{1}\Big{|}_{\mathbb{C}_{-}^m},\dots,\Theta_{m}\Big{|}_{\mathbb{C}_{-}^m}$ are pointwise linearly independent, 
$$
	\Pi_{1}\Big{|}_{\mathbb{C}_{-}^m}(\theta_{l})=0\iff a_{lj}\Big{|}_{\mathbb{C}_{-}^m}=0 \quad \forall 1\le j\le m
$$
Since $\operatorname{closure}\mathbb{C}_{-}^m=\mathbb{C}^m$,
$$
	\Pi_{1}\Big{|}_{\mathbb{C}_{-}^m}(\theta_{l})=0\iff a_{lj}=0 \quad \forall 1\le j\le m
$$
\end{proof}

\begin{lemma}\label{10.12}
$\mathcal{P}_{k}$ is a Poisson vector field on $(\mathbb{C}^m,\Pi_{1})$ if and only if  $\mathcal{P}_{k}(a_{ij})=(x_{k}\frac{\partial}{\partial x_{k}}+y_{k}\frac{\partial}{\partial y_{k}})(a_{ij})=0$ for all $1\leq i<j\leq m$.
\end{lemma}
\begin{proof}
     By the graded Leibniz rule and Corollary \ref{10.2}, 
    $$
		\Big[\mathcal{P}_{k},\Pi_{1}\Big]=\sum_{1\leq i<j\leq m} \Big[\mathcal{P}_{k},a_{ij}\Theta_{i}\wedge \Theta _{j}\Big]=\sum_{1\leq i<j\leq m} \Big[\mathcal{P}_{k},a_{ij}\Big]\Theta_{i}\wedge \Theta _{j}
	$$
	
Since $\Theta_{i}\wedge \Theta _j$ are pointwise linearly independent on a dense open set,
$$
\Big[\mathcal{P}_{k},\Pi_{1}\Big]=0 \iff \Big[\mathcal{P}_{k},a_{ij}\Big]=0 \quad \forall1\leq i<j\leq m \,.
$$
	
\end{proof}
\begin{theorem}\label{10.13}
	If there exist $K,L \subset\{1,\dots,m\}$ such that  $\mathcal{P}_{k}(a_{ij})= 0$ and $a_{lj}=0$ for all $1\leq i<j\leq m$, $k\in K$ and $l\in L$, then
	$$
		\sum_{k\in K}f_{k}\mathcal{P}_{k}+\sum_{l\in L}h_{l}\Theta_{l}\in\operatorname{Pois}(\mathfrak{X}^1(\mathbb{C}^m),\Pi_1)\iff \text{ Each } f_k,h_l\in \operatorname{Cas}(\mathbb{C}^m,\Pi_{1}) 
	$$
\end{theorem}
\begin{proof}
	By Lemmas \ref{10.10}, \ref{10.12} and the hypotheses of the statement, each $\Theta_{l}$ and $\mathcal{P}_{k}$ is a Poisson vector field, where $l\in L$ and $k\in K$. By Construction \ref{10.20} and Lemmas \ref{10.11} and \ref{10.20'}, each $\theta_{l}$ and $p_{k}|_{\mathbb{C}_{-}^m}$ where $l\in L$ and $k\in K$ are Casimir functions on $\Big(\mathbb{C}_{-}^m,\Pi_{1}|_{\mathbb{C}_{-}^m}\Big)$ such that $\operatorname{closure}\big(\mathbb{C}_{-}^m\big)=\mathbb{C}^m$ and
$$
	\Theta_{l}|_{\mathbb{C}_{-}^m}(\theta_{l'}) = \begin{cases} 1 & \text{if } l=l' \\ 0 & \text{if } l\neq l' \end{cases} \quad \Theta_{l}|_{\mathbb{C}_{-}^m}(p_{j}|_{\mathbb{C}_{-}^m})=0
$$
$$
	\mathcal{P}_{{k}}|_{\mathbb{C}_{-}^m}(p_{k'}|_{\mathbb{C}_{-}^m}) = \begin{cases} p_{k}|_{\mathbb{C}_{-}^m} & \text{if } k=k' \\ 0 & \text{if } k\neq k' \end{cases} \quad \mathcal{P}_{k}|_{\mathbb{C}_{-}^m}(\theta _{j})=0
$$
Thus, by Theorem \ref{10.9}, the statement holds true.
\end{proof}

\begin{theorem}
	Under the same hypotheses as Theorem \ref{10.13} with $K=\{k_{1},\dots,k_{s}\}$ and $L=\{l_{1},\dots,l_{t}\}$, the map
	$$
		\bigoplus_{i=1}^{t+s} \mathrm{Cas}(\mathbb{C}^m,\Pi_{1})\to H^1_{\Pi_{1}}\quad \quad (f_{1},\dots,f_{s},f'_{1},\dots,f'_{t})\longmapsto \Big[\sum_{i=1}^{s}f_{i}\mathcal{P}_{k_{i}}+\sum_{j=1}^{t} f_{j}'\Theta_{l_{j}} \Big]
	$$
	is injective.
	
\end{theorem}
\begin{proof}
Let $(f_{1},\dots,f_{s},f'_{1}\dots,f'_{t}),(h_{1},\dots,h_{s},h'_{1}\dots,h'_{t})\in\bigoplus_{i=1}^{s+t} \mathrm{Cas}(\mathbb{C}^m,\Pi_{1})$ such that,
$$
\Big[\sum_{i=1}^{s}(f_{i}-h_{i})\mathcal{P}_{k_{i}}+\sum_{j=1}^{t} (f'_{j}-h'_{j})\Theta_{l_{j}}  \Big]=0
$$
By the hypotheses of the theorem, $a_{lj}=0$ for all $l\in L$ and $1\leq j\leq m$. As a result, each Hamiltonian vector field in $(\mathbb{C}^m,\Pi_{1})$ is of the form,
	$$
		\sum_{\substack{1\leq j\leq m\\j\not\in L}}\alpha_{j}\Theta_{j}
	$$
	where $\alpha_{{j}}\in C^{\infty}(\mathbb{C}^m)$. Thus, by assumption,
$$
	\sum_{i=1}^{s}(f_{i}-h_{i})\mathcal{P}_{k_{i}}+\sum_{j=1}^{t} (f'_{j}-h'_{j})\Theta_{l_{j}}=\sum_{\substack{1\leq j\leq m\\j\not\in L}}\beta_{j}\Theta_{j}
$$	for some $\beta_j\in C^\infty(\mathbb{C}^m)$. Since the vector fields $\mathcal{P}_{k_{1}},\dots,\mathcal{P}_{k_{s}},\Theta_{1},\dots,\Theta_{m}$ are pointwise linearly independent almost everywhere, 
	$$
		(f_{1}-h_{1},\dots,f_{s}-h_{s},f'_{1}-h'_{1},\dots,f'_{t}-h'_{t})=0
	$$
	$$
	\implies	(f_{1},\dots,f_{s},f'_{1},\dots,f'_{t})=(h_{1},\dots,h_{s},h'_{1},\dots,h'_{t})
	$$
\end{proof}

\subsection{Poisson cohomology of $\Pi_{3}$ and $\Pi^R_3$}

$$
\Pi_{3}=\sum_{1\leq i<j\leq m}c_{ij}\mathcal{P}_{i}\wedge \mathcal{P}_{j}
$$

where $c_{ij}$ are constant functions.
$$
\Pi^R_3= \sum_{1\leq k<j\leq 3} c^R_{i_ji_k}\mathcal{P}_{i_{j}}\wedge \mathcal{P}_{i_{k}}
$$

$$
c^R_{i_ji_k}=c^R_{i_ji_k}(x_{l}^2+y_{l}^2)\in C^{\infty}(\mathbb{C}^m)
$$
where $l\in\{1,\dots,m\}$ and $l\neq i_j,i_k$. In other words, $\mathcal{P}_{i_{j}}(c^R_{i_ji_k})=\mathcal{P}_{i_{k}}(c^R_{i_ji_k})=0$ for each pair $(k,j)\in \{(1,2),(2,3),(1,3)\}$.

\begin{lemma}\label{lemma:casimir_merged} 
The following properties hold for Casimir functions of the Poisson structures $\Pi_3$ and $\Pi_3^R$: 
\begin{enumerate}
\item $p_l = x_l^2 + y_l^2\in\operatorname{Cas}(\mathbb{C}^m, \Pi_3)$ if and only if $c_{lj} = 0 \quad \forall 1 \le j \le m$. 
\item For each $l \in \{1, \dots, m\} \setminus \{i_1, i_2, i_3\}$, the coordinate functions $x_l,y_l\in\operatorname{Cas}(\mathbb{C}^m, \Pi_3^R)$. Consequently, for each $l \in \{1, \dots, m\} \setminus \{i_1, i_2, i_3\}$, $p_l\in\operatorname{Cas}(\mathbb{C}^m, \Pi_3^R)$.
\item The angular coordinate functions $\theta_1, \dots, \theta_m \in C^\infty(\mathbb{C}_{-}^m)$ are Casimir functions on both $\left(\mathbb{C}_{-}^m, \Pi_3|_{\mathbb{C}_{-}^m}\right)$ and $\left(\mathbb{C}_{-}^m, \Pi_3^R|_{\mathbb{C}_{-}^m}\right)$. \end{enumerate} \end{lemma}

\begin{lemma}\label{lemma:poisson_vector_merged}
For each vector field $\mathcal{P}_l$ such that $1 \le l \le m$, the following hold:
\begin{enumerate}
    \item $\mathcal{P}_l\in \operatorname{Pois}(\mathfrak{X}^1(\mathbb{C}^m),\Pi_3)$.
    \item $\mathcal{P}_l\in\operatorname{Pois}(\mathfrak{X}^1(\mathbb{C}^m), \Pi_3^R)$ if and only if $\mathcal{P}_l(c_{i_j i_k}^R) = 0$ for all $1 \le j < k \le 3$.
\end{enumerate}
\end{lemma}

\begin{remark} Below is an example of a function that is not a Casimir but still produced a Poisson vector field. Take the Poisson bivector field on $\mathbb{C}^m$,
$$
\Pi=\mathcal{P}_{1}\wedge \mathcal{P}_{2}
$$
By Lemmas \ref{lemma:poisson_vector_merged}  and \ref{lemma:casimir_merged} , $\mathcal{P}_{1}$ is a Poisson vector field and $p_{2}$ is not a Casimir function for the given Poisson bivector. Since $p_{2}$ satisfies the condition in Lemma \ref{10.3}, $p_{2}\mathcal{P}_{1}$ is a Poisson vector field.
\end{remark}
\begin{theorem}\label{10.25}
If there exist $K \subset\{1,\dots,m\}\setminus\{i_{1},i_{2},i_{3}\}$ and $L \subset\{1,\dots,m\}$ such that  $\mathcal{P}_{k}(c_{i_j i_k}^R)= 0$ and $c_{ll'}=0$ for all $i_{j},i_{k}\in\{i_{1},i_{2},i_{3}\}$, $1\leq l'\leq m$, $k\in K$ and $l\in L$, then
	$$
		\sum_{l\in L}f_{l}\mathcal{P}_{l}+\sum_{i=1}^mh_{i}\Theta_{i}\in \operatorname{Pois}(\mathfrak{X}^1(\mathbb{C}^m), \Pi_3)\iff \text{ Each } f_l,h_i\in\operatorname{Cas}(\mathbb{C}^m,\Pi_{3}) 
	$$
	$$
		\sum_{k\in K}f_{k}\mathcal{P}_{k}+\sum_{i=1}^mh_{i}\Theta_{i}\in \operatorname{Pois}(\mathfrak{X}^1(\mathbb{C}^m), \Pi_3^R)\iff \text{ Each }f_k,h_i\in\operatorname{Cas}(\mathbb{C}^m,\Pi_{3}^R) 
	$$
\end{theorem}

\begin{theorem}
	Under the same hypotheses as Theorem \ref{10.25} with $K=\{k_{1},\dots,k_{s}\}$ and $L=\{l_{1},\dots,l_{t}\}$, the following maps,
	$$
		\bigoplus_{i=1}^{m+t} \mathrm{Cas}(\mathbb{C}^m,\Pi_{3})\to H^1_{\Pi_{3}}\quad \quad (f_{1},\dots,f_{t},f'_{1}\dots,f'_{m})\longmapsto \Big[\sum_{i=1}^{t}f_{i}\mathcal{P}_{l_{i}}+\sum_{j=1}^{m} f'_{j}\Theta_{j} \Big] \,,
	$$
	$$
		\bigoplus_{i=1}^{m+s} \mathrm{Cas}(\mathbb{C}^m,\Pi_{3}^R)\to H^1_{\Pi_{3}^R}\quad \quad (f_{1},\dots,f_{s},f'_{1}\dots,f'_{m})\longmapsto \Big[\sum_{i=1}^{s}f_{i}\mathcal{P}_{k_{i}}+\sum_{j=1}^{m} f_{j}'\Theta_{j} \Big] \,,
	$$
	are injective.
	
\end{theorem}

\subsection{Poisson cohomology of $\Pi_{4}$ and $\Pi^R_4$}
$$
\Pi_{4}=\sum_{1\leq i<j\leq m} {e_{ij}}\mathcal{P}_{i}\wedge \mathcal{P}_{j}+\sum_{1\leq r<s\leq m} d_{rs}\Theta_{r}\wedge \Theta_{s}
$$

where each $e_{ij}$ and $d_{rs}$ are constant functions.

$$
\Pi^R_{4}=\sum_{1\leq k<j\leq 3} {e^R_{i_ji_k}}\mathcal{P}_{i_{j}}\wedge \mathcal{P}_{i_{k}}+\sum_{1\leq r<s\leq m} d^R_{rs}\Theta_{r}\wedge \Theta_{s}
$$

$$
e^R_{i_ji_k}=e^R_{i_ji_k}(x_{l}^2+y_{l}^2)\in C^{\infty}(\mathbb{C}^m)
$$
where $l\in\{1,\dots,m\}$ and $l\neq i_j,i_k$. In other words, $\mathcal{P}_{i_{j}}(e^R_{i_ji_k})=\mathcal{P}_{i_{k}}(e^R_{i_ji_k})=0$ for each pair $(k,j)\in \{(1,2),(2,3),(1,3)\}$ and $d^R_{rs}$ are constant functions.
\begin{lemma}\label{lemma:casimir_merged} 
The following properties hold for the Casimir functions of the Poisson structures $\Pi_3$ and $\Pi_3^R$: 
\begin{enumerate}
\item The function $p_l = x_l^2 + y_l^2\in\operatorname{Cas}(\mathbb{C}^m, \Pi_4)$ if and only if $e_{lj} = 0 \quad \forall 1 \le j \le m$. 
\item For each $l \in \{1, \dots, m\} \setminus \{i_1, i_2, i_3\}$, $p_l\in \operatorname{Cas}(\mathbb{C}^m, \Pi_4^R)$.
\item  $\theta_k \in C^\infty(\mathbb{C}_{-}^m)$ is a Casimir function on $\left(\mathbb{C}_{-}^m, \Pi_4|_{\mathbb{C}_{-}^m}\right)$ if and only if $d_{kj}=0\quad \forall 1 \le j \le m$. 
\item  $\theta_k \in C^\infty(\mathbb{C}_{-}^m)$ is a Casimir function on $\left(\mathbb{C}_{-}^m, \Pi_4^R|_{\mathbb{C}_{-}^m}\right)$ if and only if $d_{kj}^R=0\quad \forall 1 \le j \le m$.
\end{enumerate} 
\end{lemma}
\begin{lemma}\label{lemma:poisson_vector_merged}
For each vector field $\mathcal{P}_l$ such that $1 \le l \le m$, the following hold:
\begin{enumerate}
    \item $\mathcal{P}_l\in\operatorname{Pois}(\mathfrak{X}^1(\mathbb{C}^m), \Pi_4)$.
    \item $\mathcal{P}_l\in \operatorname{Pois}(\mathfrak{X}^1(\mathbb{C}^m), \Pi_4^R)$ if and only if $\mathcal{P}_l(e_{i_j i_k}^R) = 0$ for all $1 \le j < k \le 3$.
\end{enumerate}
\end{lemma}

\begin{theorem}\label{10.25}
If there exist $K^R \subset\{1,\dots,m\}\setminus\{i_{1},i_{2},i_{3}\}$ and $K,L,L^R \subset\{1,\dots,m\}$ such that
\begin{enumerate}
  \item $e_{kj}=0$ and $d_{lj}=0$ for all $k\in K$, $l\in L$ and $1\leq j\leq m$
  \item $\mathcal{P}_{k}(e_{i_j i_k}^R)= 0$ and $d^R_{lj}=0$ for all $k\in K^R$, $l\in L^R$, $1\leq j\leq m$ and $i_{j},i_{k}\in\{i_{1},i_{2},i_{3}\}$.
\end{enumerate}
Then,
	$$
		\sum_{k\in K}f_{k}\mathcal{P}_{k}+\sum_{l\in L}h_{l}\Theta_{l}\in\operatorname{Pois}(\mathfrak{X}^1(\mathbb{C}^m),\Pi_{4})\iff \text{ Each }f_k,h_l\in\operatorname{Cas}(\mathbb{C}^m,\Pi_{4}) 
	$$
	$$
       \sum_{k\in K^R}f_{k}\mathcal{P}_{k}+\sum_{l\in L^R}h_{l}\Theta_{l}\in\operatorname{Pois}(\mathfrak{X}^1(\mathbb{C}^m),\Pi_{4}^R)\iff \text{ Each }f_k,h_l\in\operatorname{Cas}(\mathbb{C}^m,\Pi_{4}^R) 
    $$
\end{theorem}

\begin{theorem}
	Under the same hypotheses as Theorem \ref{10.25} with $K=\{k_{1},\dots,k_{s}\}$, $L=\{l_{1},\dots,l_{t}\}$, $K^R=\{k'_{1},\dots,k'_{s'}\}$ and $L^R=\{l'_{1},\dots,l'_{t'}\}$, the following maps
	$$
		\bigoplus_{i=1}^{s+t} \mathrm{Cas}(\mathbb{C}^m,\Pi_{4})\to H^1_{\Pi_{4}}\quad \quad (f_{1},\dots,f_{s},f'_{1}\dots,f'_{t})\longmapsto \Big[\sum_{i=1}^{s}f_{i}\mathcal{P}_{k_{i}}+\sum_{j=1}^{t} f'_{j}\Theta_{l_{j}} \Big]
	$$
	$$
		\bigoplus_{i=1}^{s'+t'} \mathrm{Cas}(\mathbb{C}^m,\Pi_{4}^R)\to H^1_{\Pi_{4}^R}\quad \quad (f_{1},\dots,f_{s'},f'_{1}\dots,f'_{t'})\longmapsto\Big[\sum_{i=1}^{s'}f_{i}\mathcal{P}_{k'_{i}}+\sum_{j=1}^{t'} f_{j}'\Theta_{l'_{j}} \Big]
	$$
	are injective.
	
\end{theorem}
\section{Poisson Cohomology of $\mathcal{Z}_{\Delta}$}\label{P cohom of Z}
\begin{construction}\label{11.1}
Let $I=(i_1, \ldots, i_r),K=(k_1, \ldots, k_s)\subset \{1,\dots,m\}$ such that $1 \leq i_1 < i_2 \ldots < i_r \leq m$, $1\leq k_1< \ldots < k_{s}\leq m$ and $F_{I}$ and $F_{K}$ are faces of $\Delta$. Corresponding to them, let $J=(j_{1},\dots,j_{m-r})=I^c$  (complement of $I$) and $L=(l_1,\dots,l_{m-s})=K^c$ (complement of $K$) such that $1 \leq j_1 < j_2 \ldots < j_{m-r} \leq m$ , $1 \leq l_1 < \ldots < l_{m-s} \leq m$. The 4 possible cases are listed below:
\begin{itemize}
\item $I=K\iff J=L$
\item $I\subsetneq K\iff L\subsetneq J$ 
\item $K\subsetneq I\iff J\subsetneq L$
\item $I\not\subset K$ and $K\not\subset I \iff L\not\subset J$ and $J\not\subset L$
\end{itemize}

Recall Construction\ref{7.1} and \ref{7.3}. Construct a vector field $(P\zeta)^I_{u}\in \mathfrak{X}^{1}(\mathbb{C}^m)$ pointwise. At each point $z \in \mathbb{C}^m$,
$$
(P\zeta)^I_{u}|_{z}=P^I(z)\zeta^I_{u}|_{z}
$$
where $P^I=\prod_{v=1}^{m-r}p_{j_{v}}\in C^{\infty}(\mathbb{C}^m)$ and $\zeta^I_{u}=\sum^{m-r}_{l=1}\left( _{l}\nu^I_{u}\left( \frac{x_{j_{l}}}{p_{j_{l}}}\frac{\partial}{\partial x_{j_{l}}}+ \frac{y_{j_{l}}}{p_{j_{l}}} \frac{\partial}{\partial y_{j_{l}}} \right) \right)\in T_{z}(\mathbb{C}^m)$. 
\end{construction}
Clearly $(P\zeta)^I_{u}$ is smooth and well-defined at each point in $\mathbb{C}^m$;  thus $(P\zeta)^I_{u}$ is a vector field on $\mathbb{C}^m$. 
\begin{theorem}\label{11.2}
At each point $q\in\mathcal{Z}_{\Delta}$, $(P\zeta)^I_{u}|_{q} \in T_{q}(\mathcal{Z}_{\Delta})$.
\end{theorem}
\begin{proof}
Let $q \in \mathcal{Z}_{\Delta}$. By Construction \ref{6.1}, there exists $K=(k_1, \ldots, k_s)$ such that $F_{K}$ is a face of $\Delta$ and $q\in N_{K}=\pi^{-1}(i_{A}(F^o_{K}))$ (a submanifold in the stratified $\mathcal{Z}_{\Delta}$).
\textbf{Case-1}\\
If $K=I$, then by Theorem \ref{7.4}, $\zeta^I_{u}|_{q}\in T_{q}(\mathcal{Z}_{\Delta})\implies(P\zeta)^I_{u}|_{q} \in T_{q}(\mathcal{Z}_{\Delta})$.
\\
\textbf{Case-2}\\
If $I\subsetneq K\iff L\subsetneq J$, then there exists $j_{w}\in J\setminus L$ and  $p_{j_{w}}=x^{2}_{j_{w}}+y^{2}_{j_{w}}\in C^{\infty}(\mathbb{C}^m)$ such that $(p_{j_{w}})|_{N_{K}}=0$. Since $q\in N_K$, we get $P^I(q)=\prod_{v=1}^{m-r}(p_{j_{v}}(q))=p_{j_{1}}(q)\times\cdots\times p_{j_{w}}(q)\times\cdots\times p_{j_{m-r}}(q)=0\implies(P\zeta)^I_{u}|_{q}=0 \in T_{q}(\mathcal{Z}_{\Delta})$.
\\
\textbf{Case-3}\\
If $K\subsetneq I\iff J\subsetneq L$, then by Proposition \ref{7.5}, $\zeta^I_{u}|_{q}\in T_{q}(\mathcal{Z}_{\Delta})\implies(P\zeta)^I_{u}|_{q} \in T_{q}(\mathcal{Z}_{\Delta})$.
\\
\textbf{Case-4}\\
If $I\not\subseteq K$ and $K\not\subseteq I$, then $J\cap K \not= \emptyset$, so there exists $j_{w}\in J\cap K$ such that $(p_{j_{w}})|_{N_{K}}=0$. Then by the same reasoning as in Case-2, $(P\zeta)^I_{u}|_{q}=0 \in T_{q}(\mathcal{Z}_{\Delta})$.
\end{proof}

\begin{corollary}\label{11.3}
	$(P\zeta)^I_{u}|_{\mathcal{Z}_{\Delta}}$ is a well-defined vector field on  $\mathcal{Z}_{\Delta }$.
\end{corollary}
 Recall that $\mathscr{I}$ denotes the set of  tuples $I$ such that $F_I$ is a face of $\Delta$.
\begin{definition}
$$
	\mathscr{I}^{\ast}:=\{(I,u)\in \mathscr{I}\times\{1,\dots,n\}\mid 1\leq u\leq n-r\text{ where $n-r$ is the dimension of $F^o_I$}\}
$$
\end{definition}
\begin{remark}
To each pair $(I,u)\in \mathscr{I}^*$, associate the vector field $(P\zeta)^I_{u}|_{\mathcal{Z}_{\Delta}}$, which involves only radial directions and coefficients depending only on radial coordinates.
\end{remark}
\begin{lemma}\label{11.5}
	Each $\Theta_{i}|_{\mathcal{Z}_{\Delta}}$ is a well-defined vector field on  $\mathcal{Z}_{\Delta }$.
\end{lemma}
\begin{proof}
	At each $q\in \mathcal{Z}_{\Delta}$, by Proposition \ref{7.6} , $\Theta _{i}|_{q}\in T_{q}(\mathcal{Z}_{\Delta})$ for all $1\leq i\leq m$.
\end{proof}

\begin{lemma}\label{11.6}
	For each $(I,u)\in \mathscr{I}^{\ast}$ and $1\leq i\leq m$, $(P\zeta)^I_{u}|_{\mathcal{Z}_{\Delta}}$ and $\Theta_{i}|_{\mathcal{Z}_{\Delta}}$ commute, 
	$$
		[(P\zeta)^I_{u}|_{\mathcal{Z}_{\Delta}},\Theta_{i}|_{\mathcal{Z}_{\Delta}}]=0
	$$
	As a consequence, 
	$$
		[(P\zeta)^I_{u}|_{\mathcal{Z}_{\Delta}},\Theta_{i}|_{\mathcal{Z}_{\Delta}}\wedge \Theta_{j}|_{\mathcal{Z}_{\Delta}}]=0
	$$
\end{lemma}

\begin{lemma}\label{11.7}
	If $X\in \mathfrak{X}^{1}(\mathcal{Z}_{\Delta})$ is a fundamental vector field produced by the action of $\mathbb{T}^m$ on $\mathcal{Z}_{\Delta}$, then $X$ is a Poisson vector field on $\mathcal{Z}_{\Delta}$ for each of the Poisson structures $\Pi_{1}|_{\mathcal{Z}_{\Delta}}$, ${}_{K}\Pi^R_{3}|_{\mathcal{Z}_{\Delta}}$ and ${}_{K}\Pi^R_{4}|_{\mathcal{Z}_{\Delta}}$
\end{lemma}
\begin{proof}
	The proof is the same as that of Lemma \ref{10.10}.
\end{proof}

In the rest of the section, the proofs use Lemma \ref{lem: Poisson on submanifold} and the techniques used in Section \ref{P cohom Cm}. 
\subsection{Poisson Cohomology of $\Pi_{1}|_{\mathcal{Z}_{\Delta}}$}
$$
\Pi_{1}|_{\mathcal{Z}_{\Delta}}= \sum_{1\leq i<j\leq m} a_{ij}|_{\mathcal{Z}_{\Delta}}\Theta_{i}\wedge \Theta_{j}
$$
where each $a_{ij}|_{\mathcal{Z}_{\Delta}}\in C^\infty(\mathcal{Z}_{\Delta})$ and $\Theta_{k}|_{\mathcal{Z}_{\Delta}}(a_{ij})=0 \quad \forall1\leq k \leq m$. 
We have an embedding, 
$$
\mathcal{E}:\mathcal{Z}_{\Delta} \rightarrow \mathbb{C}^m
$$
defined in the commutative square in Section 5.

\begin{lemma}\label{11.8}
Each $\mathcal{E}^*(p_{i})=\mathcal{E}^*(x_{i}^{2}+y_{i}^{2})\in C^{\infty}(\mathcal{Z}_{\Delta})$ is a Casimir function on $(\mathcal{Z}_{\Delta},\Pi_{1}|_{\mathcal{Z}_{\Delta}})$.
\end{lemma}
\begin{construction}\label{11.11}
Recall from Construction \ref{7.1}, that corresponding to each $(I,u)\in \mathscr{I}^*$, we have vectors $\nu_{u}^{I}\in \operatorname{Ker}(\mathscr{E}_{I})$  and $\mu_{u}^{I}\in \operatorname{Ker}(\mathscr{E}_{\emptyset})$. Let $A=\{(I_{1},u_{1}),\dots,(I_{t},u_{t})\}\subset \mathscr{I}^*$ such that the vectors $\{\mu_{u_{1}}^{I_{1}},\dots,\mu_{u_{t}}^{I_{t}}\}\subset \operatorname{Ker}(\mathscr{E}_{\emptyset})$ are linearly independent. One can construct $\{\eta_{u_{1}}^{I_{1}},\dots,\eta_{u_{t}}^{I_{t}}\}\subset \operatorname{Ker}(\mathscr{E}_{\emptyset})$ 
$$
(\eta_{u_{i}}^{I_{i}})^T=({}_{1}\eta_{u_{i}}^{I_{i}},\dots,{}_{m}\eta_{u_{i}}^{I_{i}})
$$
such that,
$$
	\langle\eta^{I_{i}}_{u_{i}},\mu^{I_{j}}_{u_{j}}\rangle=\delta _{u_{i}u_{j}}\delta_{I_{i}I_{j}}
$$
$$
\textbf{if }k\in \bigcap_{(I,u)\in A}I\implies  {}_{k}\eta_{u}^{I}=0\quad\forall (I,u)\in A
$$
For each $(I,u)\in A$, define $\eta^I_{u}p\in C^{\infty}(\mathbb{C}^m)$ 
$$
	\eta^I_{u}p=\sum_{l=1}^{m}{}_{l}\eta_{u}^{I}p_{l}
$$

\end{construction}

\begin{definition}\label{11.18}
Define
$$
	\mathcal{Z}_{\Delta}^{-} :=\mathcal{Z}_{\Delta}\cap \mathbb{C}^m_{-}\underset{\text{open}}{\subset}\mathcal{Z}_{\Delta}
$$
 Then $\operatorname{closure}(\mathcal{Z}_{\Delta}^-)=\mathcal{Z}_{\Delta}$ and $\theta_{1},\dots,\theta _m$ are well-defined functions on $\mathcal{Z}_{\Delta}^{-}$.

\end{definition}
\begin{corollary}\label{11.12}
Under the assumptions of Construction \ref{11.11}, 
$$
	(P\zeta)^{I_i}_{u_i}|_{\mathcal{Z}_{\Delta}}(\eta_{u_j}^{I_j}p|_{\mathcal{Z}_{\Delta}}) = P^{I_i}|_{\mathcal{Z}_{\Delta}} \delta _{u_{i}u_{j}}\delta_{I_{i}I_{j}}
$$
\end{corollary}
\begin{corollary}
$\theta_{l}$ is a Casimir function on $\Big(\mathcal{Z}_{\Delta}^-,\Pi_{1}|_{\mathcal{Z}_{\Delta}^-}\Big)$ if and only if $a_{lj}|_{\mathcal{Z}_{\Delta}}=0 \quad \forall 1\le j\le m$.
\end{corollary}

\begin{lemma}\label{11.9}
The vector field $(P\zeta)^I_{u}|_{\mathcal{Z}_{\Delta}}$ on $\mathcal{Z}_{\Delta}$ associated to pair $(I,u)\in \mathscr{I}^*$ is Poisson on $(\mathcal{Z}_{\Delta},\Pi_{1}|_{\mathcal{Z}_{\Delta}})$ if and only if $(P\zeta)^I_{u}|_{\mathcal{Z}_{\Delta}}(a_{ij}|_{\mathcal{Z}_{\Delta}})=0\quad \forall 1\leq i<j\leq m$. Consequently, if each $a_{ij}$ is a constant function, then each vector field $(P\zeta)^I_{u}$ (associated to $(I,u)$ where $1\leq u\leq n-r$) is Poisson on $(\mathcal{Z}_{\Delta},\Pi_{1}|_{\mathcal{Z}_{\Delta}})$.
\end{lemma}

\begin{theorem}\label{11.13}
	Let $A=\{(I_{1},u_{1}),\dots,(I_{t},u_{t})\}\subset \mathscr{I}^*$ and $B\subseteq \{1,\dots,m\}$ such that the vectors $\{\mu_{u_{1}}^{I_{1}},\dots,\mu_{u_{t}}^{I_{t}}\}\subset \operatorname{Ker}(\mathscr{E}_{\emptyset})$ are linearly independent and for each $(I,u)\in A$ and $w\in B$, $(P\zeta)^I_{u}|_{\mathcal{Z}_{\Delta}}$ is Poisson on $(\mathcal{Z}_{\Delta},\Pi_{1}|_{\mathcal{Z}_{\Delta}})$ and $a_{wj}|_{\mathcal{Z}_{\Delta}}=0$ for all $1\leq j\leq m$. Then,
	$$
		\sum_{(I,u)\in A} f^I_{u}(P\zeta)^I_{u}|_{\mathcal{Z}_{\Delta}}+\sum_{w\in B}h_{w}\Theta_{w} |_{\mathcal{Z}_{\Delta}}\in \operatorname{Pois}(\mathfrak{X}^1(\mathcal{Z}_{\Delta}),\Pi_{1}|_{\mathcal{Z}_{\Delta}})
    $$    
    $$    
    \Big\Updownarrow 
    $$
    $$
    \text{Each }f^I_{u} , h_{w}\in\operatorname{Cas}(\mathcal{Z}_{\Delta},\Pi_{1}|_{\mathcal{Z}_{\Delta}})
	$$
    
\end{theorem}
\begin{proof}
The statement follows from Theorem \ref{10.9} using $\{\eta_{u}^{I}p|_{\mathcal{Z}_{\Delta}}\}_{(I,u)\in A}$ and $\{\theta_{w} |_{\mathcal{Z}_{\Delta}^-}\}_{w\in B}$ as Casimir functions and the fact that
$$
	(P\zeta)^{I_i}_{u_i}|_{\mathcal{Z}_{\Delta}}(\eta_{u_j}^{I_j}p|_{\mathcal{Z}_{\Delta}}) = P^{I_i}|_{\mathcal{Z}_{\Delta}} \delta _{u_{i}u_{j}}\delta_{I_{i}I_{j}} \quad (P\zeta)^I_{u}|_{\mathcal{Z}_{\Delta}^-}(\theta _{j})=0
$$
$$
	\Theta_{i}|_{\mathcal{Z}_{\Delta}^-}(\theta_{j}|_{\mathcal{Z}_{\Delta}^-}) = \begin{cases} 1 & \text{if } i=j \\ 0 & \text{if } i\neq j \end{cases} \quad \Theta_{i}|_{\mathcal{Z}_{\Delta}^-}(p_{j}|_{\mathcal{Z}_{\Delta}^-})=0
$$
\end{proof}

\begin{theorem}\label{11.14}
	Each Poisson vector field of the form,
	$$
   \sum_{(I,u)\in \mathscr{I}^*}  f^I_{u}(P\zeta)^I_{u}|_{\mathcal{Z}_{\Delta}}+\sum_{w\in B}h_{w}\Theta_{w} |_{\mathcal{Z}_{\Delta}}
$$
	where $f^I_{u}\in C^\infty(\mathcal{Z}_{\Delta})$, $h_{w}\in \operatorname{Cas}(\mathcal{Z}_{\Delta},\Pi_{1}|_{\mathcal{Z}_{\Delta}})$ and for each $w\in B$, $a_{wj}|_{\mathcal{Z}_{\Delta}}=0$ for all $1\leq j\leq m$, represents distinct classes in $H^1_{\Pi_1|_{\mathcal{Z}_{\Delta}}}(\mathcal{Z}_{\Delta})$.
\end{theorem}

\subsection{Poisson Cohomology of $_K\Pi^R_{3}|_{\mathcal{Z}_{\Delta}}$}
$$
_K\Pi^R_{3}|_{\mathcal{Z}_{\Delta}}=\sum_{1\leq i<j\leq (n-q)}\Big(P^K\sigma^K_{ij}\Big)|
_{\mathcal{Z}_{\Delta}}\zeta^K_{i}|_{\mathcal{Z}_{\Delta}}\wedge \zeta^K_{j}|_{\mathcal{Z}_{\Delta}}
$$
where $\sigma^K_{ij}$ are constant functions, $K=(k_{1},\dots,k_{q})$ and $L=(l_{1},\dots,l_{m-q})=K^c$.

\begin{definition}
$$
\mathscr{K}=\{I\in\mathscr{I}|I^c\cap K^c=J\cap L\neq\emptyset\}
$$
Then $\mathscr{I}\setminus\mathscr{K}\subset\mathscr{I}$ is the set of tuples $I=J^c$ such that $I^c\cap K^c=J\cap L=\emptyset$. Define
$$
	(\mathscr{I}\setminus\mathscr{K})^*=\{(I,u)\in \mathscr{I}^*|\quad I\in\mathscr{I}\setminus\mathscr{K}\}
$$
\end{definition}
\begin{lemma}\label{11.16}
The following properties hold for Casimir functions on $(\mathcal{Z}_{\Delta},{}_{K}\Pi^R_{3}|_{\mathcal{Z}_{\Delta}})$:
\begin{enumerate}
\item For each $k\in K$, $\mathcal{E}^*(x_{k}),\mathcal{E}^*(y_{k})\in \operatorname{Cas}(\mathcal{Z}_{\Delta},_K\Pi^R_{3}|_{\mathcal{Z}_{\Delta}})$. Consequently, for 
each $k\in K$, $\mathcal{E}^*(p_{k})\in\operatorname{Cas}(\mathcal{Z}_{\Delta},_K\Pi^R_{3}|_{\mathcal{Z}_{\Delta}})$.
\item $\theta_{1}|_{\mathcal{Z}_{\Delta}^-},\dots,\theta _m|_{\mathcal{Z}_{\Delta}^-}\in\operatorname{Cas}(\mathcal{Z}_{\Delta}^-,{}_{K}\Pi^R_{3}|_{\mathcal{Z}_{\Delta}^-})$.
\end{enumerate}
\end{lemma}
 
\begin{lemma}\label{11.21}
For each $(I,u)\in (\mathscr{I}\setminus\mathscr{K})^*$, $(P\zeta)^I_{u}|_{\mathcal{Z}_{\Delta}}\in \operatorname{Pois}(\mathfrak{X}^1(\mathcal{Z}_{\Delta}),{}_{K}\Pi_{3}^R|_{\mathcal{Z}_{\Delta}})$.
\end{lemma}

\begin{theorem}\label{11.22}
Let $A=\{(I_{1},u_{1}),\dots,(I_{t},u_{t})\}\subset (\mathscr{I}\setminus\mathscr{K})^*$ such that the vectors, $\{\mu_{u_{1}}^{I_{1}},\dots,\mu_{u_{t}}^{I_{t}}\}\subset \operatorname{Ker}(\mathscr{E}_{\emptyset})$ are linearly independent. Then
$$
		\sum_{(I,u)\in A} f^I_{u}(P\zeta)^I_{u}|_{\mathcal{Z}_{\Delta}}+\sum_{w=1}^mh_{w}\Theta_{w} |_{\mathcal{Z}_{\Delta}}\in \operatorname{Pois}(\mathfrak{X}^1(\mathcal{Z}_{\Delta}),{}_{K}\Pi_{3}^R|_{\mathcal{Z}_{\Delta}})
$$    
    $$    
    \Big\Updownarrow 
    $$
    $$
        \text{Each }f^I_{u} , h_{w}\in\operatorname{Cas}(\mathcal{Z}_{\Delta},{}_{K}\Pi_{3}^R|_{\mathcal{Z}_{\Delta}})
	$$
\end{theorem}

\begin{theorem}\label{11.24}
	Each Poisson vector field of the form,
	$$
   \sum_{(I,u)\in(\mathscr{I}\setminus\mathscr{K})^*} f^I_{u}(P\zeta)^I_{u}|_{\mathcal{Z}_{\Delta}}+	  	\sum_{j=1}^mh_{j}\Theta_{j}|_{\mathcal{Z}_{\Delta}}
	$$
	
where $f^I_{u}\in C^\infty(\mathcal{Z}_{\Delta})$ and $h_{j}\in \operatorname{Cas}(\mathcal{Z}_{\Delta},{}_{K}\Pi^R_{3}|_{\mathcal{Z}_{\Delta}})$, represents distinct classes in $H^1_{{}_{K}\Pi^R_{3}|_{\mathcal{Z}_{\Delta}}}(\mathcal{Z}_{\Delta})$.
\end{theorem}

\subsection{Poisson Cohomology of $_K\Pi^R_{4}$}
$$
_K\Pi^R_{4}=\sum_{1\leq i<j\leq (n-q)}(P^K\upsilon^K_{ij})|_{\mathcal{Z}_{\Delta}}\zeta^K_{i}|_{\mathcal{Z}_{\Delta}}\wedge \zeta^K_{j}|_{\mathcal{Z}_{\Delta}}+\sum_{1\leq i<j\leq m} \tau_{ij}\Theta_{i}|_{\mathcal{Z}_{\Delta}}\wedge \Theta_{j}|_{\mathcal{Z}_{\Delta}}
$$
where $\upsilon^K_{ij}$ and $\tau_{ij}$ are constant functions, $K=(k_{1},\dots,k_{q})$ and $L=(l_{1},\dots,l_{m-q})=K^c$.

\begin{lemma}\label{11.16}
The following properties hold regarding the Casimir functions on $(\mathcal{Z}_{\Delta},{}_{K}\Pi^R_{4}|_{\mathcal{Z}_{\Delta}})$:
\begin{enumerate}
\item For each $k\in K$, $\mathcal{E}^*(p_{k})\in \operatorname{Cas}(\mathcal{Z}_{\Delta},_K\Pi^R_{4}|_{\mathcal{Z}_{\Delta}})$.
\item $\theta_{l}|_{\mathcal{Z}_{\Delta}^-}\in \operatorname{Cas}\Big(\mathcal{Z}_{\Delta}^-,{}_{K}\Pi^R_{4}|_{\mathcal{Z}_{\Delta}^-}\Big)$ if and only if $\tau_{lj}=0 \quad \forall 1\le j\le m$.
\end{enumerate}
\end{lemma}

\begin{lemma}\label{11.27}
For each $(I,u)\in (\mathscr{I}\setminus\mathscr{K})^*$, $(P\zeta)^I_{u}|_{\mathcal{Z}_{\Delta}}\in \operatorname{Pois}(\mathfrak{X}^1(\mathcal{Z}_{\Delta}),{}_{K}\Pi_{4}^R|_{\mathcal{Z}_{\Delta}})$.
\end{lemma}

\begin{theorem}\label{11.28}
Let $A=\{(I_{1},u_{1}),\dots,(I_{t},u_{t})\}\subset (\mathscr{I}\setminus\mathscr{K})^*$ and $B\subseteq \{1,\dots,m\}$ such that the vectors, $\{\mu_{u_{1}}^{I_{1}},\dots,\mu_{u_{t}}^{I_{t}}\}\subset \operatorname{Ker}(\mathscr{E}_{\emptyset})$ are linearly independent and for each $w\in B$, $\tau_{wj}=0$ for all $1\leq j\leq m$. Then,
$$
		\sum_{(I,u)\in A} f^I_{u}(P\zeta)^I_{u}|_{\mathcal{Z}_{\Delta}}+\sum_{w\in B}h_{w}\Theta_{w} |_{\mathcal{Z}_{\Delta}}\in \operatorname{Pois}(\mathfrak{X}^1(\mathcal{Z}_{\Delta}),{}_{K}\Pi_{4}^R|_{\mathcal{Z}_{\Delta}})
        $$    
    $$    
    \Big\Updownarrow 
    $$
    $$
        \text{Each }f^I_{u} , h_{w}\in\operatorname{Cas}(\mathcal{Z}_{\Delta},{}_{K}\Pi_{4}^R|_{\mathcal{Z}_{\Delta}})
	$$

\end{theorem}

\begin{theorem}\label{11.30}
	Under the assumptions of Theorem \ref{11.28}, each Poisson vector field of the form,
	$$
   \sum_{(I,u)\in (\mathscr{I}\setminus\mathscr{K})^*} f^I_{u}(P\zeta)^I_{u}|_{\mathcal{Z}_{\Delta}}+	  	\sum_{w\in B}h_{w}\Theta_{w}|_{\mathcal{Z}_{\Delta}}
	$$
	
where $f^I_{u},\in C^\infty(\mathcal{Z}_{\Delta})$, $h_{w}\in \operatorname{Cas}(\mathcal{Z}_{\Delta},{}_{K}\Pi^R_{4}|_{\mathcal{Z}_{\Delta}})$ and for each $w\in B$, $\tau_{wj}=0$ for all $1\leq j\leq m$, represents distinct classes in $H^1_{{}_{K}\Pi^R_{4}|_{\mathcal{Z}_{\Delta}}}(\mathcal{Z}_{\Delta})$.
\end{theorem}

\section{Poisson Cohomology of $\mathcal{O}$}\label{P cohom of quotients}

So far, we have discussed the Poisson cohomology of different Poisson structures on $\mathbb{C}^m$ and $\mathcal{Z}_{\Delta}$.  The reduced Poisson structures on $\mathcal{Z}_{\Delta}$ are obtained from the reducible Poisson structures on $\mathbb{C}^m$ with the action of the trivial group,  under Marsden-Ratiu reduction theory. This section focuses on studying Poisson cohomology by taking non-trivial $G=\mathbb{T}^d_{D}$ acting freely on $\mathcal{Z}_{\Delta}$ and orbit space,
$$
\mathcal{O}=\mathcal{Z}_{\Delta}/\mathbb{T}^d_{D}
$$
Let $E^d_{D}$ denote the distribution produced by $\mathbb{T}^d_{D}$. Denote the reduced Poisson structures induced by $\Pi_{1}$, $_K\Pi^R_{3}$ and $_K\Pi^R_{4}$ on $\mathcal{O}$ as $\overline{\Pi_{1}}$, ${}_{K}\overline{\Pi^R_{3}}$ and ${}_{K}\overline{\Pi^R_{4}}$ respectively.

\subsection{Pushforward map on $\mathcal{O}$}
Let
$$
\mathcal{Q}:\mathcal{Z}_{\Delta} \rightarrow \mathcal{O}\,, \quad \mathcal{Q}(x)=[x] \,
$$
be the quotient map from $\mathcal{Z}_{\Delta}$ to the orbit space $\mathcal{O}$.
\begin{definition}\label{12.1}
A smooth function $F\in C^{\infty}(\mathcal{Z}_{\Delta})$ is $G$-invariant if 
$$
	F(gx)=F(x)\quad \forall g\in G,x \in \mathcal{Z}_{\Delta}
$$
Denote the subalgebra of such functions by $C^{\infty}(\mathcal{Z}_{\Delta})^G$.
\end{definition}
\begin{remark} If the smooth functions are viewed as $0$-th degree multivector fields on the manifold then the subalgebra of $G$-invariant smooth functions can also be denoted as $\mathfrak{X}^{0}(\mathcal{Z}_{\Delta})^G$.
\end{remark}

Let $F\in C^{\infty}(\mathcal{Z}_{\Delta})^G$. 
Define the pushforward function
$f = \mathcal{Q}_{*}(F)\in C^{\infty}(\mathcal{O})$   by
$$
	f\Big([x]\Big)=F(x)\quad \forall [x]\in \mathcal{O}
$$

\begin{lemma}\label{12.2}
	The pushforward map $\mathcal{Q}_{*}$ 
	$$
		\mathcal{Q}_{*}:\mathfrak{X}^0(\mathcal{Z}_{\Delta})^G\rightarrow \mathfrak{X}^0(\mathcal{O})\,, \quad \mathcal{Q}_{*}(F)=f \,,
	$$
	is a ring isomorphism, with the pullback map,
	$$
\mathcal{Q}^*: C^{\infty}(\mathcal{O})\rightarrow C^{\infty}(\mathcal{Z}_{\Delta})^G \,, \quad \mathcal{Q}^*(f)(x) = f ([x]) \,,
	$$
	as the inverse map.
\end{lemma}

The idea of pushing forward $G$-invariant smooth functions
can be extended to pushing forward $G$-invariant multivector fields.
\begin{definition}\label{12.4}
	A multivector field $X\in \mathfrak{X}^{k}(\mathcal{Z}_{\Delta})$ is $G$-invariant if for each $g\in G$,
	$$
		dg(X)=X\circ g, \quad {\text{i.e.},} \quad dg_{x}(X_{x})=X_{gx} \quad  \forall x\in \mathcal{Z}_{\Delta}
	$$
	Denote the set of $G$-invariant multivector fields by $\mathfrak{X}^{k}(\mathcal{Z}_{\Delta})^G$
\end{definition}

\begin{lemma}\label{12.5}
	If $X\in \mathfrak{X}^{k}(\mathcal{Z}_{\Delta})$ is $G$-invariant then for arbitrary $F_{1},\dots,F_{k}\in \mathfrak{X}^0(\mathcal{Z}_{\Delta})^G$, $X(F_{1},\dots,F_{k})\in \mathfrak{X}^0(\mathcal{Z}_{\Delta})$ is $G$-invariant.
\end{lemma}

Define the pushforward map,
$$
	\mathcal{Q}_{*}:\mathfrak{X}^{k}(\mathcal{Z}_{\Delta})^G\rightarrow \mathfrak{X}^{k}(\mathcal{O})
$$
where for any $X\in \mathfrak{X}^{k}(\mathcal{Z}_{\Delta})^G$, $\mathcal{Q}_{*}(X)$ is defined as,
$$
	\mathcal{Q}_{*}(X):C^{\infty}(\mathcal{O})\times\dots \times C^{\infty}(\mathcal{O})\rightarrow C^{\infty}(\mathcal{O})
$$
$$
	\mathcal{Q}_{*}(X)\Big(f_{1},\dots,f_{k}\Big)=\mathcal{Q}_{*}\Big(X(F_{1},\dots,F_{k})\Big)
$$
where each $F_{i}=\mathcal{Q}^*(f_{i})$. $\mathcal{Q}_{*}(X)$ can also be defined pointwise by simply pushing forward the tangent vectors by $\mathcal{Q}_{*}$.
\begin{remark} The pushforward map,
$$
\mathcal{Q}_{*}:\mathfrak{X}^{k}(\mathcal{Z}_{\Delta})^G\rightarrow \mathfrak{X}^{k}(\mathcal{O})
$$
is surjective. The proof uses a principal connection on $\mathcal{Z}_{\Delta}$, considered as a principal 
$G$-bundle over $\mathcal{O}$, to lift the multivector fields on $\mathcal{O}$ to horizontal $G$-invariant multivector fields on $\mathcal{Z}_{\Delta}$.
\end{remark}
\begin{remark} By Theorem \ref{8.1} and Lemma \ref{11.7}, $\Pi_{1}|_{\mathcal{Z}_{\Delta}},{}_{K}\Pi^R_{3}|_{\mathcal{Z}_{\Delta}}$ and ${}_{K}\Pi^R_{4}|_{\mathcal{Z}_{\Delta}}$ are $G$-invariant bivectors lying in $\mathfrak{X}^{2}(\mathcal{Z}_{\Delta})^G$. The reduced Poisson structures $\overline{\Pi_{1}},{}_{K}\overline{\Pi^R_{3}}$ and ${}_{K}\overline{\Pi^R_{4}}$ constructed using Theorem \ref{2.5} can also be obtained by pushing forward $\Pi_{1}|_{\mathcal{Z}_{\Delta}},{}_{K}\Pi^R_{3}|_{\mathcal{Z}_{\Delta}}$ and ${}_{K}\Pi^R_{4}|_{\mathcal{Z}_{\Delta}}$, respectively, as shown in Lemma \ref{2.8}.
$$
	\overline{\Pi_{1}}=\mathcal{Q}_{*}\Big(\Pi_{1}|_{\mathcal{Z}_{\Delta}}\Big)\quad {}_{K}\overline{\Pi^R_{3}}=\mathcal{Q}_{*}\Big({}_{K}\Pi^R_{3}|_{\mathcal{Z}_{\Delta}}\Big)\quad {}_{K}\overline{\Pi^R_{4}}=\mathcal{Q}_{*}\Big({}_{K}\Pi^R_{4}|_{\mathcal{Z}_{\Delta}}\Big)
$$
\end{remark}

\begin{lemma}\label{12.6}
	If $X\in \mathfrak{X}^{k}(\mathcal{Z}_{\Delta})^G$ and $Y\in \mathfrak{X}^{l}(\mathcal{Z}_{\Delta})^G$ are $G$-invariant multivector fields, then $[X,Y]\in \mathfrak{X}^{k+l-1}(\mathcal{Z}_{\Delta})^G$ and 
	$$
		\mathcal{Q}_{*}[X,Y]=[\mathcal{Q}_{*}(X),\mathcal{Q}_{*}(Y)] \,.
	$$
\end{lemma}

Let $d{\kappa}$ be the normalized Haar measure on $G$.
$$
		\int_{G}d{\kappa}(g)=1 \,.
$$
It defines the averaging map,
$$
	\mathcal{A}_{G}:\mathfrak{X}^{k}(\mathcal{Z}_{\Delta})\rightarrow \mathfrak{X}^{k}(\mathcal{Z}_{\Delta})^G
$$
$$
\mathcal{A}_{G}(Y)=\int_{G}dg(Y)d{\kappa}(g)\,,\quad \mathcal{A}_{G}(Y)|_{x}=\int_{G}dg(Y|_{g ^{-1}x})d{\kappa}(g)\,,
$$
such that $\mathcal{A}_{G}\circ \mathcal{A}_{G}=\mathcal{A}_{G}$. Consequently, if $X\in \mathfrak{X}^{k}(\mathcal{Z}_{\Delta})^G$, then
	$$
		\mathcal{A}_{G}(X)=X\,.
	$$ 
In general, $\mathcal{A}_{G}\Big([X,Y]\Big)\not= \Big[\mathcal{A}_{G}(X),\mathcal{A}_{G}(Y)\Big]$; however, the following holds.
\begin{lemma}\label{12.9}
	If $X\in \mathfrak{X}^{k}(\mathcal{Z}_{\Delta})^G$ and $Y\in \mathfrak{X}^{l}(\mathcal{Z}_{\Delta})$, then
	$$
	\mathcal{A}_{G}\Big([X,Y]\Big)= \Big[\mathcal{A}_{G}(X),\mathcal{A}_{G}(Y)\Big] =\Big[X,\mathcal{A}_{G}(Y)\Big]
	$$
\end{lemma}

\subsection{Torus action on $\mathcal{O}$}

Since  $\mathbb{T}^m$ is abelian, its action on $\mathcal{Z}_{\Delta}$ induces an action of  $\mathbb{T}^{m-d}_{\mathcal{O}} := \mathbb{T}^m/\mathbb{T}^d_{D}$ 
 on  $\mathcal{O}=\mathcal{Z}_{\Delta}/\mathbb{T}^d_{D}$. 
\begin{lemma}\label{12.11}
If $f\in C^\infty(\mathcal{O})$ and $F=\mathcal{Q}^*(f)\in C^\infty(\mathcal{Z}_{\Delta})$, then for each $[g]\in \mathbb{T}^{m-d}_{\mathcal{O}}$,  $F\circ g =\mathcal{Q}^*(f\circ[g])$.
\end{lemma}
\begin{proof}
 Let $x,y \in\mathcal{Z}_{\Delta}$ represent the same class $[x]$ in $\mathcal{O}$. So there exists $g'''\in \mathbb{T}^d_D$ with $x=g'''(y)$. To prove the statement of the lemma, it suffices to show, 
 $$
	\mathcal{Q}_{*}(F\circ g)([x])=f\circ [g]([x])\quad \forall [x]\in \mathcal{O}
 $$
 equivalently,
 $$
F\circ g(x)=f\circ [g]([x])\quad \forall x \in \mathcal{Z}_{\Delta}
$$
Since $F=\mathcal{Q}^*(f)$, 
$$
	 F(x)=f([x])\quad \forall x\in \mathcal{Z}_{\Delta}
$$
$$
\implies F\circ g(x)=F(g(x))=f([g(x)])=f\circ [g]([x]) \quad \forall x \in \mathcal{Z}_{\Delta}
$$

\end{proof}

\begin{theorem}\label{12.12}
Each $[g]\in \mathbb{T}^{m-d}_{\mathcal{O}}$ is a Poisson diffeomorphism on all three Poisson manifolds $(\mathcal{O},\overline{\Pi_{1}})$, $(\mathcal{O},{}_{K}\overline{\Pi^R_{3}})$ and $(\mathcal{O},{}_{K}\overline{\Pi^R_{4}})$.

\end{theorem}
\begin{proof}
 Let $f,h\in C^\infty(\mathcal{O})$, $[g]\in \mathbb{T}^{m-d}_{\mathcal{O}}$. To prove the statement, it is sufficient to show that for each $[g]\in \mathbb{T}^{m-d}_{\mathcal{O}}$
$$
\overline{\Pi_{1}}(f,h)\circ[g]=\overline{\Pi_{1}}(f\circ[g],h\circ[g])
$$
$$
	{}_{K}\overline{\Pi^R_{3}}(f,h)\circ[g]={}_{K}\overline{\Pi^R_{3}}(f\circ[g],h\circ[g])
$$
$$
	{}_{K}\overline{\Pi^R_{4}}(f,h)\circ[g]={}_{K}\overline{\Pi^R_{4}}(f\circ[g],h\circ[g])
$$

Let $F=\mathcal{Q}^*(f)$ and $H=\mathcal{Q}^*(h)$. By Lemma \ref{2.8},
$$
	\overline{\Pi_{1}}(f,h)([x])=\Pi_{1}|_{\mathcal{Z}_{\Delta}}(F,H)(x) \quad \forall x\in \mathcal{Z}_{\Delta}
$$
$$
	\implies\overline{\Pi_{1}}(f,h)\Big([g]([x])\Big)=\Pi_{1}|_{\mathcal{Z}_{\Delta}}(F,H)\Big(g(x)\Big)\quad \forall g\in \mathbb{T}^m,x\in \mathcal{Z}_{\Delta}
$$
$$
\implies\overline{\Pi_{1}}(f,h)\circ[g]([x])=\Pi_{1}|_{\mathcal{Z}_{\Delta}}(F,H)\circ g(x)
$$
Since each $g\in \mathbb{T}^m$ is a Poisson diffeomorphism on $(\mathcal{Z}_{\Delta},\Pi_{1}|_{\mathcal{Z}_{\Delta}})$,
$$
\Pi_{1}|_{\mathcal{Z}_{\Delta}}(F,H)\circ g(x)=\Pi_{1}|_{\mathcal{Z}_{\Delta}}(F\circ g,H\circ g)(x)\quad \forall g\in \mathbb{T}^m,x\in \mathcal{Z}_{\Delta}
$$
By Lemma \ref{12.11}, $F\circ g=\mathcal{Q}^*(f\circ[g])$ and $H\circ g=\mathcal{Q}^*(h\circ[g])$. Thus by Lemma \ref{2.8},
$$
\Pi_{1}|_{\mathcal{Z}_{\Delta}}(F\circ g,H\circ g)(x)=\overline{\Pi_{1}}(f\circ[g],h\circ[g])([x]) \quad \forall g\in \mathbb{T}^m,x\in \mathcal{Z}_{\Delta}
$$
$$
	\implies\overline{\Pi_{1}}(f,h)\circ[g]([x])=\overline{\Pi_{1}}(f\circ[g],h\circ[g])([x]) \quad \forall [g]\in \mathbb{T}^{m-d}_{\mathcal{O}},[x]\in \mathcal{O}
$$
In case of $(\mathcal{O},{}_{K}\overline{\Pi^R_{3}})$ and $(\mathcal{O},{}_{K}\overline{\Pi^R_{4}})$, similar arguments apply. 
\end{proof}
\begin{corollary}\label{12.13}
	The fundamental vector fields generated by the action of $ \mathbb{T}^{m-d}_{\mathcal{O}}$ on $\mathcal{O}$ are Poisson in $(\mathcal{O},\overline{\Pi_{1}})$, $(\mathcal{O},{}_{K}\overline{\Pi^R_{3}})$ and $(\mathcal{O},{}_{K}\overline{\Pi^R_{4}})$.
\end{corollary}

\subsection{Pushing forward Poisson and Hamiltonian vector fields}

The pushforward and averaging maps induce a map,
$$
	\mathcal{Q}_{*}\circ \mathcal{A}_{G}:\mathfrak{X}^{k}(\mathcal{Z}_{\Delta})\rightarrow \mathfrak{X}^{k}(\mathcal{O}) \,.
$$

This section focuses on pushing forward Poisson and Hamiltonian multivector fields in $\mathcal{Z}_{\Delta}$ to Poisson and Hamiltonian multivector fields in $\mathcal{O}$ using $\mathcal{Q}_{*}\circ\mathcal{A}_{G}$. As a result, one can construct a map from Poisson cohomology classes in $\mathcal{Z}_{\Delta}$ to Poisson cohomology classes on $\mathcal{O}$. Suppose that $(\mathcal{Z}_{\Delta},\Pi|_{\mathcal{Z}_{\Delta}})$ and $(\mathcal{O},\overline{\Pi})$ are Poisson manifolds such that $\Pi|_{\mathcal{Z}_{\Delta}}$ is $G$-invariant and
$$
	\mathcal{Q}_{*}(\Pi|_{\mathcal{Z}_{\Delta}})=\overline{\Pi}\,.
$$
\begin{definition}\label{12.14} We say that
	$X\in \mathfrak{X}^{k}(\mathcal{Z}_{\Delta})$ is a Poisson multivector field on $(\mathcal{Z}_{\Delta},\Pi|_{\mathcal{Z}_{\Delta}})$ if
	$$
		[\Pi|_{\mathcal{Z}_{\Delta}},X]=0 \,.
	$$
	Denote the set of Poisson multivector fields of degree $k$ by $\operatorname{Pois}(\mathfrak{X}^{k}(\mathcal{Z}_{\Delta}),\Pi|_{\mathcal{Z}_{\Delta}})$ and the set of $G$-invariant Poisson multivector fields of degree $k$ by $\operatorname{Pois}(\mathfrak{X}^{k}(\mathcal{Z}_{\Delta})^G,\Pi|_{\mathcal{Z}_{\Delta}})$.
\end{definition}
\begin{definition}\label{12.15} We say that
	$X\in \mathfrak{X}^{k}(\mathcal{Z}_{\Delta})$ is a Hamiltonian multivector field on $(\mathcal{Z}_{\Delta},\Pi|_{\mathcal{Z}_{\Delta}})$ if there exists $Z\in \mathfrak{X}^{k-1}(\mathcal{Z}_{\Delta})$ such that
	$$
		[\Pi|_{\mathcal{Z}_{\Delta}},Z]=X \,.
	$$
	Denote the set of Hamiltonian multivector fields of degree $k$ by $\operatorname{Ham}(\mathfrak{X}^{k}(\mathcal{Z}_{\Delta}),\Pi|_{\mathcal{Z}_{\Delta}})$  and the set of $G$-invariant Hamiltonian multivector fields of degree $k$ by $\operatorname{Ham}(\mathfrak{X}^{k}(\mathcal{Z}_{\Delta})^G,$ $\Pi|_{\mathcal{Z}_{\Delta}})$.
\end{definition}
Similar definitions and notations are assumed for Poisson and Hamiltonian multivector fields on $(\mathcal{O},\overline{\Pi})$. Denote the differential operators induced by $\Pi|_{\mathcal{Z}_{\Delta}}$ and $\overline{\Pi}$ by $d_{\mathcal{Z}_{\Delta}}$ and $d_{\mathcal{O}}$ respectively.
\begin{remark} 
Note that 
$$
	\operatorname{Pois}(\mathfrak{X}^{k}(\mathcal{Z}_{\Delta}),\Pi|_{\mathcal{Z}_{\Delta}})=\operatorname{Ker}\Big(d_{\mathcal{Z}_{\Delta}}:\mathfrak{X}^{k}(\mathcal{Z}_{\Delta})\rightarrow \mathfrak{X}^{k+1}(\mathcal{Z}_{\Delta})\Big)
$$
$$
	\operatorname{Ham}(\mathfrak{X}^{k}(\mathcal{Z}_{\Delta}),\Pi|_{\mathcal{Z}_{\Delta}})=\operatorname{Im}\Big(d_{\mathcal{Z}_{\Delta}}:\mathfrak{X}^{k-1}(\mathcal{Z}_{\Delta})\rightarrow \mathfrak{X}^{k}(\mathcal{Z}_{\Delta})\Big)
$$
Similar statements hold for $\operatorname{Pois}(\mathfrak{X}^{k}(\mathcal{O}),\overline{\Pi})$ and $\operatorname{Ham}(\mathfrak{X}^{k}(\mathcal{O}),\overline{\Pi})$ with $d_{\mathcal{Z}_{\Delta}}$ replaced by $d_{\mathcal{O}}$.
\end{remark}
\begin{lemma}\label{12.16}
	The differential operator $d_{\mathcal{Z}_{\Delta}}$ and the averaging operator $\mathcal{A}_{G}$ commute,
$$
	d_{\mathcal{Z}_{\Delta}}\circ\mathcal{A}_{G}=\mathcal{A}_{G}\circ d_{\mathcal{Z}_{\Delta}}
$$
\end{lemma}
\begin{proof}
	Let $X\in \mathfrak{X}^{k}(\mathcal{Z}_{\Delta})$. Since $\Pi|_{\mathcal{Z}_{\Delta}}$ is $G$-invariant, then by Lemma \ref{12.9}
	$$
		d_{\mathcal{Z}_{\Delta}}\circ\mathcal{A}_{G}\big(X\big)=\Big[\Pi|_{\mathcal{Z}_{\Delta}},\mathcal{A}_{G}(X)\Big]=\mathcal{A}_{G}\Big[\Pi|_{\mathcal{Z}_{\Delta}},X\Big]
=\mathcal{A}_{G}\circ d_{\mathcal{Z}_{\Delta}}\big(X\big)	$$

\end{proof}
\begin{corollary}\label{12.17}
	If $X\in \operatorname{Pois}(\mathfrak{X}^{k}(\mathcal{Z}_{\Delta}),\Pi|_{\mathcal{Z}_{\Delta}})$ $\big(\text{resp. }\operatorname{Ham}(\mathfrak{X}^{k}(\mathcal{Z}_{\Delta}),\Pi|_{\mathcal{Z}_{\Delta}})\big)$, then $\mathcal{A}_{G}(X)\in \operatorname{Pois}(\mathfrak{X}^{k}(\mathcal{Z}_{\Delta})^G,\Pi|_{\mathcal{Z}_{\Delta}})$ $\big(\text{resp.} \operatorname{Ham}(\mathfrak{X}^{k}(\mathcal{Z}_{\Delta})^G,\Pi|_{\mathcal{Z}_{\Delta}})\big)$
\end{corollary}

\begin{corollary}\label{12.19}
	If $X\in \operatorname{Ham}(\mathfrak{X}^{k}(\mathcal{Z}_{\Delta})^G,\Pi|_{\mathcal{Z}_{\Delta}})$, then there exists $Z\in \mathfrak{X}^{k-1}(\mathcal{Z}_{\Delta})^G$ such that,
	$$
		d_{\mathcal{Z}_{\Delta}}(Z)=X
	$$
\end{corollary}

\begin{lemma}\label{12.21}
The diagram below describes a commutative square.
	\begin{center}
\begin{tikzcd}
\mathfrak{X}^{k}(\mathcal{Z}_{\Delta})^G\arrow[r, "d_{\mathcal{Z}_{\Delta}}"] \arrow[d, "\mathcal{Q}_{*}"]
&\mathfrak{X}^{k+1}(\mathcal{Z}_{\Delta})^G  \arrow[d, "\mathcal{Q}_{*}" ] \\
\mathfrak{X}^{k}(\mathcal{O}) \arrow[r, "d_{\mathcal{O}}"]
& \mathfrak{X}^{k+1}(\mathcal{O})
\end{tikzcd}
\end{center}
\end{lemma}
\begin{proof}
	Let $X\in \mathfrak{X}^{k}(\mathcal{Z}_{\Delta})^G$. Then by Lemma \ref{12.6},
	$$
		\mathcal{Q}_{*}\circ d_{\mathcal{Z}_{\Delta}}(X)=\Big[\mathcal{Q}_{*}(\Pi|_{\mathcal{Z}_{\Delta}}),\mathcal{Q}_{*}(X)\Big]=\Big[\overline\Pi,\mathcal{Q}_{*}(X)\Big]=d_{\mathcal{O}}\circ \mathcal{Q}_{*}(X)
	$$
\end{proof}
\begin{corollary}\label{12.22}
	If $X\in \operatorname{Pois}(\mathfrak{X}^{k}(\mathcal{Z}_{\Delta})^G,\Pi|_{\mathcal{Z}_{\Delta}})$ $\big(\text{or }\operatorname{Ham}(\mathfrak{X}^{k}(\mathcal{Z}_{\Delta})^G,\Pi|_{\mathcal{Z}_{\Delta}})\big)$ then $\mathcal{Q}_{*}(X)\in \operatorname{Pois}(\mathfrak{X}^{k}(\mathcal{O}),\overline{\Pi})$ $\big(\text{resp. }\operatorname{Ham}(\mathfrak{X}^{k}(\mathcal{O}),\overline{\Pi})\big)$.
\end{corollary}

\begin{corollary}\label{12.24}
	The map,
	$$
		\mathcal{Q}_{*}:\operatorname{Ham}(\mathfrak{X}^{k}(\mathcal{Z}_{\Delta})^G,\Pi|_{\mathcal{Z}_{\Delta}})\rightarrow \operatorname{Ham}(\mathfrak{X}^{k}(\mathcal{O}),\overline{\Pi})
	$$
	is surjective.
\end{corollary}
\begin{proof}
	 Let $A\in \operatorname{Ham}(\mathfrak{X}^{k}(\mathcal{O}),\overline{\Pi})$ and $B\in \mathfrak{X}^{k-1}(\mathcal{O})$ be such that,
	$$
		[\overline{\Pi},B]=A
	$$
	There exists $Z\in \mathfrak{X}^{k-1}(\mathcal{Z}_{\Delta})^G$ such that,
	$$
		\mathcal{Q}_{*}(Z)=B
	$$
	Since
	$$
		\mathcal{Q}_{*}\Big(\Big[\Pi|_{\mathcal{Z}_{\Delta}},Z\Big]\Big)=\Big[\overline{\Pi},B\Big]=A \,,
	$$
	$X=\Big[\Pi|_{\mathcal{Z}_{\Delta}},Z\Big]\in \operatorname{Ham}(\mathfrak{X}^{k}(\mathcal{Z}_{\Delta})^G,\Pi|_{\mathcal{Z}_{\Delta}})$ is the desired multivector field such that $\mathcal{Q}_{*}(X)=A$.
	
\end{proof}
\begin{theorem}\label{12.25}
	The map
	$$
			\mathcal{Q}_{*}\circ \mathcal{A}_{G}:\mathfrak{X}^{k}(\mathcal{Z}_{\Delta})\rightarrow \mathfrak{X}^{k}(\mathcal{O})
	$$
	induces a map 
	$$
		\mathcal{H}_{*}:H^{k}_{\Pi|_{\mathcal{Z}_{\Delta}}}(\mathcal{Z}_{\Delta})\rightarrow H^{k}_{\overline{\Pi}}(\mathcal{O}) \,,
	$$
    where
	$$
		\mathcal{H}_{*}\Big([X]\Big)=\Big[\mathcal{Q}_{*}\circ \mathcal{A}_{G}(X)\Big]
	$$
\end{theorem}
\begin{proof}
	By Corollary \ref{12.17} and \ref{12.22}, if $X\in \operatorname{Pois}(\mathfrak{X}^{k}(\mathcal{Z}_{\Delta}),\Pi|_{\mathcal{Z}_{\Delta}})$ then
	
	$$
		\mathcal{Q}_{*}\circ \mathcal{A}_{G}(X)\in \operatorname{Pois}(\mathfrak{X}^{k}(\mathcal{O}),\overline{\Pi})
	$$
	Thus $\mathcal{Q}_{*}\circ \mathcal{A}_{G}$ induces a map,
	$$
		\mathcal{Q}_{*}\circ \mathcal{A}_{G}:\operatorname{Pois}(\mathfrak{X}^{k}(\mathcal{Z}_{\Delta}),\Pi|_{\mathcal{Z}_{\Delta}})\rightarrow \operatorname{Pois}(\mathfrak{X}^{k}(\mathcal{O}),\overline{\Pi})
	$$
	Similarly, 
	$$
				\mathcal{Q}_{*}\circ \mathcal{A}_{G}:\operatorname{Ham}(\mathfrak{X}^{k}(\mathcal{Z}_{\Delta}),\Pi|_{\mathcal{Z}_{\Delta}})\rightarrow \operatorname{Ham}(\mathfrak{X}^{k}(\mathcal{O}),\overline{\Pi})
	$$
	
As a result, $\mathcal{Q}_{*}\circ \mathcal{A}_{G}$ induces a map on $H^{k}_{\Pi|_{\mathcal{Z}_{\Delta}}}(\mathcal{Z}_{\Delta})$ to $H^{k}_{\overline{\Pi}}(\mathcal{O})$,
	$$
		\mathcal{Q}_{*}\circ \mathcal{A}_{G}: \frac{\operatorname{Pois}(\mathfrak{X}^{k}(\mathcal{Z}_{\Delta}),\Pi|_{\mathcal{Z}_{\Delta}})}{\operatorname{Ham}(\mathfrak{X}^{k}(\mathcal{Z}_{\Delta}),\Pi|_{\mathcal{Z}_{\Delta}})}\rightarrow \frac{\operatorname{Pois}(\mathfrak{X}^{k}(\mathcal{O}),\overline{\Pi})}{\operatorname{Ham}(\mathfrak{X}^{k}(\mathcal{O}),\overline{\Pi})} 
	$$
	$$
		\mathcal{Q}_{*}\circ \mathcal{A}_{G}\Big([X]\Big)=\Big[\mathcal{Q}_{*}\circ \mathcal{A}_{G}(X)\Big]
	$$
	which is the desired map $\mathcal{H}_{*}$.
	
\end{proof}

\subsection{Poisson cohomology classes}
\begin{lemma}\label{12.26}
	For each $1\leq i\leq m$, $\Theta_{i}|_{\mathcal{Z}_{\Delta}}$ is a $G$-invariant vector field on $\mathcal{Z}_{\Delta}$.
\end{lemma}
\begin{proof}
	As shown in the proof of Theorem \ref{8.1}, each $\Theta_{i}$ is $G$-invariant. The statement follows from that.
\end{proof}
\begin{lemma}\label{12.27}
	Each $(P\zeta)^I_{u}|_{\mathcal{Z}_{\Delta}}$ is a $G$-invariant vector field on $\mathcal{Z}_{\Delta}$.
\end{lemma}
\begin{proof}
	Using the proof of Theorem \ref{8.1}, one can show that each $(P\zeta)^I_{u}$ is $G$-invariant. The statement follows from that.
\end{proof}
\begin{lemma}\label{12.28}
	For $k=1$, the kernel of the map,
	$$
			\mathcal{Q}_{*}:\mathfrak{X}^{1}(\mathcal{Z}_{\Delta})^G\rightarrow \mathfrak{X}^{1}(\mathcal{O})
	$$
	consists of the $G$-invariant sections of $E^d_{D}$.
	$$
		\operatorname{Ker}\big(\mathcal{Q}_{*}:\mathfrak{X}^{1}(\mathcal{Z}_{\Delta})^G\rightarrow \mathfrak{X}^{1}(\mathcal{O})\big)=\big\{X\in \mathfrak{X}^{1}(\mathcal{Z}_{\Delta})^G|\text{ }X_{q}\in(E^d_{D})_{q}\quad \forall q\in \mathcal{Z}_{\Delta}\big\}
	$$
Denote $\operatorname{Ker}\big(\mathcal{Q}_{*}:\mathfrak{X}^{1}(\mathcal{Z}_{\Delta})^G\rightarrow \mathfrak{X}^{1}(\mathcal{O})\big)$ by $\Gamma(E^d_{D})^G$
\end{lemma}

\begin{proof}
	Since $\mathcal{Q}$ is the quotient map to the orbit space, at each $q\in \mathcal{Z}_{\Delta}$,
	
	$$
\operatorname{Ker}\big(\mathcal{Q}_{*}|_{q}:T_{q}(\mathcal{Z}_{\Delta})\rightarrow T_{[q]}(\mathcal{O})\big)=  (E^d_{D})_{q}
$$
The statement follows from that.
\end{proof}
\begin{remark} Any section of $E^d_{D}$ in $\mathcal{Z}_{\Delta}$ is of the form
$$
	\sum_{l=1}^mh_{l}\Theta_{l}|_{\mathcal{Z}_{\Delta}}
$$
where $h_{l}\in C^{\infty}(\mathcal{Z}_{\Delta})$.
\end{remark}
\begin{theorem}\label{12.29}
	Let $X$ and $Y$ be distinct $G$-invariant Poisson vector fields in $(\mathcal{Z}_{\Delta},\Pi|_{\mathcal{Z}_{\Delta}})$.  $\mathcal{Q}_{*}\circ \mathcal{A}_{G}(X)$ and $\mathcal{Q}_{*}\circ \mathcal{A}_{G}(Y)$ represent the same class in $H^1_{\overline{\Pi}}(\mathcal{O})$ if and only if there exist $Z\in \operatorname{Ham}(\mathfrak{X}^{1}(\mathcal{Z}_{\Delta})^G,\Pi|_{\mathcal{Z}_{\Delta}})$ and $A \in\Gamma(E^d_{D})^G$ such that
	$$
   X-Y=Z+A \,.
$$
\end{theorem}
\begin{proof}
($\implies$)
Assume $\mathcal{Q}_{*}\circ \mathcal{A}_{G}(X)$ and $\mathcal{Q}_{*}\circ \mathcal{A}_{G}(Y)$ represent the same Poisson class in $H^1_{\overline\Pi}(\mathcal O)$.
Since $X$ and $Y$ are $G$-invariant and $\mathcal{Q}_{*}$ is surjective on Hamiltonian vector fields, there exist $W\in \operatorname{Ham}(\mathfrak{X}^{1}(\mathcal{O}),\overline{\Pi})$ and $Z\in \operatorname{Ham}(\mathfrak{X}^{1}(\mathcal{Z}_{\Delta})^G,\Pi|_{\mathcal{Z}_{\Delta}})$ such that
$$
	\mathcal{Q}_{*}\circ\mathcal{A}_{G}(X-Y)=\mathcal{Q}_{*}(X-Y)=\mathcal{Q}_{*}(Z)
$$
$$
	\implies \mathcal{Q}_{*}(X-Y-Z)=0\,.
$$
Thus there exists $A\in \Gamma(E^d_{D})^G$ such that
$$
	X-Y=Z+A\,.
$$
($\impliedby$)
Assume there exist $Z\in \operatorname{Ham}(\mathfrak{X}^{1}(\mathcal{Z}_{\Delta})^G,\Pi|_{\mathcal{Z}_{\Delta}})$ and $A \in\Gamma(E^d_{D})^G$ such that
	$$
   X-Y=Z+A
$$
$$
	\implies \mathcal{Q}_{*}(X-Y)=\mathcal{Q}_{*}(Z)
$$
$$
	\implies \mathcal{Q}_{*}\circ\mathcal{A}_{G}(X-Y)=W
$$
where $\mathcal{Q}_{*}(Z)=W\in \operatorname{Ham}(\mathfrak{X}^{1}(\mathcal{O}),\overline{\Pi})$.
$$
	\implies \mathcal{H}_{*}\big(X\big)= \mathcal{H}_{*}\big(Y\big)\,.
$$
\end{proof}
\begin{theorem}\label{12.30}
	Let $X$ and $Y$ be distinct $G$-invariant Poisson vector fields in $(\mathcal{Z}_{\Delta},\Pi_{1}|_{\mathcal{Z}_{\Delta}})$ such that $X-Y$ is a radial Poisson vector field of the form
	$$
   \sum_{(I,u)\in \mathscr{I}^*}  f^I_{u}(P\zeta)^I_{u}|_{\mathcal{Z}_{\Delta}}+\sum_{w\in B}h_{w}\Theta_{w} |_{\mathcal{Z}_{\Delta}}
$$
	where $f^I_{u}\in C^\infty(\mathcal{Z}_{\Delta})^G$, $h_{w}\in \operatorname{Cas}(\mathcal{Z}_{\Delta},\Pi_{1}|_{\mathcal{Z}_{\Delta}})^G$, $\sum_{w\in B}h_{w}\Theta_{w} |_{\mathcal{Z}_{\Delta}}\not\in \Gamma(E^d_{D})^G$ and for each $w\in B$, $a_{wj}|_{\mathcal{Z}_{\Delta}}=0$ for all $1\leq j\leq m$. Then $\mathcal{Q}_{*}\circ\mathcal{A}_{G}(X)$ and $\mathcal{Q}_{*}\circ\mathcal{A}_{G}(Y)$ represent distinct classes in $H^1_{\overline{\Pi_{1}}}(\mathcal{O})$.
\end{theorem}
\begin{proof}
Any vector field $Z+A$ where $Z\in \operatorname{Ham}(\mathfrak{X}^{1}(\mathcal{Z}_{\Delta})^G,{\Pi_{1}}|_{\mathcal{Z}_{\Delta}})$ and $A \in\Gamma(E^d_{D})^G$  is of the form
	$$
	  		\sum_{\substack{j=1\\j\not\in B}}^mh_{j}\Theta_{j}|_{\mathcal{Z}_{\Delta}}+\sum_{j=1}^me_{j}\Theta_{j}|_{\mathcal{Z}_{\Delta}}
	$$
	for some $e_{j},h_{j}\in C^{\infty}(\mathcal{Z}_{\Delta})^G$. Since $X$ and $Y$ are distinct, by the hypothesis there do not exist any $f^I_{u},h_{w}\in C^{\infty}(\mathcal{Z}_{\Delta})^G$ such that
	$$
		X-Y\not= 0  \quad\text{and}\quad X-Y=Z+A
	$$
	Thus, by Theorem \ref{12.29}, the statement holds.
\end{proof}
\begin{corollary}\label{12.31}
	Each Poisson vector field in $(\mathcal{O},\overline{\Pi_{1}})$ of the form
	$$
   \mathcal{Q}_{*}\left(\sum_{(I,u)\in \mathscr{I^*}} f^I_{u}(P\zeta)^I_{u}|_{\mathcal{Z}_{\Delta}}\right)
$$
	where $f^I_{u}\in C^\infty(\mathcal{Z}_{\Delta})^G$, represents a distinct class in $H^1_{\overline{\Pi_{1}}}(\mathcal{O})$.
\end{corollary}
Using the same technique as in Theorem 
\ref{12.30}, similar results hold for ${}_{K}\overline{\Pi^R_{3}}$ and ${}_{K}\overline{\Pi^R_{4}}$.
\begin{theorem}\label{12.32}
	Let $X$ and $Y$ be distinct $G$-invariant Poisson vector fields in $(\mathcal{Z}_{\Delta},{}_{K}\Pi^R_{3}|_{\mathcal{Z}_{\Delta}})$ such that $X-Y$ is a Poisson vector field of the form
	$$
   \sum_{(I,u)\in (\mathscr{I}\setminus\mathscr{K})^*} f^I_{u}(P\zeta)^I_{u}|_{\mathcal{Z}_{\Delta}}+	  	\sum_{j=1}^mh_{j}\Theta_{j}|_{\mathcal{Z}_{\Delta}}
	$$
    
where $f^I_{u}\in C^\infty(\mathcal{Z}_{\Delta})^G$, $h_{j}\in \operatorname{Cas}(\mathcal{Z}_{\Delta},{}_{K}\Pi_{3}^R|_{\mathcal{Z}_{\Delta}})^G$ and $\sum_{j=1}^mh_{j}\Theta_{j}|_{\mathcal{Z}_{\Delta}}\not\in \Gamma(E^d_{D})^G$. Then $\mathcal{Q}_{*}\circ\mathcal{A}_{G}(X)$ and $\mathcal{Q}_{*}\circ\mathcal{A}_{G}(Y)$ represent distinct classes in $H^1_{{}_{K}\overline{\Pi^R_{3}}}(\mathcal{O})$.
\end{theorem}



\begin{corollary}\label{12.33}
Each Poisson vector field in $(\mathcal{O},{}_{K}\overline{\Pi^R_{3}})$ of the form
	$$
   \mathcal{Q}_{*}\left(\sum_{(I,u)\in (\mathscr{I}\setminus\mathscr{K})^*} f^I_{u}(P\zeta)^I_{u}|_{\mathcal{Z}_{\Delta}}\right)
$$
	where $f^I_{u}\in C^\infty(\mathcal{Z}_{\Delta})^G$, represents a distinct class in $H^1_{{}_{K}\overline{\Pi^R_{3}}}(\mathcal{O})$.
\end{corollary}

\begin{theorem}\label{12.34}
	Let $X$ and $Y$ be distinct $G$-invariant Poisson vector fields in $(\mathcal{Z}_{\Delta},{}_{K}\Pi^R_{4}|_{\mathcal{Z}_{\Delta}})$ such that $X-Y$ is a Poisson vector field of the form
	$$
   \sum_{(I,u)\in (\mathscr{I}\setminus\mathscr{K})^*} f^I_{u}(P\zeta)^I_{u}|_{\mathcal{Z}_{\Delta}}+\sum_{w\in B}h_{w}\Theta_{w} |_{\mathcal{Z}_{\Delta}}
	$$
	
where $f^I_{u}\in C^\infty(\mathcal{Z}_{\Delta})^G$,  $h_{w}\in \operatorname{Cas}(\mathcal{Z}_{\Delta},{}_{K}\Pi_{4}^R|_{\mathcal{Z}_{\Delta}})^G$, $\sum_{w\in B}h_{w}\Theta_{w} |_{\mathcal{Z}_{\Delta}}\not\in \Gamma(E^d_{D})^G$ and for each $w\in B$, $\tau_{wj}|_{\mathcal{Z}_{\Delta}}=0$ for all $1\leq j\leq m$. Then $\mathcal{Q}_{*}\circ\mathcal{A}_{G}(X)$ and $\mathcal{Q}_{*}\circ\mathcal{A}_{G}(Y)$ represent distinct classes in $H^1_{{}_{K}\overline{\Pi^R_{4}}}(\mathcal{O})$.
\end{theorem}


\begin{corollary}\label{12.35}
	Each Poisson vector field in $(\mathcal{O},{}_{K}\overline{\Pi^R_{4}})$ of the form
	$$
   \mathcal{Q}_{*}\left( \sum_{(I,u)\in (\mathscr{I}\setminus\mathscr{K})^*} f^I_{u}(P\zeta)^I_{u}|_{\mathcal{Z}_{\Delta}}\right)
$$
	where $f^I_{u}\in C^\infty(\mathcal{Z}_{\Delta})^G$, represents a distinct class in $H^1_{{}_{K}\overline{\Pi^R_{4}}}(\mathcal{O})$.
\end{corollary}

\section*{Ackowledgements} We thank Dheeraj Kulkarni for useful discussions. We thank IISER Pune where most of the work was carried out as part of the master's thesis of the first author under the supervision of the second author. The first author also thanks Vignesh M Pai and Eshwar.

\bigskip

\noindent
Vansh Aggarwal,\\
Department of Mathematics, IISER-Pune, \\
Pune, India\\ 
e-mail: {\tt v.aggarwal.research@gmail.com}

\bigskip
\noindent
Mainak Poddar,\\
Department of Mathematics, IISER-Pune, \\
Pune, India\\ 
e-mail: {\tt mainak@iiserpune.ac.in}

\end{document}